\documentclass[a4paper,11pt, reqno]{amsart}
\usepackage[margin=2cm]{geometry}

\usepackage{natbib}

\usepackage{amssymb,amsmath,amsthm}
\usepackage{color}
\usepackage[table]{xcolor}
\usepackage{adjustbox}
\usepackage{graphicx}
\usepackage{subcaption}

\usepackage{float}
\usepackage{hyperref}
\usepackage{multirow}
\usepackage{multicol}

\newtheorem{thm}{Theorem}

\newtheorem{ex}{Example}

\def\N{\mathcal{N}}

\def\Per{\mathcal{P}}

\newcommand{\dsum}{\displaystyle\sum}

\newcommand{\dmax}{\displaystyle\max}

\def\K{\mathcal{K}}

\def\R{\mathbb{R}}

\def\Z{\mathbb{Z}}

\def\H{\mathcal{H}}

\usepackage{tikz}

\usepackage{pgfplots}
\pgfplotsset{compat=newest}
\usetikzlibrary{decorations.markings}

\let\origmaketitle\maketitle
\def\maketitle{
	\begingroup
	\def\uppercasenonmath##1{} 
	\let\MakeUppercase\relax 
	\origmaketitle
	\endgroup
}

\begin{document}

\title[]{\Large Exact Matrix Seriation through Mathematical Optimization:\\Stress and Effectiveness-Based Models}

\author[V. Blanco, A. Mar\'in, \MakeLowercase{and} J. Puerto]{
{\large V\'ictor Blanco$^{\dagger}$, Alfredo Mar\'in$^{\ddagger}$, and Justo Puerto$^{\star}$}\medskip\\
$^\dagger$Institute of Mathematics (IMAG), Universidad de Granada\\
$^\ddagger$Dpt. Stats \& OR, Universidad de Murcia\\
$^\star$Institute of Mathematics (IMUS), Universidad de Sevilla\\
\texttt{vblanco@ugr.es}, \texttt{amarin@um.es}, \texttt{puerto@us.es}
}

\maketitle

\begin{abstract}
Matrix seriation, the problem of permuting the rows and columns of a matrix to uncover latent structure, is a fundamental technique in data science, particularly in the visualization and analysis of relational data. Applications span clustering, anomaly detection, and beyond. In this work, we present a unified framework grounded in mathematical optimization to address matrix seriation from a rigorous, model-based perspective. Our approach leverages combinatorial and mixed-integer optimization to represent seriation objectives and constraints with high fidelity, bridging the gap between traditional heuristic methods and exact solution techniques.

We introduce new mathematical programming models for neighborhood-based stress criteria, including nonlinear formulations and their linearized counterparts. For structured settings such as Moore and von Neumann neighborhoods, we develop a novel Hamiltonian path-based reformulation that enables effective control over spatial arrangement and interpretability in the reordered matrix.

To assess the practical impact of our models, we carry out an extensive set of experiments on synthetic and real-world datasets, as well as on a newly curated benchmark based on a coauthorship network from the matrix seriation literature. Our results show that these optimization-based formulations not only enhance solution quality and interpretability but also provide a versatile foundation for extending matrix seriation to new domains in data science.
\end{abstract}
 

 \section*{Introduction}

 The primary goal of matrix seriation is to rearrange the rows and/or columns of a given matrix in order to reveal its possible hidden information. This technique is particularly useful to identify similar patterns between individuals in a dataset by means of its similarity, where adequately reordering the observations may result in detecting the groups of more similar observations,  visualizing heatmaps, or selecting the most correlated variables in a dataset~\citep{bertin2011graphics}. Its importance in data visualization has already been proven, with a public library available in R for this task, \href{https://cran.r-project.org/web/packages/seriation/index.html}{seriation}~\citep{hahsler2008getting}.

Matrix seriation was first introduced by \cite{petrie1899sequences} for archaeological purposes to order chronological assemblages of artifacts, which is particularly relevant when absolute dating techniques are unavailable, and these techniques allow one to construct the temporal sequence of cultural developments by analyzing patterns in artifact types and frequencies across different excavation sites. Nevertheless, this is not the only field where seriation methods are useful. For example, in biology/ecology, paleoecologists have applied matrix seriation to analyze the presences and absences of taxa in samples~\citep{brower1988seriation}, in anthropology~\citep{kuzara1966seriation}, psychology~\citep{brusco2001compact}, sociology~\citep{forsyth1946matrix} and chemistry~\citep{toth2018seriation}, among many others. Since its introduction, numerous methodological studies have been published to either propose metrics to measure the convenience of a rearranged matrix or propose solution methods for general or structured matrices \citep[see e.g.][]{borst2024connectivity,gendreau1994generalized,laporte1976comparison,laporte1978seriation,laporte1987solving}. More details on the history of seriation methods and applications can be found in the review papers~\citep{liiv2010seriation,hahsler2008getting}.

A matrix seriation problem has two key ingredients: the selection of the criterion to assess the goodness of a matrix and the methodology to reorder the rows and/or columns of the matrix under this criterion. Different measures have been proposed to this end. \cite{chen2002generalized} and \cite{hubert2001combinatorial} propose the so-called \emph{gradient} and \emph{divergence} criteria, that are based on comparing the resulting matrix with an \emph{ideal} ordered dissimilarity matrix (the Anti-Robinson matrix) where the values in all rows and columns only increase when moving away from the main diagonal. One can measure this closeness between a matrix and the Anti-Robinson matrix applying different seriation criteria: count the number of violations of the Anti-Robinson matrix conditions~\citep{chen2002generalized}, the difference between these violations and the agreements~\citep{hubert2001combinatorial}, or generalized versions that count only these (weighted) violations in bands of certain length~\citep{tien2008methods}. \cite{brusco2002integer} provides integer linear optimization models for these methods based on a previous formulation by \cite{decani1972branch}. Other family of measures are based on finding a trade-off between the values in the resulting matrix and their rank difference in the order, that is, on penalizing the existence of very different values close in the reordered matrix. The most popular criteria within this family are the \emph{least squares} measure and the \emph{inertia}, proposed by \cite{caraux2005permutmatrix}, which consider the squared difference between the transformed entry and the
absolute rank differences and the product of the entries by its squared rank difference, respectively. Other criteria of this family are the 2-sum~\citep{barnard1993spectral}, which considers the product of the inverse values by the squared rank differences, and the linear seriation criterion~\cite{hubert1976quadratic}, which considers the product of the values by their inverse absolute rank differences. The third family of criteria that have been used to measure the goodness of reordering a matrix are based on identifying the matrix with a (weighted) bipartite graph, where the nodes are rows and columns of the matrix, and the edges linking rows and columns are weighted by the corresponding entry in the matrix. \cite{hubert1976quadratic} and \cite{caraux2005permutmatrix} already noticed that sorting the rows/columns of the matrix is equivalent to finding a Hamiltonian path in the set of rows/columns with the adequate weights. Finally, \cite{niermann2005optimizing} proposed the \emph{stress} criterion to quantify the similarity of an entry by comparing it with its neighbor entries. This measure consists of computing the (squares or absolute) differences between each entry and the values of the neighbor entries, where different neighborhood shapes can be defined. They are defined for a given matrix, where one can easily compute its \emph{seriation goodness} value under any of the above criteria. Although most of them are defined for square matrices where rows and columns are sorted using the same permutation, some of the criteria can be adapted to the case in which columns and rows are separately reordered. In this paper, we first provide a general stress criterion that considers generalized neighborhoods and $\ell_p$-norm aggregations of the dissimilarities between an entry and their neighbors.

Once the seriation criterion is decided, the computation of the \emph{best} permutation of rows and/or columns of the matrix under such a measure is computationally challenging. As already stated in \cite{hahsler2008getting,hahsler2017experimental}, all proposed methodologies are either enumerative~\citep{hubert2001combinatorial,brusco2002integer} or heuristics, as the Bond Energy Algorithm (BEA)~\cite{mccormick1972problem}, hierarchical clustering, the rank-two ellipse seriation~\citep{chen2002generalized}, the spectral seriation~\citep{barnard1993spectral}, or the application of any of the vast amount of heuristic algorithms for the Traveling Salesman Problem (TSP) and the Quadratic Assignment Problem (QAP). Most of these measures and approximate approaches have been implemented and made publicly available through the \texttt{R} package \texttt{seriation}~\cite{hahsler2008getting}. Nevertheless, in most of the cases the methodologies are not aligned with the criteria, that is, the existing methods try to optimize a measure but the assessment of the transformed matrix is performed with a different metric, and also there is a clear lack of exact approaches for these matrix seriation problems, and especially for those based on stress measures. 

In this article we propose a mathematical optimization framework for a family of matrix seriation problems that, as far as we know, have only been partially studied: Matrix seriation problems under stress or effectiveness measures. Although the available software, as \texttt{seriation} in \texttt{R}, allows for the computation of some of the loss measures minimized in these methods, there is not a dedicated method available to obtain the optimal sorting with these criteria. We study the general problem of deciding how to rearrange the rows and columns of a given matrix $A\in \R^{n\times m}$ separately or, when $n=m$, jointly, by minimizing a stress measure or maximizing the measure of effectiveness. We derive different mathematical optimization formulations for the problem. First we propose a mixed integer nonlinear formulation that can be applied to any type of neighborhoods, then we derive two alternative linearizations of the model and strengthen this formulation using some symmetry breaking constraints. For particular neighborhood shapes, we develop another mixed integer linear model based on constructing two Hamiltonian paths, and using a compact flow-based formulation to avoid subtours. For more general neighborhood shapes which induce non separable stress measures, we adapt the Hamiltonian paths formulation at the price of including products of binary variables, which must be linearized. For maximizing the measure of effectiveness, we provide a different Hamiltonian path approach to cast this measure into the same combinatorial optimization problem. 
 
An important characteristic of our approaches is that they are exact, meaning they provide provably optimal solutions to the seriation problems. This exactness translates into significant advantages when evaluating structural measures, such as Moore or Von Neumann stress, or Meassure of Effectivenes, as the solutions truly optimize the corresponding objective functions. Consequently, the visualizations derived from our reordered matrices are more informative and interpretable, especially for datasets where subtle relational patterns are easily masked by heuristic approximations. Moreover, our formulations naturally allow for the incorporation of general ad-hoc requirements on the structure of the reordered matrices through additional constraints. For instance, specific rows or columns can be forced to remain adjacent or separated, or bounds can be imposed on their pairwise distances in the final layout. This flexibility, together with the exactness of our methods, offers a powerful and customizable tool for high-quality data visualization and analysis.

The use of integer optimization in Data Science is not new, and it has gained increasing attention in recent years due to its ability to model complex combinatorial structures with high precision and flexibility. For instance, recent Integer Optimization models have been proposed for selecting relevant features in clustering~\cite{benati2018mixed} and supervised classification methods~\cite{maldonado2011simultaneous,maldonado2014feature}, or in Community Detection in graphs~\cite{li2016quantitative,ponce2024mixed}. In data visualization, combinatorial optimization methods have been employed to draw graphs with fairness criteria~\cite{eades2025introducing}, or for visualizing dynamic multidimensional data~\cite{carrizosa2019visualization}. These contributions highlight the growing relevance of mathematical optimization as a principled framework for both the modeling and interpretation of data-driven tasks. Our work builds on this tradition by showing that exact seriation via integer optimization not only improves structural measures but also enhances the visual and analytical value of the data representations.

The rest of the paper is organized as follows. In Section \ref{sec:1} we introduce the notation and the matrix seriation problems that we analyze in this work. Section \ref{sec:2} presents the mathematical optimization models that we propose for the different seriation problems. First, we study general neighborhood seriation problems with a flexible but high dimensional mixed integer non linear optimization problem. Then, we propose two different types of linearizations for the problem. Finally, in that section we also develop novel Hamiltonian path-based formulation for the different seriation problems. Section \ref{sec:3} is devoted to report the results of a series of computational experiments. We first analyze the performance of our approaches on synthetic data sets, and then we analyze real-data sets available in the package \texttt{seriation} to compare our results with those obtained with other existing approaches. Finally, we provide a new benchmark dataset on the scientific coauthorship matrix on the topic \textit{matrix seriation}, and provide different insigths on the obtained results.

\section{Stress and Effectiveness Seriation Problems} \label{sec:1}

In this section we introduce the matrix seriation problems under analysis and set the notation for the rest of the sections.

We are given a real matrix $A \in \R^{n\times m}$. We denote by $N=\{1, \ldots, n\}$ and $M=\{1, \ldots, m\}$ the index sets for the rows and columns of $A$. Rearranging the rows and columns of $A$ is equivalent to determine two permutations $\sigma_r \in \Per_n, \sigma_c \in \Per_m$ (here $\Per_n$ and $\Per_m$ stand for the permutation sets of $N$ and $M$, respectively). After reordering $A$ with these permutations we obtain a transformed matrix that we denote by $A_{\sigma_r,\sigma_c}$. Note that this reordered matrix is the result of the matrix multiplication $A_{\sigma_r,\sigma_c} = P\cdot A \cdot Q$, where $P \in \{0,1\}^{n\times n}$ and $Q \in \{0,1\}^{m\times m}$ are the permutation matrices derived from $\sigma_r$ and $\sigma_c$, respectively.

The determination of the most appropriate permutations $\sigma_r$ and $\sigma_c$ depends on the measure used. One of the most common families of measures is the so-called stress measure family. In it, the goodness of the transformation is affected, for each entry of the matrix, by the closeness to its \emph{neighbor} entries. For each pair $(i,j) \in N\times M$, we denote by $\mathcal{N}_{ij}(A)$ the neighborhood of entry $a_{ij}$ in $A$. A parameterized family of neighborhoods that is useful in seriation problems is in the form:
$$
\mathcal{N}^\varepsilon_{ij}=\{(k,\ell ) \in N\times M: \|(k,\ell )-(i,j)\|_2\leq \varepsilon\}
$$
for some $\varepsilon \geq 0$.
In Figure \ref{fig:neigh} we show the neighborhoods of some entries of a matrix (represented as cells) for the cases $\varepsilon=1$ (the so-called von Neumann neighborhood), $\varepsilon=1.5$ (the Moore neighborhood) and $\varepsilon=2$.
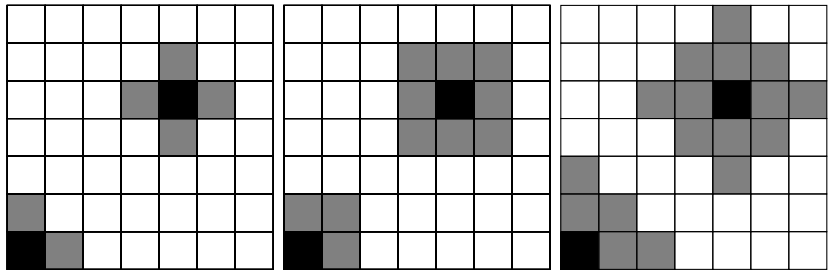
\begin{figure}[h]
\begin{center}
    \begin{tikzpicture}[scale=0.5]
        \def\n{7} 

        \def\hx{5} 
        \def\hy{5} 
        
        \fill[black] (\hx,\hy) rectangle (\hx+1,\hy+1);

        \foreach \dx/\dy in {0/1, 1/0, 0/-1, -1/0} {
            \draw[gray, fill=gray] (\hx+\dx,\hy+\dy) rectangle (\hx+\dx+1,\hy+\dy+1);
        }
        \def\hx{1} 
        \def\hy{1} 
        
        \fill[black] (\hx,\hy) rectangle (\hx+1,\hy+1);

        \foreach \dx/\dy in {0/1, 1/0} {
            \fill[gray] (\hx+\dx,\hy+\dy) rectangle (\hx+\dx+1,\hy+\dy+1);
        }

        \foreach \x in {1,...,\n} {
            \foreach \y in {1,...,\n} {
                \node at (\x+0.5,\y+0.5) {};
            }
        \foreach \x in {1,...,\n} {
            \foreach \y in {1,...,\n} {
                \draw[black] (\x,\y) rectangle (\x+1,\y+1);
            }
        }
        }
    \end{tikzpicture}~\begin{tikzpicture}[scale=0.5]
        \def\n{7} 

        \def\hx{5} 
        \def\hy{5} 
        
        \fill[black] (\hx,\hy) rectangle (\hx+1,\hy+1);

        \foreach \dx/\dy in {0/1, 1/0, 0/-1, -1/0, 1/-1, 1/1, -1/1, -1/-1} {
            \fill[gray] (\hx+\dx,\hy+\dy) rectangle (\hx+\dx+1,\hy+\dy+1);
        }

        \def\hx{1} 
        \def\hy{1} 
        
        \fill[black] (\hx,\hy) rectangle (\hx+1,\hy+1);

        \foreach \dx/\dy in {0/1, 1/0, 1/1} {
            \fill[gray] (\hx+\dx,\hy+\dy) rectangle (\hx+\dx+1,\hy+\dy+1);
        }

        \foreach \x in {1,...,\n} {
            \foreach \y in {1,...,\n} {
                \node at (\x+0.5,\y+0.5) {};
            }
        \foreach \x in {1,...,\n} {
            \foreach \y in {1,...,\n} {
                \draw[black] (\x,\y) rectangle (\x+1,\y+1);
            }
        }
        }
    \end{tikzpicture}~\begin{tikzpicture}[scale=0.5]
        \def\n{7} 

        \def\hx{5} 
        \def\hy{5} 
        
        \fill[black] (\hx,\hy) rectangle (\hx+1,\hy+1);

        \foreach \dx/\dy in {0/1, 1/0, 0/-1, -1/0, 1/-1, 1/1, -1/1, -1/-1, 0/2, 2/0, -2/0, 0/-2}  {
            \fill[gray] (\hx+\dx,\hy+\dy) rectangle (\hx+\dx+1,\hy+\dy+1);
        }

        \def\hx{1} 
        \def\hy{1} 
        
        \fill[black] (\hx,\hy) rectangle (\hx+1,\hy+1);

        \foreach \dx/\dy in {0/1, 1/0, 1/1, 2/0, 0/2} {
            \fill[gray] (\hx+\dx,\hy+\dy) rectangle (\hx+\dx+1,\hy+\dy+1);
        }

        \foreach \x in {1,...,\n} {
            \foreach \y in {1,...,\n} {
                \node at (\x+0.5,\y+0.5) {};
            }
        }
        \foreach \x in {1,...,\n} {
            \foreach \y in {1,...,\n} {
                \draw[black] (\x,\y) rectangle (\x+1,\y+1);
            }
        }
    \end{tikzpicture}
\end{center}
\caption{Different neighborhoods in the family $\mathcal{N}^\varepsilon$ (from left to right: $\varepsilon=1$, $\varepsilon=1.5$ and $\varepsilon=2$). The entries $(i,j)$ are highlighted in black whereas the neighbors are colored in gray\label{fig:neigh}}
\end{figure}

The goal of a stress seriation problem is to find the permutations $\sigma_r$, $\sigma_c$ minimizing the differences between the entries and their neighbors in the transformed matrix. These differences are usually computed as the $\ell_1$ or $\ell_2$ norms, i.e., for an entry $(i,j)$ in the original matrix, and permutations $\sigma_r$, $\sigma_c$, its stress measure in the transformed matrix is:
\begin{equation}\label{eq:stress:lp}
S^{\ell_p}_{ij}(\sigma_r,\sigma_c) := 
\sum_{\stackrel{(k,\ell ):}{(\sigma_r(k),\sigma_c(\ell )) \in 
\mathcal{N}_{\sigma_r(i)\sigma_r(j)}}} 
|a_{\sigma_r(i)\sigma_c(j)} - a_{\sigma_r(k)\sigma_c(\ell )}|^p.
\end{equation}
That is to say, the transformed entry $a_{\sigma_r(i)\sigma_c(j)}$ is compared against the neighbor entries in the transformed matrix.

Note that each entry itself is assumed to be in its neighborhood, its contribution to the measure is null, and one could remove it from its own neighborhood.

With the above measure, the stress matrix seriation problem consists of finding the permutations minimizing the overall sum of the stress measures for each entry:
\begin{equation}\label{eq:stressmin}
\min_{ \sigma_r \in \mathcal{P}_n\atop\sigma_c \in \mathcal{P}_m} 
\sum_{i\in N} \sum_{j \in M} S^{\ell_p}_{ij}(\sigma_r,\sigma_c).
\end{equation}

In some situations, where the rows and columns of a square matrix $A \in \R^{n\times n}$ represent the same set of individuals, the permutations $\sigma_r$ and $\sigma_c$ must coincide (as in similarity/dissimilarity matrices). In that case, the problem above simplifies to finding a single permutation $\sigma$ in the problem above.

This problem was proposed by  \cite{niermann2005optimizing} for the $\ell_2$ stress measure and both the von Neumann and the Moore neighborhoods, together with an evolutionary algorithm to compute a reasonable solution for the problem. The available software, called \texttt{seriation} in \texttt{R}~\citep{hahsler2008getting}, although allows for the computation of the loss function induced by these measures (for the three types of neighborhoods described above), does not allow for the computation of the optimal sorting with any of these criteria.

In order to measure the conciseness of a matrix, one can use the following modified homogeneity operator based on the stress measure:
$$
\H_p(A) = \frac{1}{nm} \sum_{k\in N}\sum_{\ell  \in M}  
\frac{1}{|\mathcal{N}_{k\ell }|}\sum_{(k',\ell ')\in \mathcal{N}_{k\ell }} 
\Big| a_{k\ell } - a_{k'\ell '}\Big|^p.
$$
This operation is designed to evaluate how well structured a matrix seriation is by considering both the differences between neighboring values and the size of the local neighborhood for each matrix entry. The larger the value of $\H_p(A)$, the more homogeneous is the matrix. For normalized matrices (matrix entries between $0$ and $1$), the above measure is an index in $[0,1]$, being $0$ achieved when all the elements in the matrix are equal. The upper bound of $1$ is reached just for particular cases. For the von Neumann neighborhood and binary matrices, $\H_p(A)=1$ only holds when the matrix alternates in columns and rows the $0$s and $1$s (the chessboard matrix). 

Yet another related seriation problem is to maximize the Measure of Effectiveness (ME) as defined by \cite{mccormick1972problem}. Its definition is
\begin{equation}
    {\rm ME}(A)=\frac{1}{2} \sum_{k \in N} \sum_{\ell \in M} a_{k\ell}(a_{k,\ell+1}+
    a_{k,\ell-1}) + \sum_{k \in N} \sum_{\ell \in M} a_{k\ell}(a_{k+1,\ell}+a_{k-1,\ell}) \label{eq:Mcornick}.
\end{equation}
ME is maximized if each element is as closely related numerically to its four neighbors as possible and gets maximal only if all large values are grouped together around the main diagonal. We will show that ME can be efficiently computed using similar approaches to those we introduce for the stress seriation.

\begin{ex}
    In Figure \ref{fig:matrices} we plot a heatmap of a $20\times 20$ matrix with entries in $[0,1]$ with $1$s coloured in black in the plots. In the left plot we show the original matrix, while in the right plot the rows and columns are sorted, in that case optimizing the Measure of Effectiveness \eqref{eq:Mcornick}. In that case, both the rows and the columns of the matrix are separately sorted (a different order is allowed for columns and rows), revealing a four cluster structure in the data, that would be difficult to identify in the original matrix.

    \begin{figure}
    \centering\includegraphics[width=0.9\textwidth]{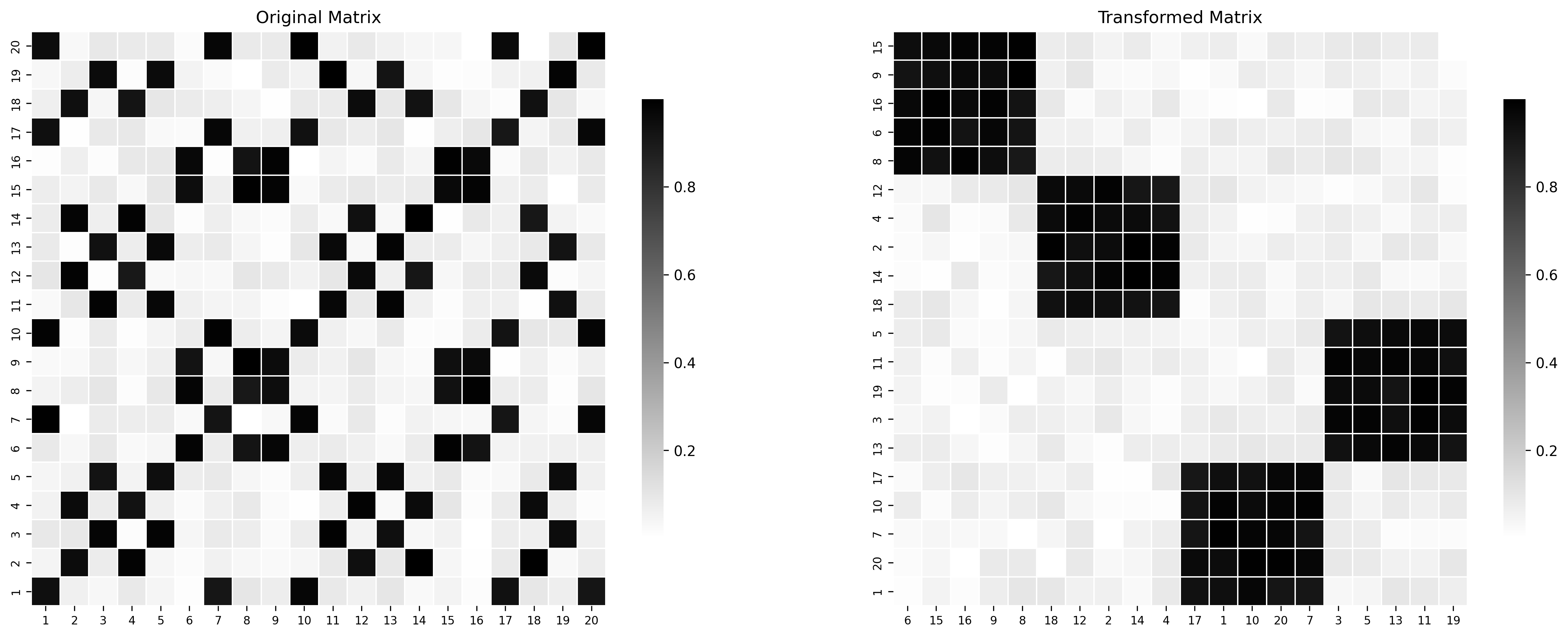}
    \caption{Graphical visualization of the impact of applying a matrix seriation method. In the left plot we show the heatmap of a $[0,1]$-matrix, while in the right plot the sorted heatmap is shown, detecting the hidden clusters in the original matrix\label{fig:matrices}}
    \end{figure}
\end{ex}

\section{Position Assignment Models for Stress and Effectiveness Seriation}  \label{sec:2}

In this section we develop a unified mathematical optimization model to compute optimal re-arrangements of the rows and columns of a given matrix by minimizing  stress criteria or maximizing the measure of effectiveness. The formulation that we present belongs to the family of mixed integer non-linear programming problems. Next, we develop two equivalent models that allow to linearize the nonconvex terms that appear in the model. We call these models the \emph{Position Assignment Models} (PAM) since they are based on adequately modeling through binary variables the assignment of each of the rows and columns in the matrix to a final position in the sorting.

We are given a matrix $A =(a_{ij}) \in \R^{n\times m}$. First we assume that a stress seriation measure has to be optimized. Then, we are given a general neighbor, $\N_{ij}$, for each entry of the matrix $(i,j) \in N \times M$. As mentioned above, the final configuration (re-arrangement) of the matrix will be constructed by comparing each of the entries with all its neighbors.  

We consider the following decision variables that allow determine the position of the rows and columns of the original matrix in the re-arranged matrix. 
$$
x_{ik} = \begin{cases}
    1 & \mbox{if row $i$ is sorted in position $k$}\\
    0 & \mbox{otherwise}
\end{cases} \; y_{j\ell} = \begin{cases}
    1 & \mbox{if column $j$ is sorted in position $\ell$}\\
    0 & \mbox{otherwise}
\end{cases},
$$
for all $i,k\in N$, $j,\ell \in M$.

Thus the induced row and columns permutations are
$$
\sigma_{r}(i) = k \text{ if $x_{ik}=1$}, \quad 
\sigma_{c}(j) = \ell  \text{ if $y_{j\ell }=1$}.
$$

With these variables, the problem can be formulated as follows:
\begin{align} \hbox{(PAM) }
    \min & \sum_{k\in N}\sum_{\ell \in M} \sum_{(k',\ell ') \in \mathcal{N}_{k\ell}} 
    \Big|\sum_{i\in N}\sum_{j\in M} a_{ij}x_{ik} y_{j\ell} -
         \sum_{i\in N}\sum_{j\in M} a_{ij}x_{ik'}y_{j\ell '} \Big|^p \label{f1:obj}\\
    \mbox{s.t. } 
    & \sum_{i \in N} x_{ik} = 1\ \ \ \ \forall k \in N\label{f1:ctr1}\\
    & \sum_{k \in N} x_{ik} = 1 \ \ \ \ \forall i \in N\label{f1:ctr2}\\
    &\sum_{j \in M} y_{j\ell} = 1 \ \ \ \ \forall l \in M\label{f1:ctr3}\\
    & \sum_{\ell\in N} y_{i\ell} = 1\ \  \ \ \forall j \in M\label{f1:ctr4}\\
    & x_{ik}, y_{j\ell} \in \{0,1\}\ \ \ \ \forall i,k \in N\ j,\ell \in M.\label{f1:ctr5}
\end{align}
In the above formulation, the objective function \eqref{f1:obj} represents the overall sum of the stress measures on all the entries of the transformed matrix. Constraints \eqref{f1:ctr1} guarantee that all row positions are assigned to a row and constraints \eqref{f1:ctr2} that all rows are sorted in a position. Constraints \eqref{f1:ctr3} and \eqref{f1:ctr4} are the same conditions for columns. 

The one above is a nonlinear binary formulation that can be reformulated as a mixed integer second order cone optimization problem by linearizing the products of the binary variables in the objective function and representing the $\ell_p$ norm as second order cone constraints. 

Observe that the norm-based constraints can be rewritten as follows:
\begin{align}
&v_{k\ell k'\ell '}  \ge 
\sum_{i\in N}\sum_{j\in M} a_{ij} x_{ik} y_{j\ell } -
\sum_{i\in N}\sum_{j\in M} a_{ij}x_{ik'} y_{j\ell'} & \forall k \in N, \ell  \in M, (k',\ell ')\in \N_{k\ell }\label{ctr:NL1}\\
&v_{k\ell k'\ell '}  \ge 
-\sum_{i\in N}\sum_{j\in M} a_{ij} x_{ik} y_{j\ell } +
 \sum_{i\in N}\sum_{j\in M} a_{ij}x_{ik'} y_{j\ell'} & \forall k \in N, \ell \in M, (k',\ell ')\in \N_{k\ell }\label{ctr:NL2}\\
&w_{k\ell k'\ell '}  \ge v_{k\ell k'\ell '}^p & \forall k \in N, \ell \in M, (k',\ell ')\in \N_{k\ell } \label{f1:p}\\
&\theta_{k\ell }  \ge \sum_{(k',\ell ')\in \N_{k\ell }} w_{k\ell k'\ell '} &
\forall k \in N, \ell \in M.
\end{align}
Thus, with this rewriting, the objective function is $
\dsum_{k\in N} \dsum_{\ell \in M} \theta_{k\ell}
$, 
and constraints \eqref{f1:p} can be efficiently expressed as a set of second order cone constraints \citep[see e.g.][]{blanco2024minimal,blanco2014revisiting}. 

The rest of the non-linear constraints, namely \eqref{ctr:NL1} and \eqref{ctr:NL2} can be \emph{linearized} using different techniques. In what follows, we provide two alternative linearization for these constraints.

Without loss of generality, we assume that the matrix $A \in \R^{n\times m}$ has all its entries in $[0,1]$. Otherwise, one can transform the matrix $A$ to the normalized matrix $\tilde{A}$ as:
$$
\tilde{a}_{ij} = \frac{a_{ij} - a_{\rm min}}{a_{\rm max}}
$$
where $a_{\rm min} = \min_{i'\in N, j'\in M} \{a_{i'j'}\}$ and $a_{\rm max} = \max_{i'\in N, j'\in M} \{a_{i'j'}\}$. Thus, the objective function for $\tilde{A}$ reads:
$$
 \frac{1}{|a_{\rm max}|^p}\sum_{k\in N}\sum_{\ell \in M} \sum_{(k',\ell ') \in \mathcal{N}_{k\ell}} 
    \Big|\sum_{i\in N}\sum_{j\in M} a_{ij}x_{ik} y_{j\ell} -
         \sum_{i\in N}\sum_{j\in M} a_{ij}x_{ik'}y_{j\ell '} \Big|^p.
$$
Thus, solving the problem for $\tilde{A}$ with an obtained stress measure of $\tilde{\rho}$ results in an equivalent optimal stress seriation for $A$ with stress measure $\rho = |a_{\rm max}|^p \; \tilde{\rho}$.

\subsection{Position Assignment Model: Linearization 1}


For the linearization of the product of binary variables, we proceed, as usual, by introducing auxiliary variables that represent the product of $x$- and $y$-variables:
$$
z_{ijk\ell } = x_{ik} y_{j\ell} = \begin{cases}
1 & \mbox{ if $x_{ik}=1$ and $y_{j\ell }=1$}\\
0 & \mbox{otherwise}
\end{cases} \quad \forall i, k \in N, j, \ell  \in M.
$$
Thus, all the products $x_{ik} y_{j\ell }$ in the objective function can be replaced by the variable $z_{ijk\ell }$ after adding the following constraints to the problem for all $i, k \in N$, $j, \ell  \in M$:
\begin{align*}
    z_{ijk\ell } &\geq x_{ik}+y_{j\ell }-1\\
    z_{ijk\ell } &\leq x_{ik} \\
    z_{ijk\ell } &\leq y_{j\ell }\\
    z_{ijk\ell } &\ge 0
\end{align*}
Note that the above formulation has $O(n^2 + m^2)$ binary variables, $O(n^2 m^2)$ continuous variables, $O(n^2 m^2)$ linear constraints, and $O(nm)$ nonlinear constraints. For the $\ell_1$ norm ($p=1$), the model simplifies to only linear constraints. Furthermore, the special case of square matrices where, additionally, the rows and columns are enforced to perform the same transformation, further simplifies to a single set of permutation variables, since $x=y$.

As expected, the seriation problem has multiple optimal solutions that may slow down the solution procedure. Some of these optimal solutions can be avoided by imposing the following constraint to break symmetries:
$$
\sum_{k\in N} 2^k x_{1k} \leq \sum_{k\in N} 2^k x_{2k},
$$
enforcing the first row to be permuted to a position smaller than the second row.

\subsection{Position Assignment Model: Linearization 2}

In the second linearization that we develop for the bilinear terms in formulation (PAM), we consider the following continuous (nonnegative) variables for every 
$k\in N$ and $\ell \in M$:
$$
s_{k\ell}: \text{ entry of the transformed matrix in position $(k,\ell)$.}
$$

Note that using the variables $x$ and $y$ that determine the permutations for the rows and columns, we can adequately define the $s$-variables as follows:
\begin{align*}
& s_{k\ell } \geq x_{ik} + \sum_{j\in M} a_{ij} y_{j\ell } -1
    \ \ \ \ \forall i, k \in N, \ell  \in M\\
& s_{k\ell }\le 1 - x_{ik} + \sum_{j \in M} a_{ij} y_{j\ell }\ \ \ \ \forall i,k \in N, \ell  \in M.
\end{align*}
The first set of constraints says that in case row $i$ is arranged at position $k$, then, for the index $j$ that is assigned to position $\ell $ (which is exactly one by the assignment constraints), $s_{k\ell } \geq 1 + a_{ij} -1 = a_{ij}$. Otherwise, the constraint is redundant since we assume that all the entries of $A$ are smaller than $1$. The second set of constraints, similarly, force $s_{k\ell } \leq a_{ij}$. Thus, if row $i$ is assigned to position $k$ and column $j$ is assigned to position $\ell $, $s_{k\ell }$ takes value $a_{ij}$, as desired. Otherwise, if $i$ is not assigned to position $k$, the first set of constraints require $s_{k\ell } \geq 0 + a_{ij} - 1 (\leq 0)$, so $s_{k\ell }$ is not restricted. By the second set of constraints, in this situation $s_{k\ell } \leq 1 + a_{ij} (\geq 1)$, and then $s_{k\ell }$ is not restricted. Since each row and column are enforced to be assigned to exactly one position, one of these constraints will be activated, enforcing the $s$-variables to take the appropriate entry values.

Then, one can reformulate the stress seriation problem as:
\begin{align}
\min & \; \sum_{k\in N} \sum_{\ell \in M}  \sum_{(k',\ell ') \in \mathcal{N}_{k\ell}} \Big|s_{k\ell } - s_{k'\ell '}\Big|^p\label{f1a:obj}\\
\nonumber \mbox{s.t. } & \eqref{f1:ctr1}, \eqref{f1:ctr2}, \eqref{f1:ctr3}, \eqref{f1:ctr4}, \eqref{f1:ctr5} \\
& s_{k\ell } \geq x_{ik} + \sum_{j\in M} a_{ij} y_{j\ell } -1 & \forall i, k \in N, \ell  \in M \label{f1a:ctr5}\\
& s_{k\ell }\le 1 - x_{ik} + \sum_{j \in M} a_{ij} y_{j\ell } &   \forall i,k \in N, \ell  \in M. \label{f1a:ctr6}
\end{align}

Again, the problem above can be viewed as a mixed integer $p$-order cone optimization problem, by introducing auxiliary variables $\nu_{k\ell k'\ell '}$ to represent the expressions $|s_{k\ell }-s_{k'\ell '}|$, resulting in the model:
\begin{align*}
\min & \sum_{k,\ell\in N} \rho_{k\ell } \\
\nonumber \mbox{s.t. } & \eqref{f1:ctr1}, \eqref{f1:ctr2}, \eqref{f1:ctr3}, \eqref{f1:ctr4}, \eqref{f1:ctr5},
\eqref{f1a:ctr5}, \eqref{f1a:ctr6}\\
    & \nu_{k\ell k'\ell '} \geq s_{k\ell }- s_{k'\ell '} &
    \forall k \in N, \ell  \in M, (k',\ell ') \in \mathcal{N}_{k\ell }\\
    & \nu_{k\ell k'\ell '} \geq s_{k'\ell '}-s_{k\ell } &
    \forall k \in N, \ell  \in M, (k',\ell ') \in \mathcal{N}_{k\ell }\\
    & \rho_{k\ell } \geq \sum_{(k'\ell ')\in \N_{k\ell }} \nu_{k\ell k'\ell '}^p & \forall k \in N, \ell \in M.\\
\end{align*}

Additionally, we consider the following set of valid inequalities to strengthen the model:
\begin{align*}
    & \sum_{\ell\in M} s_{k\ell } = \sum_{i\in N}\sum_{j\in M} a_{ij} x_{ik} & \forall k \in N\\
     & \sum_{k \in N} s_{k\ell } = \sum_{i\in N}\sum_{j\in M} a_{ij} y_{j\ell } & \forall \ell  \in M\\
    & s_{k\ell } \geq y_{j\ell } + \sum_{i\in N} a_{ij}x_{ik} -1 & 
    \forall j, \ell  \in M, k \in N \\
    & s_{k\ell }\le 1 - y_{j\ell } + \sum_{i \in N} a_{ij} x_{ik} & 
    \forall j,\ell  \in M, k \in N.
\end{align*}

\section{Hamiltonian Path Models}  \label{sec:3}

In this section we present different families of models based on the representation of the sequence of sorted rows and columns as Hamiltonian paths of networks induced by the matrix and the stress efficiency measures considered in this paper.

Through this section, we consider the two networks $G_R(A)$ and $G_C(A)$ defined respectively by the complete undirected graphs $\K_n$ and $\K_m$. For the sake of simplicity, we will obviate the mention of the matrix $A$ when denoting the graphs, unless necessary. The optimal sorted sequences of rows and columns under the different criteria will be obtained by calculating an optimal weighted Hamiltonian path on those networks, whose weights are induced by the measure to be optimized. Recall that a Hamiltonian path  visits every vertex exactly once. Thus, the identification between a seriation problem and Hamiltonian path is natural (every Hamiltonian path in $\K_n$ is a sequence of $n$ elements in 
$\{1,\ldots, n\}$ and, then, equivalent to a permutation $\sigma_s \in \Per_s$), as already observed in  \cite{laporte1978seriation,brusco2001compact}.

By the above identification, the matrix seriation problems allow for a Hamiltonian path reformulation of the problem.
To obtain the optimal sorting of the rows and columns, a minimum/maximum weight Hamiltonian path on the corresponding network
has to be found, i.e, a TSP has to be solved. This is an NP-complete problem~\citep{garey1979computers}. 
Several formulations have been proposed for it, which mainly differ in the way the subtours are eliminated, 
either adding exponentially many constraints, as DFJ~\citep{dantzig1954solution}, or adding a reduced number of 
auxiliary variables, as MTZ~\citep{miller1960integer} and GG~\citep{gavish1978travelling}. 
In our experiments we use the formulation proposed in \cite{gavish1978travelling} since it resulted in a good performance in our preliminary tests compared to other alternatives.

Although in the following sections we will provide the weights that are required to optimize particular criteria, 
we give in the following the mathematical optimization models to compute the two Hamiltonian paths on the networks 
$(G_R,\omega)$ and $(G_C,\tau)$.

In this formulation we use the following sets of variables:
$$
z^R_{ik} = \begin{cases}
    1 & \mbox{if row $i$ precedes row $k$} \\
    0 & \mbox{otherwise}
\end{cases}\ \ 
~z^C_{j\ell} = \begin{cases}
    1 & \mbox{if column $j$ precedes column $\ell$} \\
    0 & \mbox{otherwise}
\end{cases},
$$
for all $i, k \in N (i\neq k)$ and  $j, \ell \in M (j\neq \ell)$,
$$
t^R_i = \begin{cases}
    1 & \mbox{if row $i$ is sorted in last position } \\
    0 & \mbox{otherwise}
\end{cases}\ 
t^C_j = \begin{cases}
    1 & \mbox{if column $j$ is sorted in last position} \\
    0 & \mbox{otherwise}
\end{cases},
$$
for all $i \in N$, $j \in M$, and
$$
g^R_{ik} = \mbox{ position of row arc $(i,k)$ in the sequence},\; g^C_{j\ell} = \mbox{ position of column arc $(j,\ell)$ in the sequence}\ 
$$
for all $i, k \in N (i\neq k)$ and  $j, \ell \in M (j\neq \ell)$.

With these variables, the optimal stress seriation problem for sorting elements of 
$R$ (rows or columns of $A$) can be obtained by solving the following mixed integer linear optimization problem:
\begin{align}
 \hbox{(HPM)}\ \min / \max & \sum_{i,k\in N\atop i\neq k} \omega^R_{ik} z^R_{ik} + \sum_{j, \ell\in M\atop j\neq \ell} \omega^C_{j\ell} z^C_{j\ell}\label{tsp:obj}\\
    \mbox{s.t. } & \sum_{k\in N\atop i\neq k} z^R_{ik} + t^R_i = 1 & \forall i\in N \label{tsp:1a}\\
    & \sum_{k\in M\atop j\neq \ell} z^C_{j\ell} + t^C_j = 1 & \forall j\in N \label{tsp:1b}\\
    & \sum_{i \in N} t^R_i = 1 \label{tsp:2a} \\
    & \sum_{j \in M} t^C_j = 1 \label{tsp:2b} \\
    & \sum_{i\in N\atop i\neq k} z^R_{ik} \le 1 & \forall k\in N, \label{tsp:3a}\\
    & \sum_{j\in M\atop j\neq \ell} z^C_{j\ell} \le 1 & \forall \ell\in M, \label{tsp:3b}\\
    & \sum_{k\in N} g^R_{ik} \ge \sum_{k\in N} g^R_{ik}-n t^R_i+1 & \forall i\in N,\label{tsp:4a}\\
    & \sum_{\ell\in M} g^C_{j\ell} \ge \sum_{\ell\in M} g^C_{j\ell}- m t^C_j+1 & \forall j\in M,\label{tsp:4b}\\
    & g^R_{ik}\le (n-1)z^R_{ik} & \forall i, k \in N (i\neq k) \label{tsp:5a}\\
    & g^C_{j\ell}\le (m-1)z^C_{j\ell} & \forall j, \ell \in M (j\neq \ell) \label{tsp:5b}\\
    & g^R_{ik} \ge z^R_{ik} & \forall i, k\in N (i\neq k) \label{tsp:6a} \\
        & g^C_{j\ell} \ge z^C_{j\ell} & \forall j, \ell \in M (j\neq \ell) \label{tsp:6b} \\
    & z^R_{ik} \in \{0,1\} & \forall i, k \in N (i \neq k) \label{tsp:7a} \\
    & z^C_{j\ell} \in \{0,1\} & \forall j, \ell \in M (j \neq \ell) \label{tsp:7b} \\
    & t^R_i \in \{0,1\} & \forall i\in N,
    \label{tsp:8a}\\
     & t^C_j \in \{0,1\} & \forall j\in M,
    \label{tsp:8b},\\
    & g^R_{ik} \in \Z_+ & \forall i, k\in N (i\neq k),\label{tsp:9a}\\
    & g^C_{j\ell} \in \Z_+ & \forall j, \ell\in M (j\neq \ell).\label{tsp:9b}
\end{align}

The objective function \eqref{tsp:obj} considers the sum of the weights of the two paths.  Then, we assure the adequate construction of the Hamiltonian path, both for rows (constraints \eqref{tsp:1a}, \eqref{tsp:2a}, \eqref{tsp:3a}, \eqref{tsp:4a}, \eqref{tsp:5a}, \eqref{tsp:6a}, \eqref{tsp:7a}, \eqref{tsp:8a}, and \eqref{tsp:9a}) and for the columns (constraints \eqref{tsp:1b}, \eqref{tsp:2b}, \eqref{tsp:3b}, \eqref{tsp:4b}, \eqref{tsp:5b}, \eqref{tsp:6b}, \eqref{tsp:7b}, \eqref{tsp:8b}, and \eqref{tsp:9b}). 
Constraints \eqref{tsp:1a}  and \eqref{tsp:1b} assure that every element (row/column) has a sucessor or it is the last in the sequence. 
Constraints \eqref{tsp:2a} and \eqref{tsp:2b} stand for a single last element. 
Constraints \eqref{tsp:3a} and \eqref{tsp:3b} state that every element is preceded by at most one other. These, together with the previous constraints, assure that each of the elements except the first one is preceded by one other element. 
The existence of subtours in the solutions is avoided by constraints \eqref{tsp:4a}-\eqref{tsp:6b}. Constraints \eqref{tsp:4a} and \eqref{tsp:4b} force that, unless the element is the last in the sequence, the position of the ingoing arc to this element is strictly previous to the position of the outgoing arc from it. Constraints \eqref{tsp:5a}-\eqref{tsp:6b} impose that the value of the $g$-variables is positive if and only if the arc is in the path, and in that case it is in the interval $[1,n-1]$ (for rows) or $[1,m-1]$ (for columns). Since $|R|-1$ arcs are to be allocated, each of them will be allocated to an integer in this interval.

The first observation is that the problem above can be exactly decomposed as two independent Hamiltonian paths. Second, in case $n=m$ and the same permutation is to be applied to rows and columns, a single minimum weight Hamiltonian path is required to be computed, but here the weights for the arcs  are replaced by the sum of the weights. Thus, the above problem remains, but the objective function to be optimized is:
$$
\sum_{i,k \in N\atop i\neq k} (\omega^R_{ik} + \omega^C_{ik}) z_{ik}
$$
(note that superscript, $C$ or $R$, in the variables is no longer required since both coincide when the permutations to be obtained are the same).

In what follows we detail how to derive those weights for stress measures under the von Neumann and Moore neighborhoods as well as for the Measure of Efficiency.

\subsection{von Neumann Neighborhoods}

One of the most frequently used neighborhood shapes is the one induced by the so-called von Neumann neighbors since is one of the 
simplest. With our notation it coincides with the sets $\N_{ij}^1$, that is, when $\varepsilon=1$ (left plot in Figure \ref{fig:neigh}). 

In this case, we endow the graphs $G_R$ and $G_C$ with the edge weights $\omega^{VN}$ and $\tau^{VN}$ defined as:
$$
\omega^{VN}_{ik}:= 2\sum_{j\in M} |a_{ij} - a_{kj}|^p, \; \forall i, k \in N\ (i\neq k),
$$ 
$$
\tau^{VN}_{j\ell} := 2\sum_{i\in N} |a_{ij} - a_{i\ell}|^p, \; \forall j, 
\ell \in M\ (j\neq \ell).
$$
\begin{thm}\label{thm:1}
    The stress seriation problem for the matrix $A$ under the von Neumann neighborhood is equivalent to find two minimum weight Hamiltonian paths in the networks $(G_R,\omega^{VN})$ and $(G_C,\tau^{VN})$, respectively.
\end{thm}
\begin{proof}

Observe now that under the von Neumann neighborhoods, the stress function for entry $(i,j)$ becomes:
\begin{eqnarray*}
\begin{split}
|a_{\sigma_r(i)\sigma_c(j)} - a_{\sigma_r(i)+1,\sigma_c(j)}|^p + |a_{\sigma_r(i)\sigma_c(j)} - a_{\sigma_r(i)-1,\sigma_c(j)}|^p +\\ |a_{\sigma_r(i)\sigma_c(j)} - a_{\sigma_r(i),\sigma_c(j)+1}|^p +
|a_{\sigma_r(i)\sigma_c(j)} - a_{\sigma_r(i),\sigma_c(j)-1}|^p,
\end{split}
\end{eqnarray*}
that is, it only depends on the next and previous elements in the sequence of rows induced by the permutations $\sigma_r$ and $\sigma_c$, separately, in case they exist. Thus, the global stress function (sum over all the indices $(i,j) \in N \times M$) coincides with:
\begin{eqnarray*}
\begin{split}
{\rm VN}^p(\sigma_r,\sigma_c) &:= 
\sum_{i\in N\atop j \in M} S_{ij}^{\ell_p}(\sigma_r,\sigma_c) = 
\sum_{i\in N\atop j \in M} \Big(
|a_{\sigma_r(i)\sigma_c(j)} - a_{\sigma_r(i)+1,\sigma_c(j)}|^p + 
|a_{\sigma_r(i)\sigma_c(j)} - a_{\sigma_r(i)-1,\sigma_c(j)}|^p+ \\ &
|a_{\sigma_r(i)\sigma_c(j)} - a_{\sigma_r(i),\sigma_c(j)+1}|^p +
|a_{\sigma_r(i)\sigma_c(j)} - a_{\sigma_r(i),\sigma_c(j)-1}|^p
\Big) = \\
& \sum_{i\in N\atop j \in M} \Big(
|a_{ij} - a_{i^+j}|^p + |a_{ij} - a_{i^-j}|^p  + |a_{ij} - a_{ij^+}|^p +
|a_{ij} - a_{ij^-}|^p \Big),
\end{split}
\end{eqnarray*}
where we denote by $i^+$ and $i^-$ ($j^+$ and $j^-$) the predecessor and successor row (respectively column) of row $i\in N$ (resp.\ column $j\in M$) with respect to the permutation $\sigma_r$ (resp.\ $\sigma_c$). 

Denoting by $H_N=(N,A_N)$ and $H_M=(M,A_M)$ the Hamiltonian paths induced by the permutations, the above expression is equivalent to:
\begin{eqnarray*}
\begin{split}
&\sum_{j\in M} \Big(\sum_{(i,k) \in A_N} |a_{ij} - a_{kj}|^p + \sum_{(k,i) \in A_N} |a_{ij} - a_{kj}|^p\Big) + \sum_{i\in N} \Big( \sum_{(j,\ell ) \in A_M} |a_{ij} - a_{i\ell }|^p + \sum_{(\ell ,j) \in A_M} |a_{ij} - a_{i\ell }|^p\Big) \\
&= 2 \sum_{j\in M} \sum_{(i,k) \in A_N} |a_{ij} - a_{kj}|^p + 2  \sum_{i\in N}\sum_{(j,\ell ) \in A_M} |a_{ij} - a_{i\ell }|^p = \sum_{(i,k) \in A_N} \omega_{ik}^A + \sum_{(j,\ell ) \in A_M} \tau_{j\ell }^A.
\end{split}
\end{eqnarray*}
Thus, the stress function coincides with the sum of the weights of the Hamiltonian paths. 

Since the stress expression is then separable by the two sets of nodes in the graphs ($N$ and $M$), it is minimized by finding the two minimum weight Hamiltonian paths.
 \end{proof}

\subsection{Moore Neighbors}

Another popular neighborhood shape for the entries of the matrix is Moore neighborhood. It corresponds to the sets $\N_{ij}^{1.5}$, that is, fixing $\varepsilon = 1.5$ (center plot in Figure \ref{fig:neigh}). For a general entry $(i,j)$ in the matrix $A$, not in the first or last row or column,  the neighbors are:
$$
\N_{ij}^{1.5} = \N_{ij}^1 \cup \{a_{i+1,j+1}, a_{i+1,j-1}, a_{i-1,j+1}, a_{i-1,j-1}\}.
$$
In this case, although still the two permutations to be found can be seen as Hamiltonian paths in a certain graph, the problem is no longer separable. In what follows, we detail how to develop a model, based on the construction of Hamiltonian paths with products of variables in the original space of variables. A similar approach can be also used to formulate the $\varepsilon=2$ case although in this case three indices controlling the sequences of three consecutive rows and columns are required.

Let us consider the graph $ G = \K_n \sqcup \K_m$, the disjoint union of the complete graphs with node sets the rows and columns of the matrix to be seriated. As already mentioned in the previous section, computing two disjoint Hamiltonian graphs in $G$ (one in $\K_n$ and the other in $\K_m$) results in a permutation of rows and columns of the matrix. To asses the goodness of the Hamiltonian path in view of the stress seriation measure for Moore neighbors, it is required to consider the interaction between the two Hamiltonian paths. Specifically, the stress measure in this case requires, apart from the stress function for the von Neumann neighbors described above, the following terms:
\begin{eqnarray*}
\begin{split}
|a_{\sigma_r(i)\sigma_c(j)} - a_{\sigma_r(i)+1,\sigma_c(j)+1}|^p + 
|a_{\sigma_r(i)\sigma_c(j)} - a_{\sigma_r(i)+1,\sigma_c(j)-1}|^p + \\ 
|a_{\sigma_r(i)\sigma_c(j)} - a_{\sigma_r(i)-1,\sigma_c(j)+1}|^p +
|a_{\sigma_r(i)\sigma_c(j)} - a_{\sigma_r(i)-1,\sigma_c(j)-1}|^p.
\end{split}
\end{eqnarray*}
Now, the measure does not only depend on the next and previous elements in the sequence of rows or columns, separately, but also on the position of the other permuted index.
Summing up all these terms in the global stress function we have:
\begin{eqnarray*}
&\displaystyle{\sum_{i\in N\atop j \in M} S_{ij}^{\ell_p}(\sigma_r,\sigma_c) =}\\
&\displaystyle{{\rm M}^p(\sigma_r,\sigma_c) + \sum_{i\in N\atop j \in M} \Big(|a_{ij} - a_{i^+j^+}|^p + |a_{ij} - a_{i^+j^-}|^p +|a_{ij} - a_{i^-j^+}|^p +|a_{ij} - a_{i^-j^-}|^p\Big).}
\end{eqnarray*}

Thus, denoting by $H_N=(N,A_N)$ and $H_M=(M,A_M)$ the Hamiltonian paths induced by the rows and columns permutations, the above expression is equivalent to:
\begin{eqnarray*}
\begin{split}
& 2 \sum_{j\in M} \sum_{(i,k) \in A_N} |a_{ij} - a_{kj}|^p + 
2  \sum_{i\in N}\sum_{(j,\ell) \in A_M} |a_{ij} - a_{i\ell}|^p = \\
& \sum_{(i,k) \in A_N} \omega_{ik}^N + \sum_{(j,\ell) \in A_M} \omega^M_{j\ell} + 
2 \dsum_{(i,k) \in A_N} \sum_{(j,\ell) \in A_M} \Big(|a_{ij}-a_{k\ell}|^p + |a_{i\ell}-a_{kj}|^p \Big).
\end{split}
\end{eqnarray*}

Thus, using the $z$-, $t$- and $g$-variables given above, 
the following model constructs the optimal stress seriation for the matrix under the Moore neighbors:
\begin{align}
   \min & 
    \sum_{i,k \in N\atop i\neq k}       \omega^N_{ik} z^N_{ik} + 
    \sum_{j,\ell\in M\atop j\neq \ell } \omega^M_{j\ell } z^M_{j\ell } + \nonumber\\
    & \dsum_{i,k \in N\atop i\neq k} \sum_{j,\ell \in M\atop j\neq\ell } 
    \Big( |a_{ij}-a_{k\ell }|^p + |a_{i\ell }-a_{kj}|^p \Big) z^N_{ik} z^M_{j\ell } \label{tsp:obj12} \\
\mbox{s.t. } & \sum_{s\in R\atop r\neq s} z^R_{rs} + t^R_r = 1 & \forall r\in R, & R=N,M \label{tsp:11}\\
    & \sum_{r \in R} t^R_r = 1 && R=N,M \label{tsp:12} \\
    & \sum_{r\in R\atop r\neq s} z^R_{rs} \le 1 & \forall s\in R, & R=N,M\label{tsp:13}\\
    & \sum_{s\in R} g^R_{rs} \ge \sum_{s\in R} g^R_{sr}-|R|t^R_r+1 & \forall r\in R, & R=N,M \label{tsp:14}\\
    & g^R_{rs}\le (|R|-1)z^R_{rs} & \forall r,s \in R (r\neq s) &, R=N,M \label{tsp:15}\\
    & g^R_{rs} \ge z^R_{rs} & \forall r,s\in R (r\neq s), & R=N,M \label{tsp:16} \\
    & z^R_{rs} \in \{0,1\} & \forall r,s \in R (r\neq s), & R=N,M \label{tsp:17} \\
    & t^R_r \in \{0,1\} & \forall r\in R, & R=N,M. \label{tsp:18}
\end{align}

An alternative form of objective function for this problem is
$$ \dsum_{i, k \in N\atop i\neq k} \sum_{j,\ell \in M\atop j\neq \ell} \Big(
|a_{ij}-a_{k\ell}|^p + |a_{i\ell}-a_{kj}|^p + |a_{ij}-a_{i\ell}|^p + |a_{ij}-a_{kj}|^p \Big) z^N_{ik} z^M_{j\ell}.$$
This, although valid, results in worse linear programming bounds and therefore in worse computing times.

\subsection{Measure of Efficiency}

Finding the optimal seriation with respect to the measure of efficiency can be also cast in the framework of computing a maximal Hamiltonian path on the set of rows/columns with certain weights. Specifically, defining:
    \begin{eqnarray*}
        \omega_{ik}^{ME} = \sum_{j\in M} a_{ij}a_{kj}, \quad \forall i,k\in N,\\
        \tau_{j\ell}^{ME}= \sum_{i\in N} a_{ij}a_{i\ell}, \quad \forall j,\ell\in M
    \end{eqnarray*}
we get the following result.
\begin{thm}
    The maximization of the measure of effectiveness  (\ref{eq:Mcornick}) is equivalent to find two maximum weighted Hamiltonian paths with weights $\omega^{ME}$, $\tau^{ME}$ on the networks $G_R$ and $G_C$, respectively.
\end{thm}
\begin{proof}
    The proof runs analogous  to that of Theorem \ref{thm:1}.
\end{proof}

\subsection{Other neighborhoods}

The models described above for the computation of an optimal stress matrix seriation can be adapted to other types of neighbors different from the classical ones of von Neumann and Moore. The main difference stems in the objective function measuring the \emph{distance} between the transformed cells and its neighbors. In what follows we describe two other types of neighbors where this methodology can be applied.

{\bf 2-steps cross-shaped neighbors: }  We assume now that the neighbors of a cell $(i,j)$ in the matrix are
$$
\N_{ij}^{\oplus} = \{(k,\ell) \in N \times M: (|k-i|\leq 2  \text{ and } \ell =0) \text{ or } (k =0 \text{ and } |j-\ell |\leq 2)\}
$$
as those drawn in Figure \ref{fig:neigh_cross}.

\begin{figure}[h]
    \begin{center}
    \begin{tikzpicture}[scale=0.5]
        \def\n{7} 

        \def\hx{5} 
        \def\hy{5} 
        
        \fill[black] (\hx,\hy) rectangle (\hx+1,\hy+1);

        \foreach \dx/\dy in {0/1, 1/0, 0/-1, -1/0, 0/2, 2/0, 0/-2, -2/0} {
            \fill[gray] (\hx+\dx,\hy+\dy) rectangle (\hx+\dx+1,\hy+\dy+1);
        }
        \def\hx{1} 
        \def\hy{1} 
        
        \fill[black] (\hx,\hy) rectangle (\hx+1,\hy+1);

        \foreach \dx/\dy in {0/1, 1/0, 0/2, 2/0} {
            \fill[gray] (\hx+\dx,\hy+\dy) rectangle (\hx+\dx+1,\hy+\dy+1);
        }

        \foreach \x in {1,...,\n} {
            \foreach \y in {1,...,\n} {
                \node at (\x+0.5,\y+0.5) {};
            }
        \foreach \x in {1,...,\n} {
            \foreach \y in {1,...,\n} {
                \draw[black] (\x,\y) rectangle (\x+1,\y+1);
            }
        }
        }
    \end{tikzpicture}
\end{center}
\caption{2-steps cross shaped neighbors\label{fig:neigh_cross}}
\end{figure}
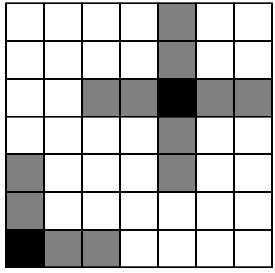

In this case, one can slightly modify the model \eqref{tsp:obj12}-\eqref{tsp:18} by introducing the variables $u^R_r$, 
for $R=N,M$ and $r\in R$, to account for the costs of comparing row (respectively column) $r$ with row (resp.\ column)
two upper rows (columns) in the rearranged matrix.
With the same set of variables as above, the new variables are adequately defined by the following set of constraints:
\begin{align*}
u^R_r \geq  \sum_{t\in R} \omega_{rt}^R z^R_{st} - \Delta^R_r (1-z^R_{rs})\ \ \ \forall r,s \in R,
\end{align*}
where $\Delta_r^R=\dmax_{s \in R} \omega^R_{rs}$, for all $r\in R$ and $R=N,M$.

The above constraints assure that in case $r$ precedes $s$ and row $s$ precedes row $t$ in the Hamiltonian path, then $u^R_r$ is at least $\omega^R_{rt}$. These constraints, together with the minimization of the objective function
$$
\sum_{i,k \in N\atop i\neq k}      \omega^N_{ik} z^N_{ik} + 
\sum_{j,\ell\in M\atop j\neq \ell }\omega^M_{j\ell } z^M_{j\ell} +
\sum_{i \in N} u^N_i + \dsum_{j\in M} u^M_j$$
assure the correct stress measure.

\subsection{Tailoring Seriation through linear Constraints}

The main advantage of the optimization-based methodologies we propose for computing the optimal seriation of a matrix under various measures is their flexibility to incorporate additional constraints on the reordered matrix. In what follows, we describe several types of constraints that may be relevant in this framework::
\begin{itemize}
    \item {\bf Cluster-based constraints}: Pre-specified groups of observations are usually required to closely appear in the resulting reordered matrix. This has been the case of archeological pieces whose nature has to be identified, but that have been found in the same region. Thus, the \emph{indices} between some groups observations want to be bounded above. Let then $\mathcal{C} \subset N$ the groups of indices that are to be \emph{protected}, and $\kappa \in \Z_+$ with $\kappa <n$. The following constraint ensures that the observations in $\mathcal{C}$ verify this condition:
    $$
    \left|\sum_{k\in N} g^N_{ik} - \sum_{k\in N} g^N_{jk}\right| \leq \kappa,\ \forall i, j \in \mathcal{C}.
    $$
    \item \textbf{Specification of flexible sorting positions:} In some applications, it may be desirable to position certain rows or columns at specific locations in the reordered matrix. For instance, some observations might need to appear at the center or at the edges of the matrix to enhance visual interpretation. 

For the case of rows, let $\mathcal{C} \subset N$ be a set of observations, and let $\mathcal{P} \subset N$ be a set of desired positions. The following constraint:
$$
\sum_{k \in \mathcal{P}} g_{ik}^N = 1, \quad \forall i \in \mathcal{C},
$$
enforces that each observation in $\mathcal{C}$ is assigned to a position in $\mathcal{P}$.
\end{itemize}

\section{Computational Experiments}  \label{sec:4}

In this section we report the results of the computational experiments that we run to validate our proposals and analyze the performance of the different approaches we provide.

All the codes and datasets used in our experiments are available at our GitHub repository: \url{https://github.com/vblancoOR/seriation_mathopt}.

\subsection{Synthetic Data}

We have generated different instances for our matrix seriation problems by creating multiple structured matrices of different types and sizes. The matrices vary in size, density and structural properties, making them suitable for evaluating our approaches.

The generated matrices are parameterized by their number of rows ($n$) and columns ($m$), where the values are drawn from the set $\{10, 20, 30, 40, 50, 100\}$ and such that either $n=m$ or $n<m$. Each configuration is iterated five times to introduce variability across instances.

We generate four types of instances:
\begin{itemize}
    \item \texttt{easy} instances: square matrices ($n \times n$) designed to contain structures that facilitate ordering. To generate them, we draw $n$ points in $[0,100]$ and output their distance matrix (pairwise absolute differences).
    \item \texttt{structured square} instances: Square matrices ($n\times n$) that contain more complex structured patterns. They are aso generated as pairwise Euclidean distances between randomly generated points in $[0,100]\times [0,100]$.
	\item \texttt{nonsquare} instances: For $n < m$, rectangular matrices ($n \times m$) are created with structured patterns that introduce complexity in ordering tasks. They are generated also as pairwise Euclidean distances between two sets of ($n$ and $m$) randomly generated points in $[0,100]\times [0,100]$.
    \item \texttt{binary square} instances: We randomly generate binary square instances ($n\times n$), but with certain densities (percent of elements that are $1$ in the whole matrix). We choose densities in $\{25\%,50\%,75\%\}$.
	\item \texttt{binary nonsquare} instances: Rectangular matrices ($n \times m$) that contain only binary values, and with the same density parameters as the binary square instances.
\end{itemize}
In total, we generate 450 instances that are available at the \texttt{github} repository \url{https://github.com/vblancoOR/seriation_mathopt}. For each of them we run the following models:
\begin{description}
    \item[Standard Stress Measures] For all the instances with $n\leq 20$, we run the three models proposed for the von Neumann and the Moore Strees Neighborhoods, namely, the \emph{position assignment model} (PAM) with the two type of linearizations, that we call \texttt{PAM\_L1} and \texttt{PAM\_L2}. For all the instances, we run the Hamiltonian path model (\texttt{HPM}).
    \item[Measure of Effectiveness] For all the instances we run the Hamiltonian path model \texttt{HPM} for the ME measure.
    \item[2-Crossed Neighborhoods] For the non-standard measures based on neighborhoods $\mathcal{N}_\varepsilon$ with $\varepsilon=3$, we run the position assignment models with the two linearizations.
    \item[Coordinated Rows/Columns] For the square matrices, we run all the models both requiring coordination of rows and columns in the seriation or not. For the non square matrices, we only run the models where the coordination is not imposed.
\end{description}
In total, we run $2760$ models for the generated instances. The formulations have been coded in Python 3.8 in a Huawei FusionServer Pro XH321 (\texttt{albaic\'in} at Universidad de Granada - \url{https://supercomputacion.ugr.es/arquitecturas/albaicin/}) with an Intel Xeon Gold 6258R CPU @ 2.70GHz with 28 cores. We used Gurobi 10.0.3 as  optimization solver. A time limit of 1 hour was fixed for each instance.

First we summarize some of the computational results of our experiments for the small ($n\leq 20$) instances. 

In Figure \ref{fig:Time_Formulation} we compare the performance on the small instances for the three types of formulations, namely, \texttt{PAM\_L1}, \texttt{PAM\_L2} and \texttt{HPM}. The $x$-axis represents CPU time (in seconds) for solving the instance, while the $y$-axis indicates the percentage of instances for which the approach is able to provide the optimal solution of the problem. As can be observed, the performance of the Hamiltonian path approach is much better than the other two approaches, being able to solve most of the instances in less than 10 minutes. The second best approach seems to be the one based on the alternative linearization that we propose for the general formulation.

\begin{figure}[h]
\centering\includegraphics[width=0.8\textwidth]{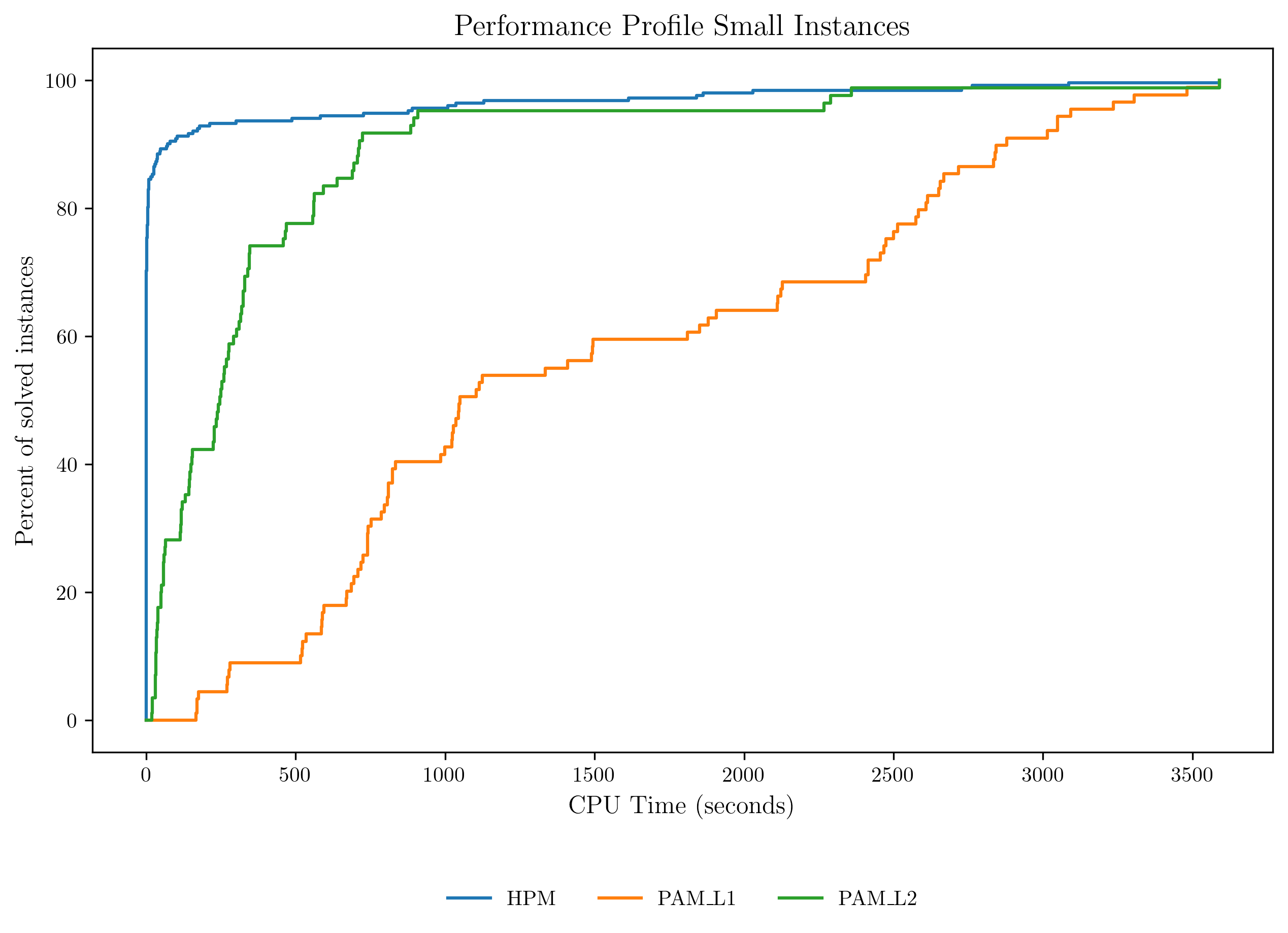}
\caption{Performance profile of our three types of models for the small instances\label{fig:Time_Formulation}}
\end{figure}

In Figure \ref{fig:Time_by_type} we report the performance profiles of each of the approaches, for the different types of instances that we tested in our experiments, and that were able to be optimally solved in less than 1 hour. The \texttt{HPM} approach is clearly the most \emph{stable} approach, with a similar performance for all types of instances. The general approach \texttt{PAM\_L1} performs very differently for the different types, being apparently the \texttt{eas} instances easier to solve than the others, followed by \texttt{sqr}, then \texttt{bin}, and finally \texttt{nsq}, for which none of the instances were solved up to optimality within the time limit. The performance of \texttt{PAM\_L2} looks a bit better although, again, the type of instances \texttt{nsq} seems to be very challenging for this formulation, and none of the instances of this type was optimally solved. We do not distinguish the performance of the \texttt{HPM} approach for the different densities of the \texttt{bin} type since there are no significant differences between them.

\begin{figure}[h]
\includegraphics[width=0.5\textwidth]{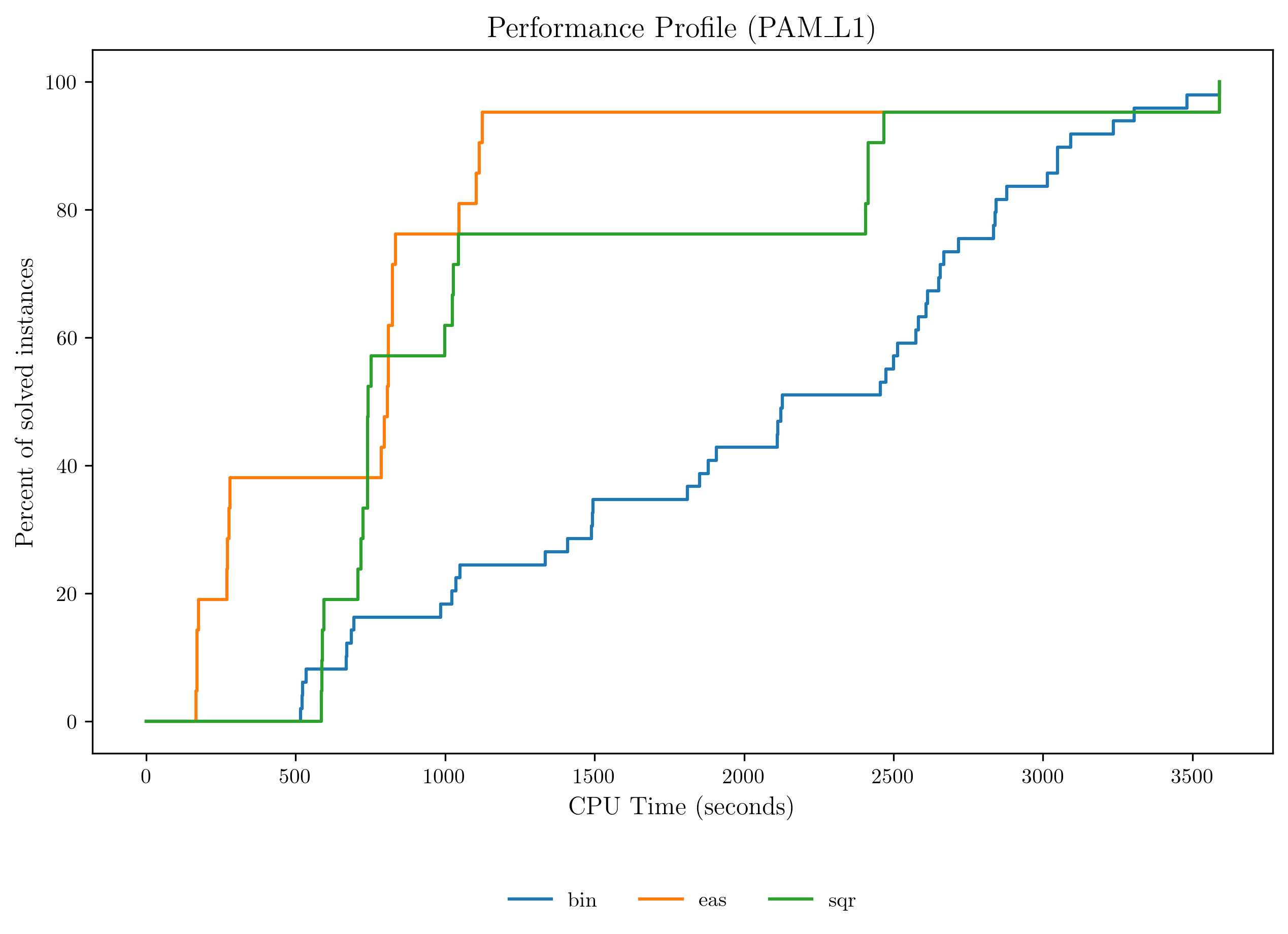}~\includegraphics[width=0.5\textwidth]{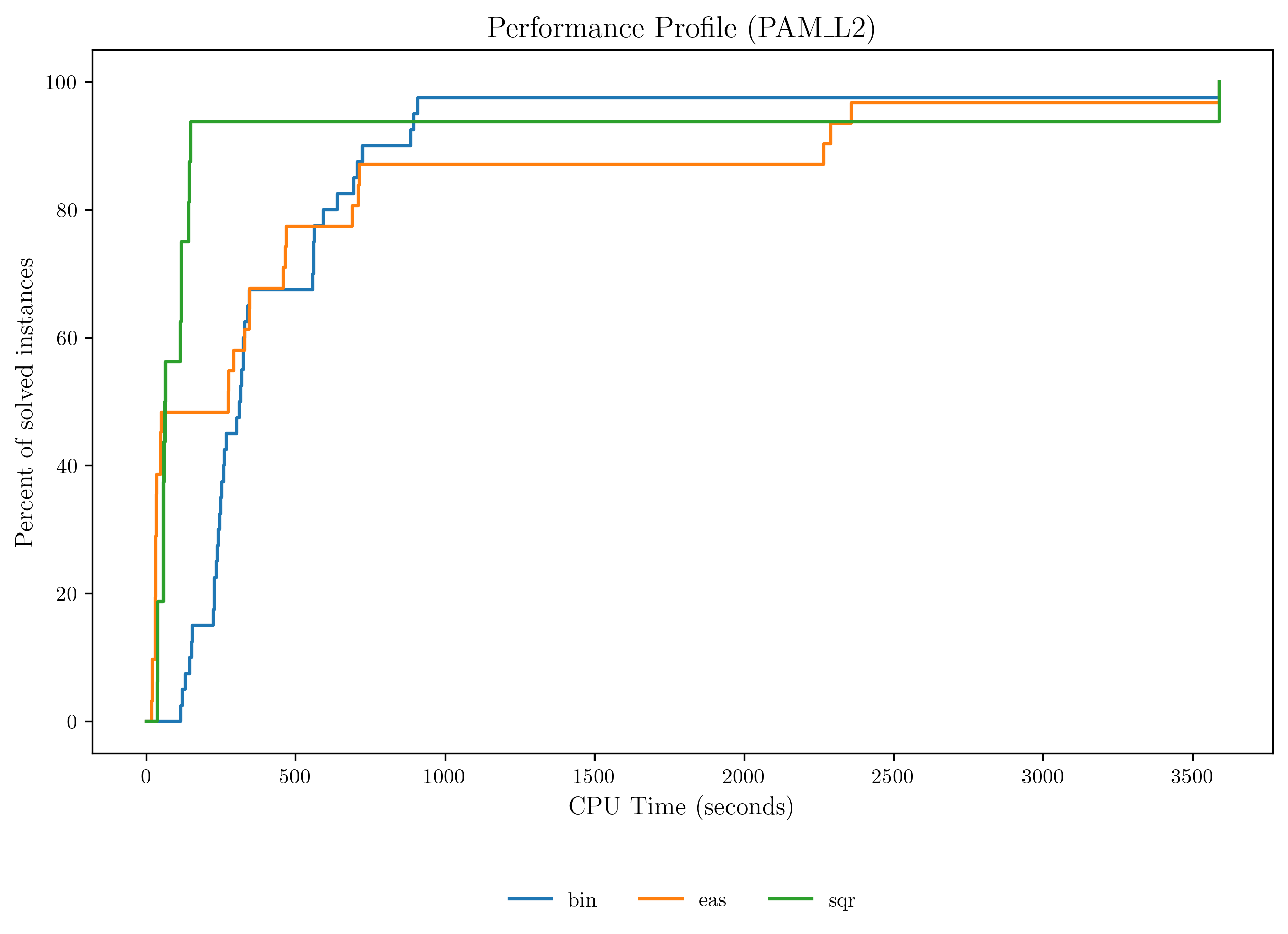}\\
\centerline{\includegraphics[width=0.5\textwidth]{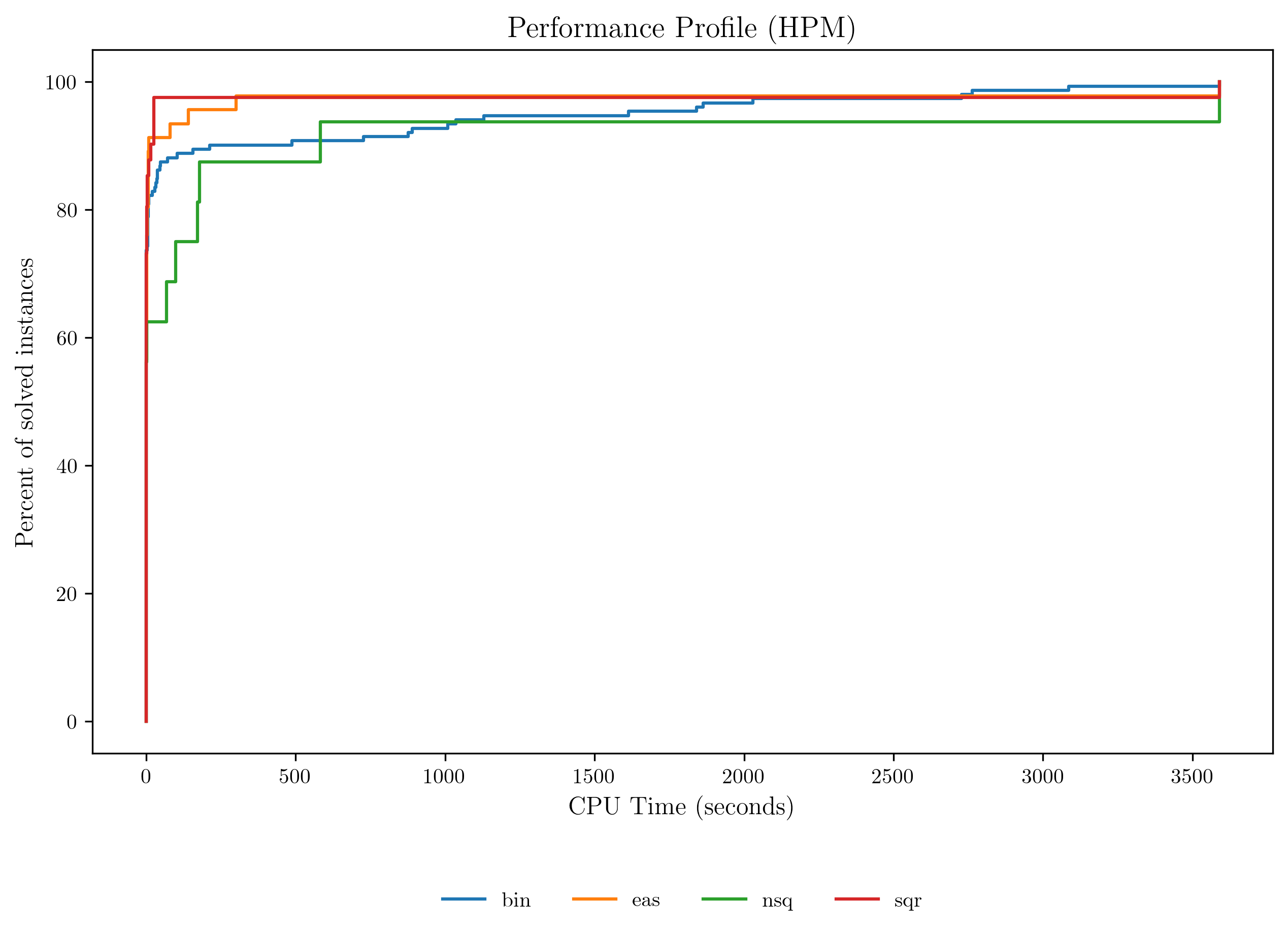}}
\caption{Performance profile of our three types of models for the small instances, for the different classes of instances that we tested\label{fig:Time_by_type}}
\end{figure}

A similar performance is observed in Figure \ref{fig:Time_by_coord} when distinguishing, for the square matrices, between the performance for the instances when it is forced to sort the rows and columns jointly or separately. While \texttt{PAM\_L1} was not able to solve optimally any of the \emph{separated} instances, \texttt{PAM\_L2} found a lot of differences in their performances, and in \texttt{HPM} the differences are negligible.

\begin{figure}[h]
\includegraphics[width=0.5\textwidth]{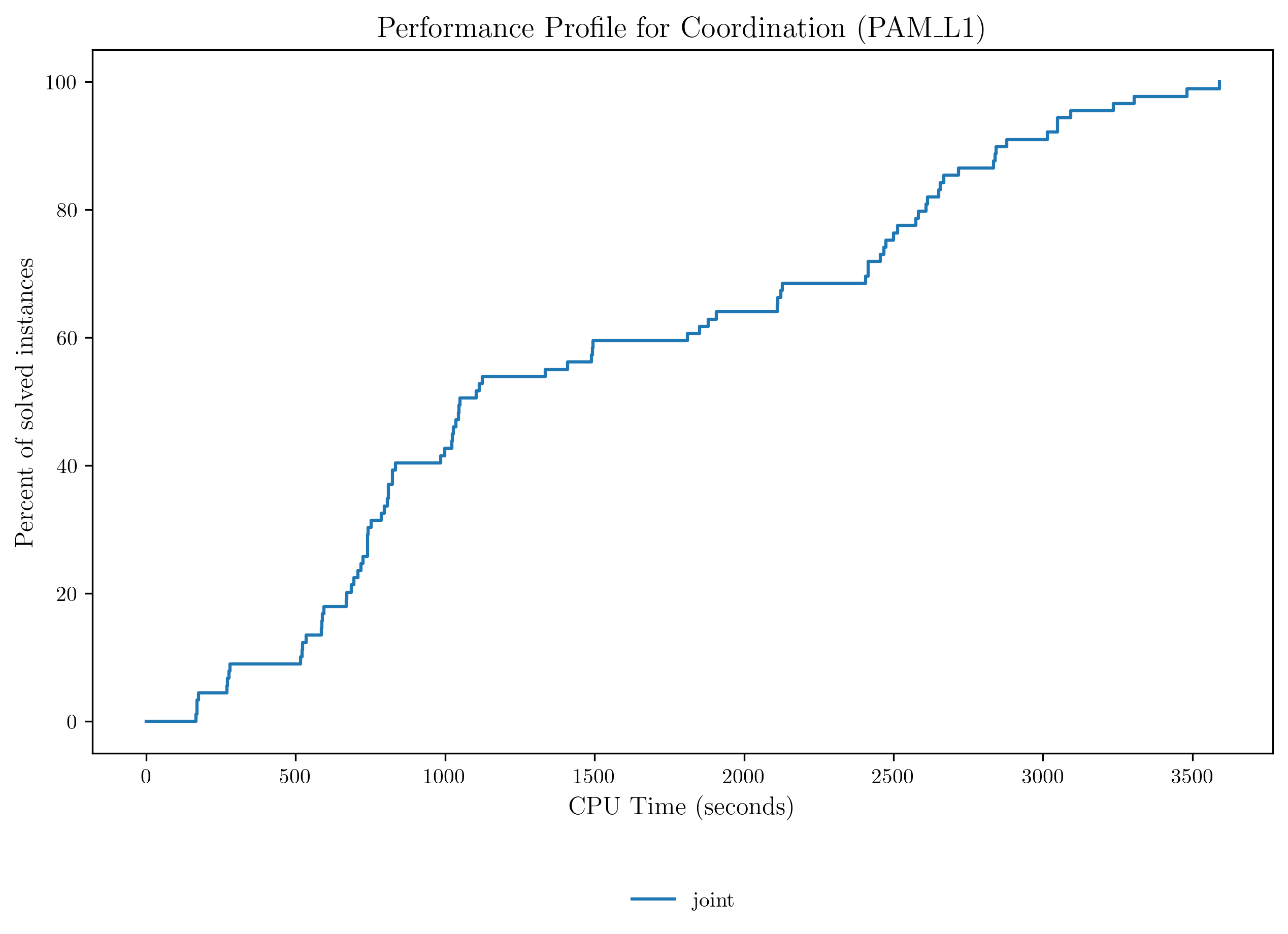}~\includegraphics[width=0.5\textwidth]{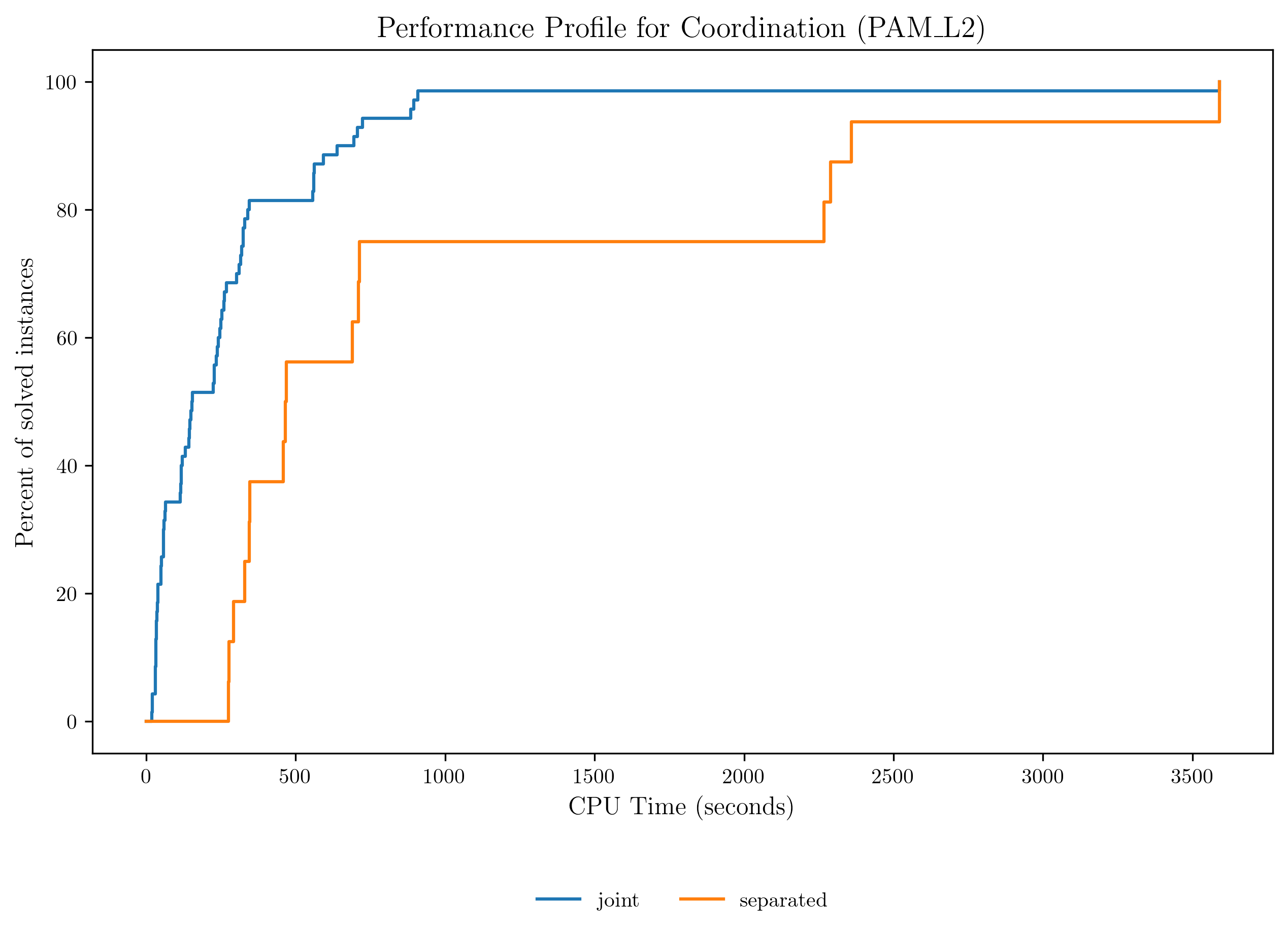}\\
\centerline{\includegraphics[width=0.5\textwidth]{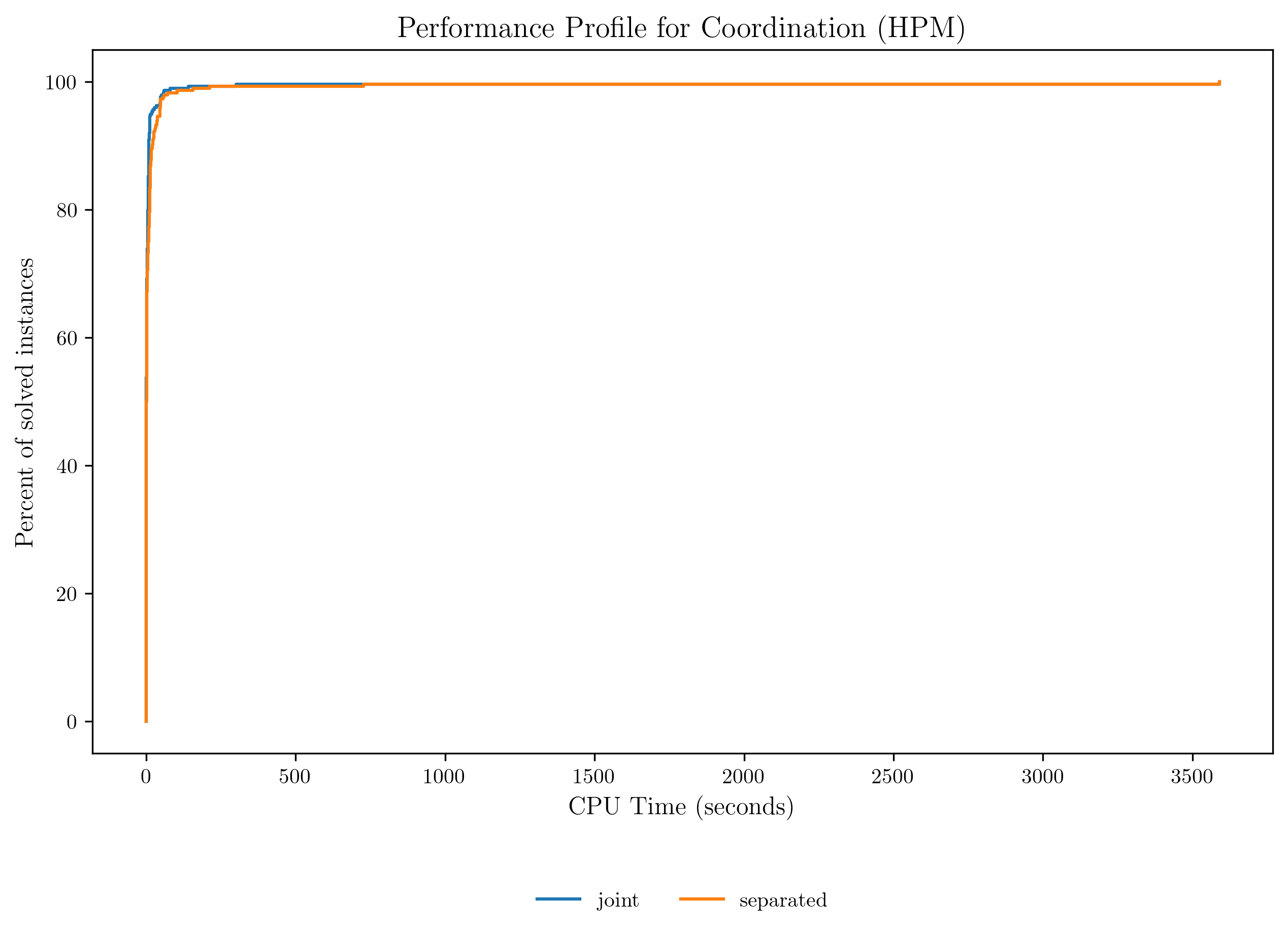}}
\caption{Performance profile of our three types of models for the small instances, for the joint or separate coordination of sorting rows and columns\label{fig:Time_by_coord}}
\end{figure}

A different perspective is found when analyzing the effect of the different measures (ME, von Neumann, Moore, and 2-Crossed) in the performance of the formulations, shown in Figure \ref{fig:Time_by_meas}. The position assignment models \texttt{PAM\_L1} and \texttt{PAM\_L2}, if able, compute the optimal solutions in similar times for the different approaches, whereas \texttt{HPM} found differences between the different measures. Furthermore, although possible, the incorporation of non standard neighborhoods, as the 2-Crossed one, is more difficulty considered in the Hamiltonian path approaches.

\begin{figure}[h]
\includegraphics[width=0.5\textwidth]{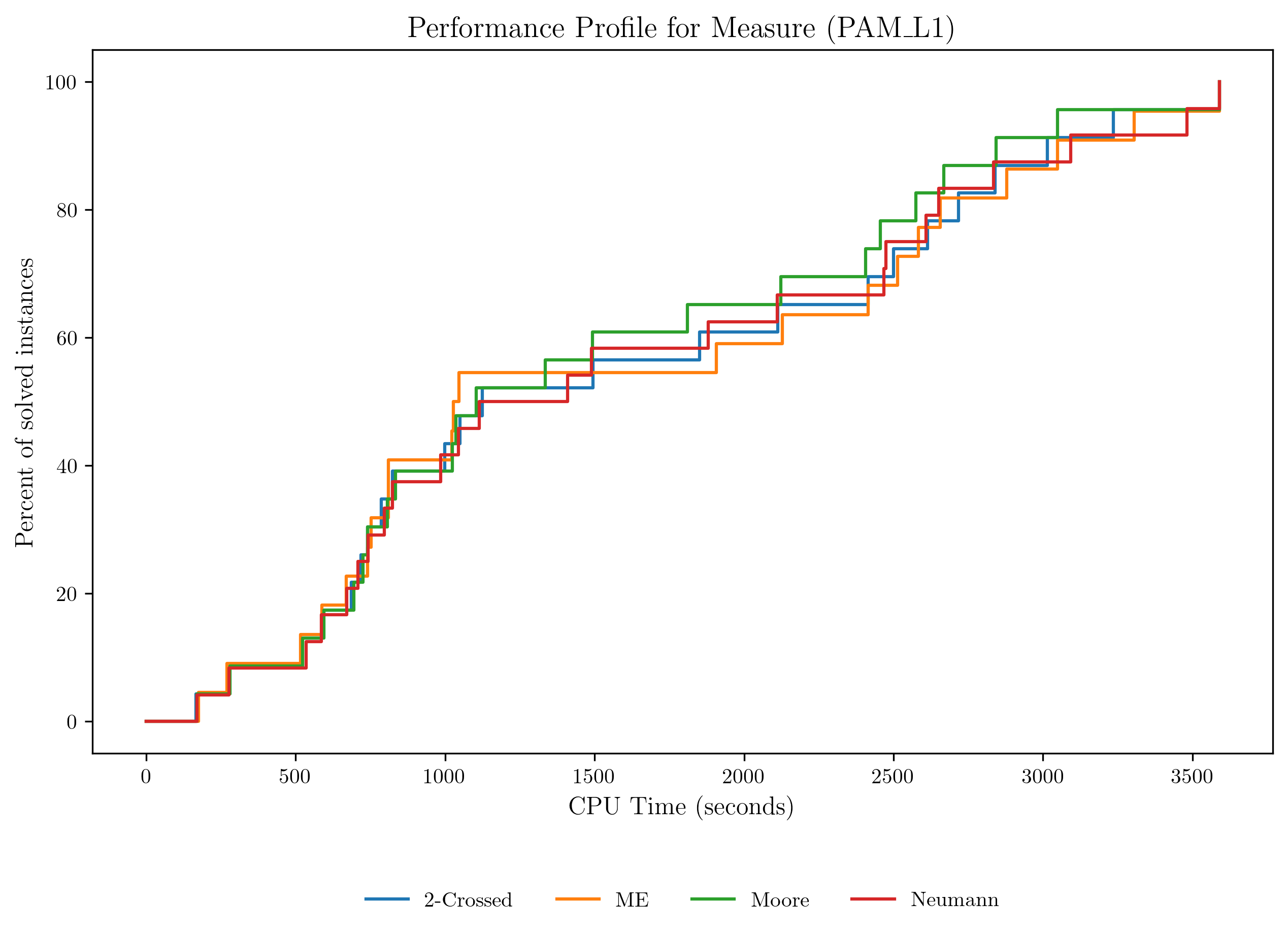}~\includegraphics[width=0.5\textwidth]{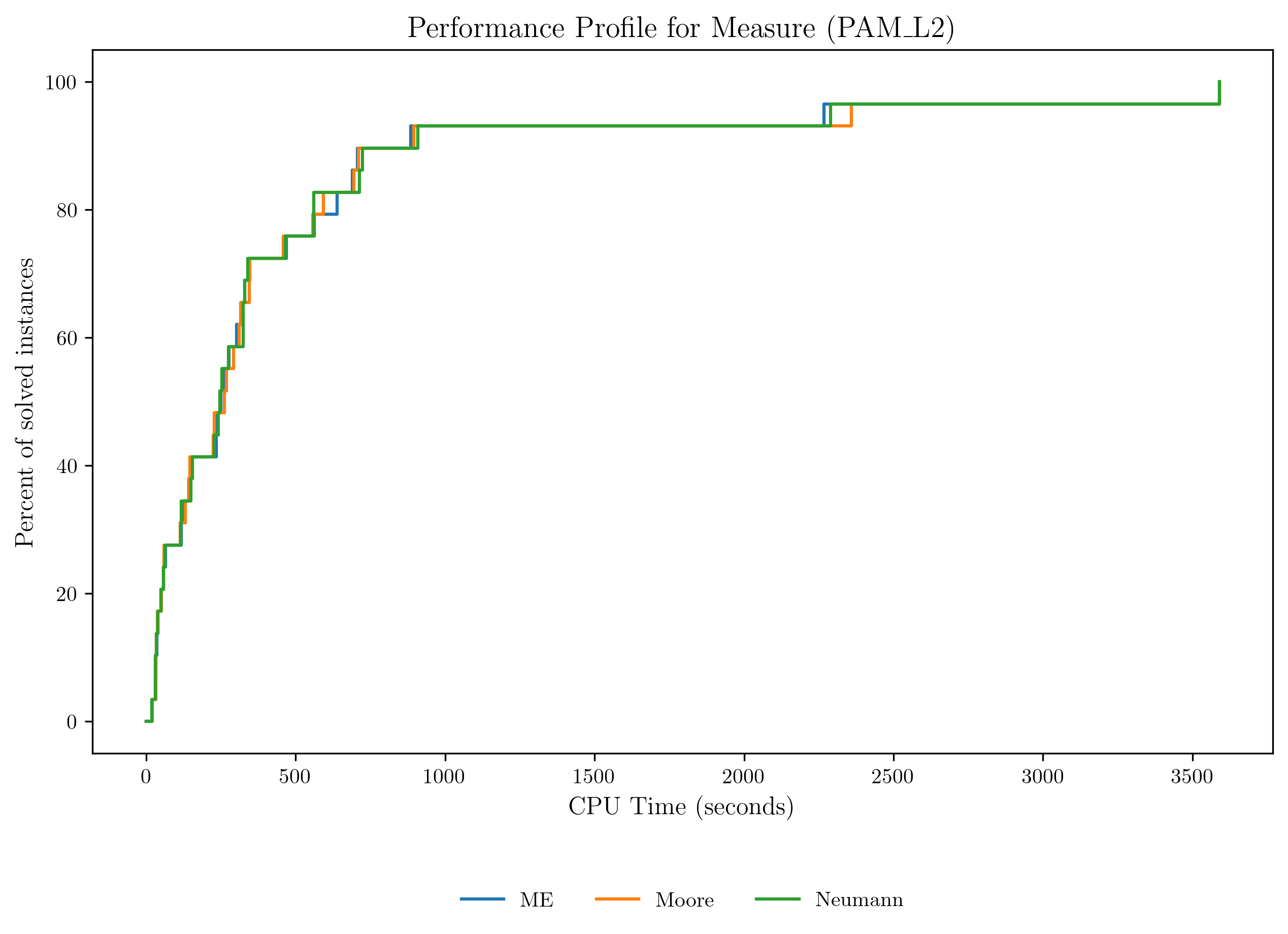}\\
\centerline{\includegraphics[width=0.5\textwidth]{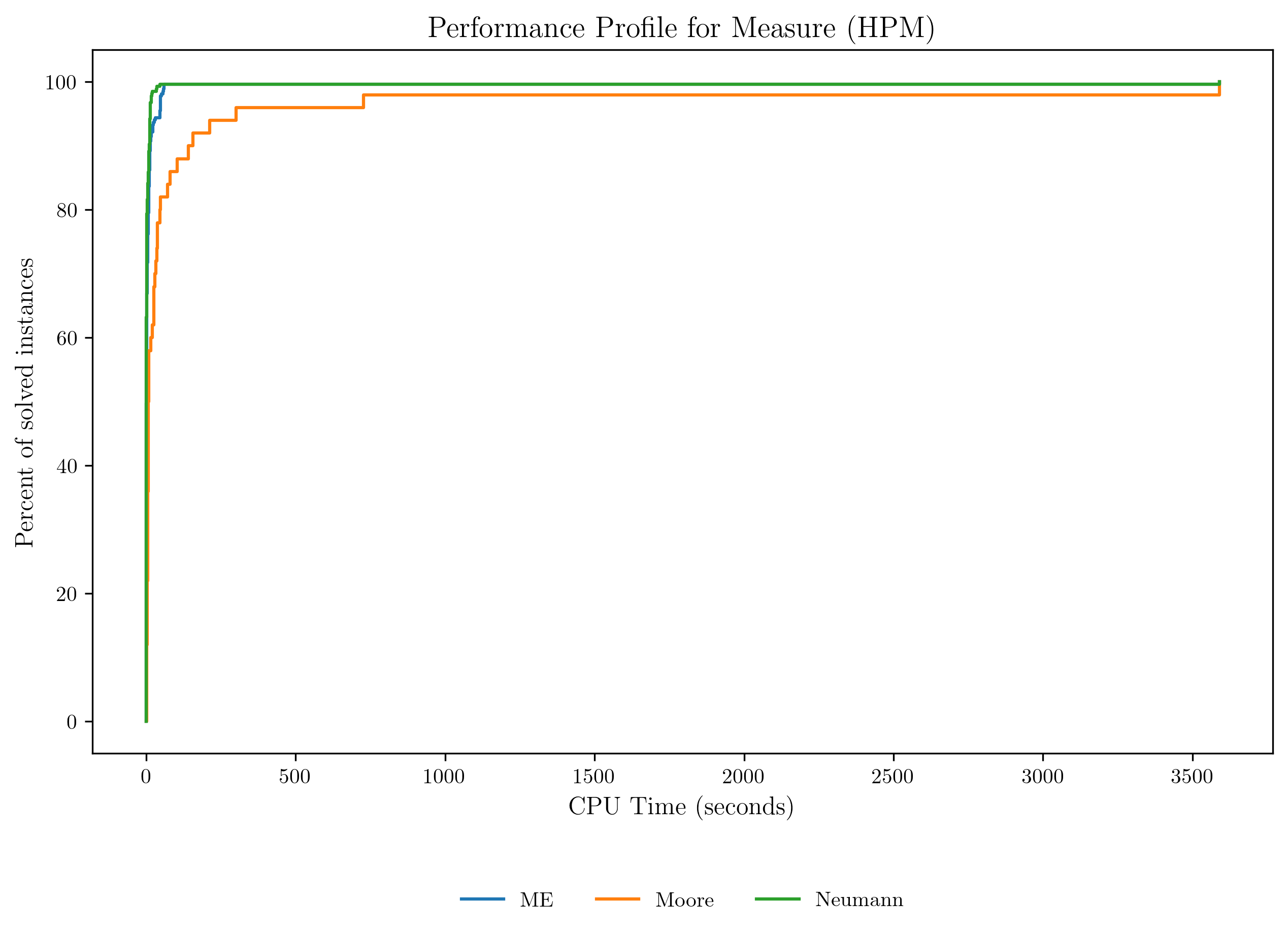}}
\caption{Performance profile of our three types of models for the small instances, for the different measures\label{fig:Time_by_meas}}
\end{figure}

In Figure \ref{fig:unsolved_small} we provide information about the unsolved small instances and their MIPGap.  
In the left we plot a bar chart with the percentage of unsolved instances with each of the models. Note that none of the $n=20$ instances were solved with the position assignment models, being only $20\%$ of the instances with $n=20$ those that were not solved by \texttt{HPM}. The performance of the MIPGAPs (right plot) also show the empirical computational complexity of the different models when solving these instances. The approach \texttt{PAM\_L1} found very large MIPGaps, \texttt{PAM\_L2} gave also large (although smaller) MIPGaps, with a large range of values, but \texttt{HPM} is quite stable with small to medium MIPGaps for the most challenging instances.

\begin{figure}[h]
\includegraphics[width=0.5\textwidth]{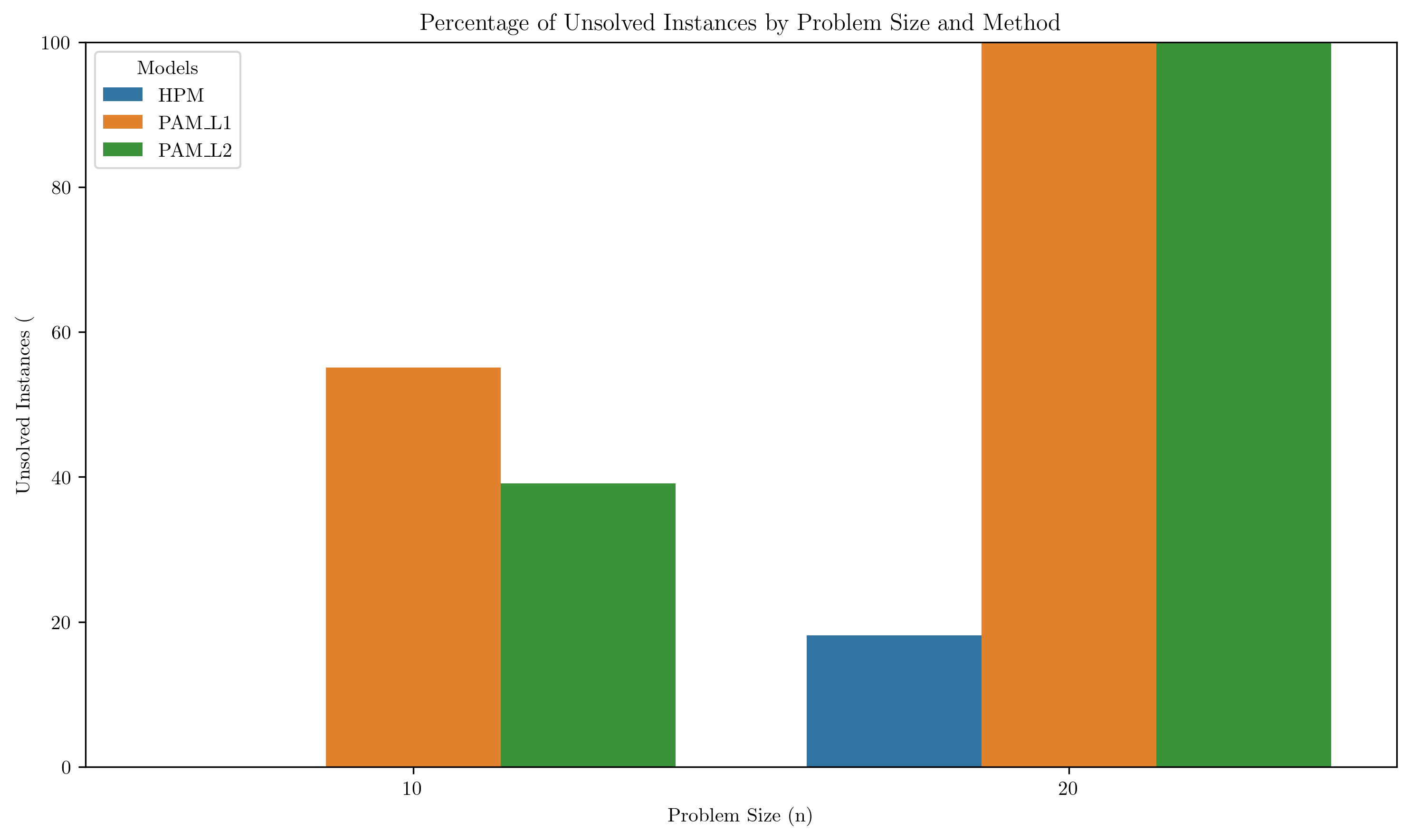}~\includegraphics[width=0.4\textwidth]{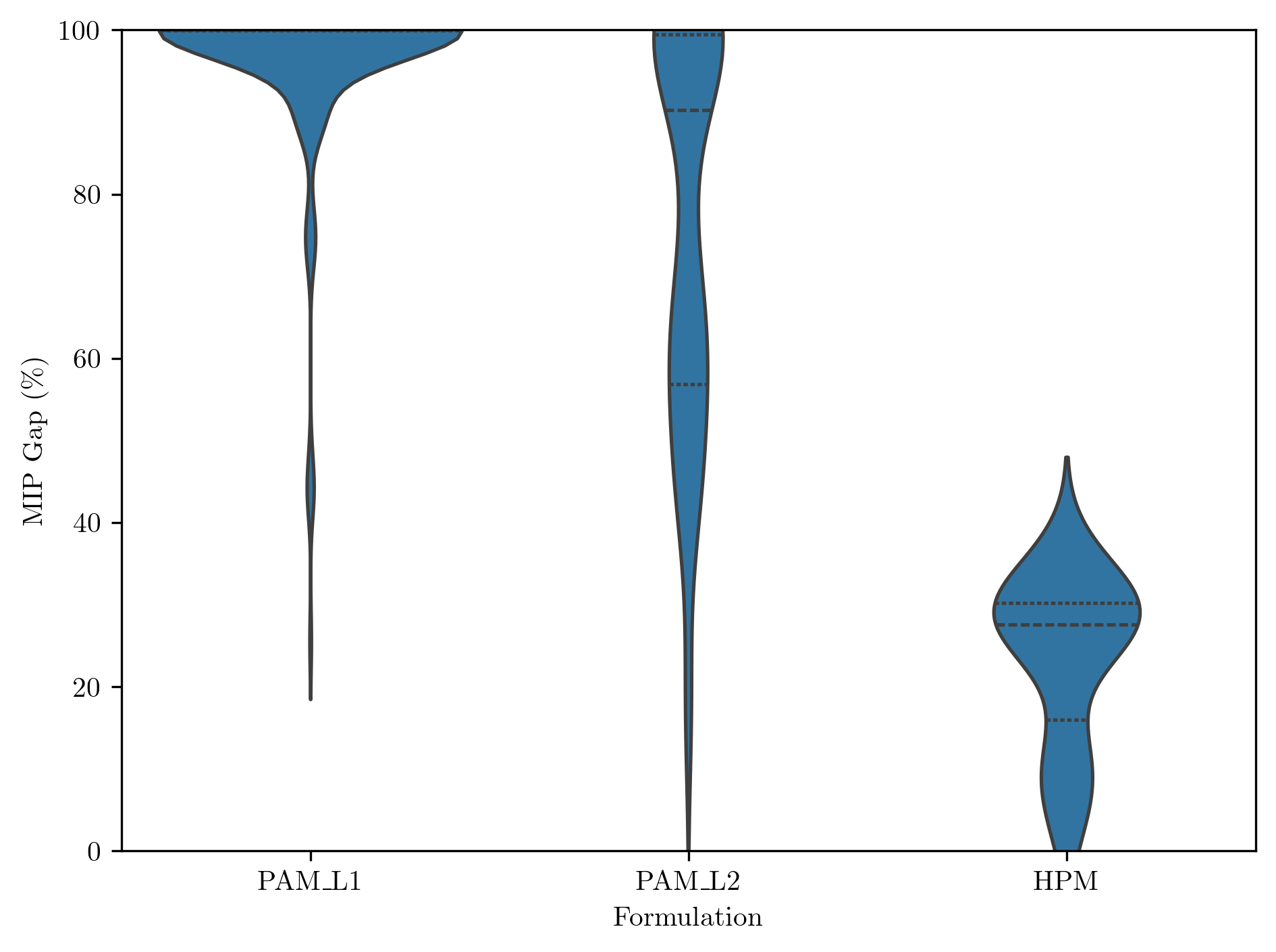}
\caption{Unsolved instances and average MIPGaps for the small instances.\label{fig:unsolved_small}}
\end{figure}

Finally, we report the results obtained for medium and large instances ($n\geq 30$) by the \texttt{HPM} model. 
In Figure \ref{fig:pp_large} we plot the performance profile for these instances for the different types of neighborhoods. 
Clearly, the Moore neighborhood was more time demanding than the other measures, being able to solve less than $10\%$ of the instances, where all the instances with the other measures were optimally solved. The simultaneous or not sorting of the rows and columns, the density of the \texttt{bin} instances, nor the size of the matrices seem to significantly affect the performance.

\begin{figure}[h]
\centering\includegraphics[width=0.5\textwidth]{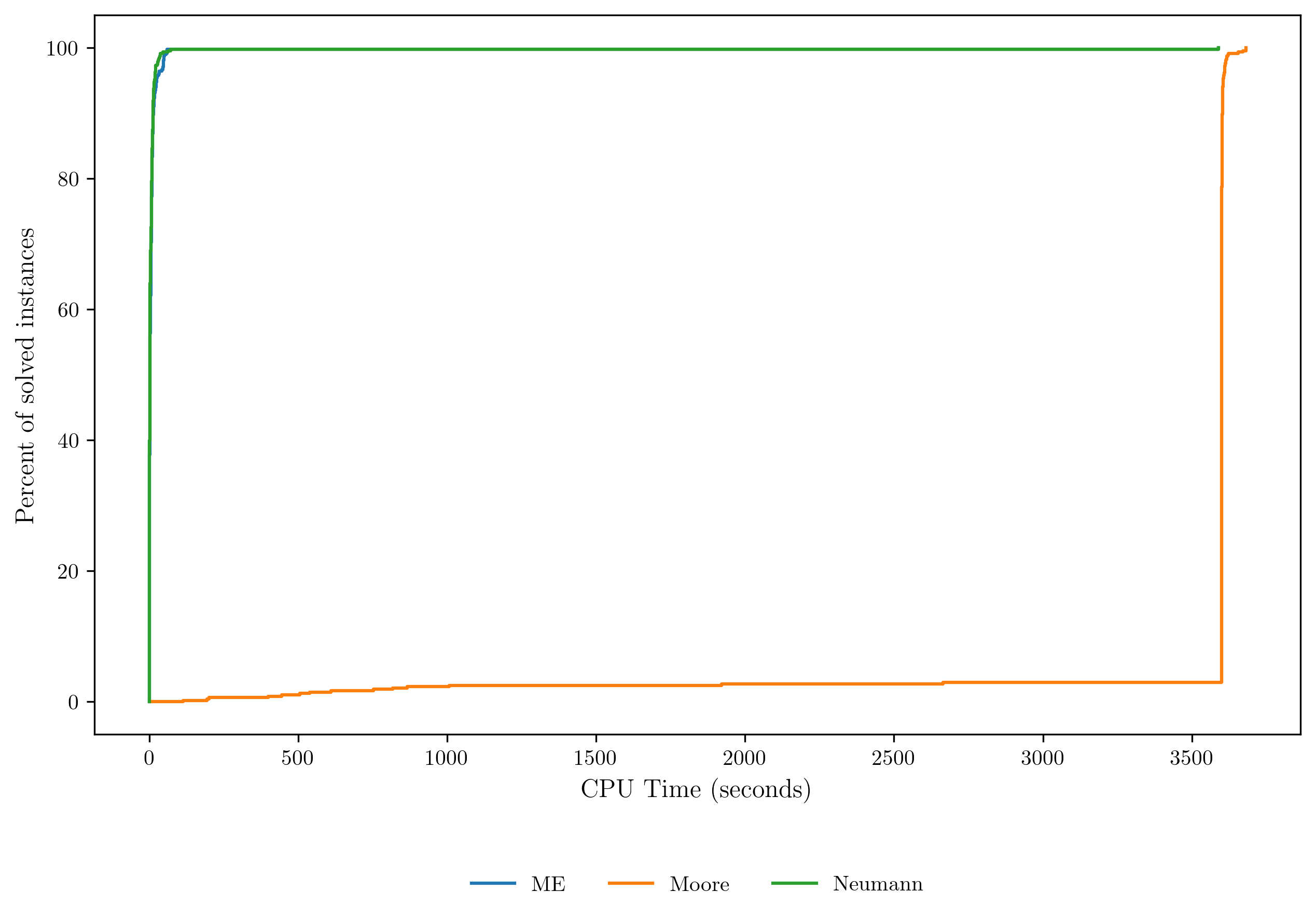}
\caption{Performance profile of \texttt{HPM} for large instances.\label{fig:pp_large}}
\end{figure}

In Figure \ref{fig:MIPGAP_large} we summarize the MIPGaps for the unsolved instances with \texttt{HPM}. As mentioned above, all of them belong to the measure \emph{Moore}, since all the others had MIPGAP=0\%. We distinguished by size $n$ (left) and by type of instances (right). The sizes of the instances seem to affect, although the increase in the MIPGAP is quite small. The type of instances that seem to be more challenging is the \texttt{bin}, being those with density $50\%$ slightly more difficult.
\begin{figure}[h]
\centering\includegraphics[width=0.45\textwidth]{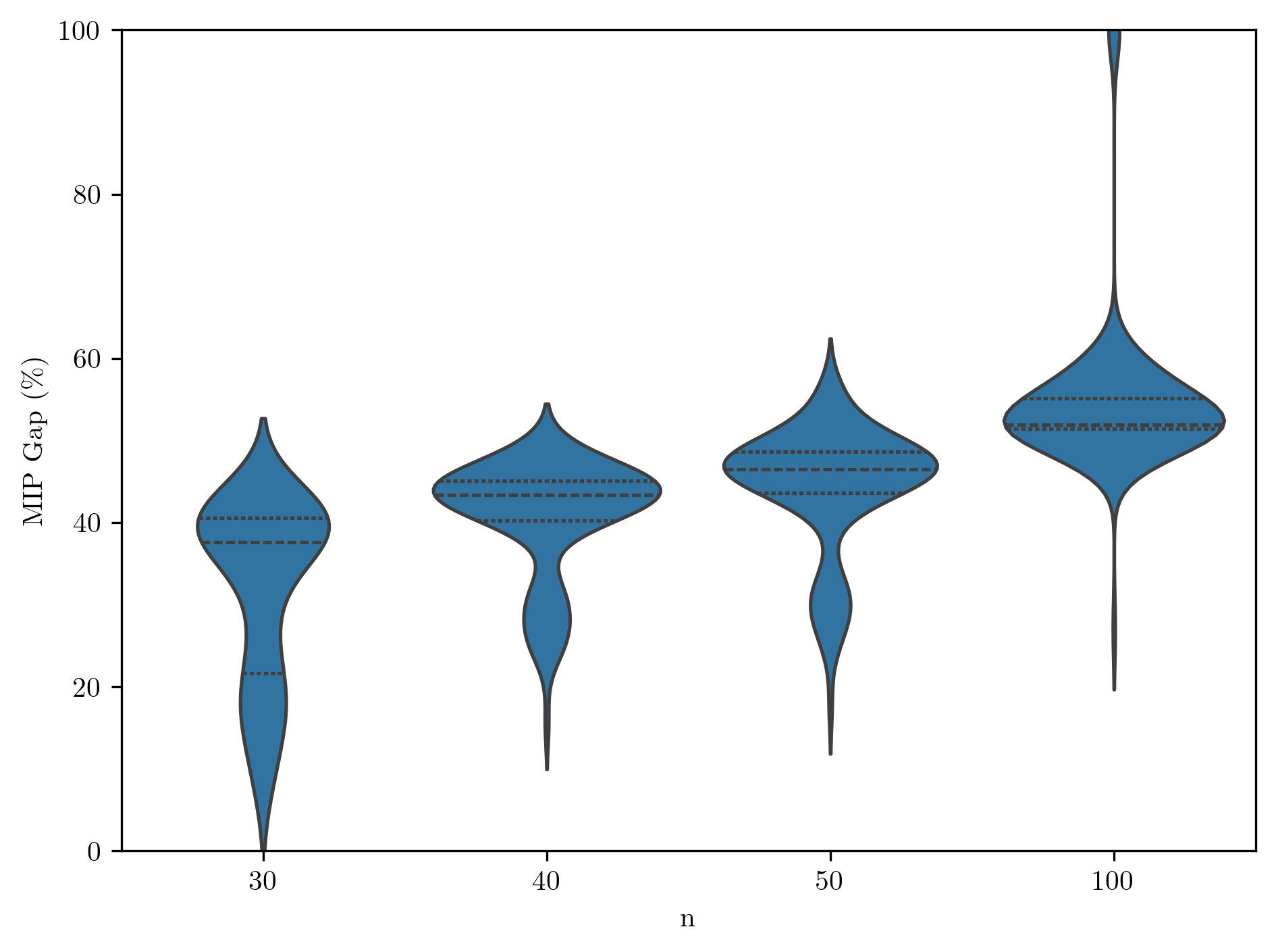}~\includegraphics[width=0.45\textwidth]{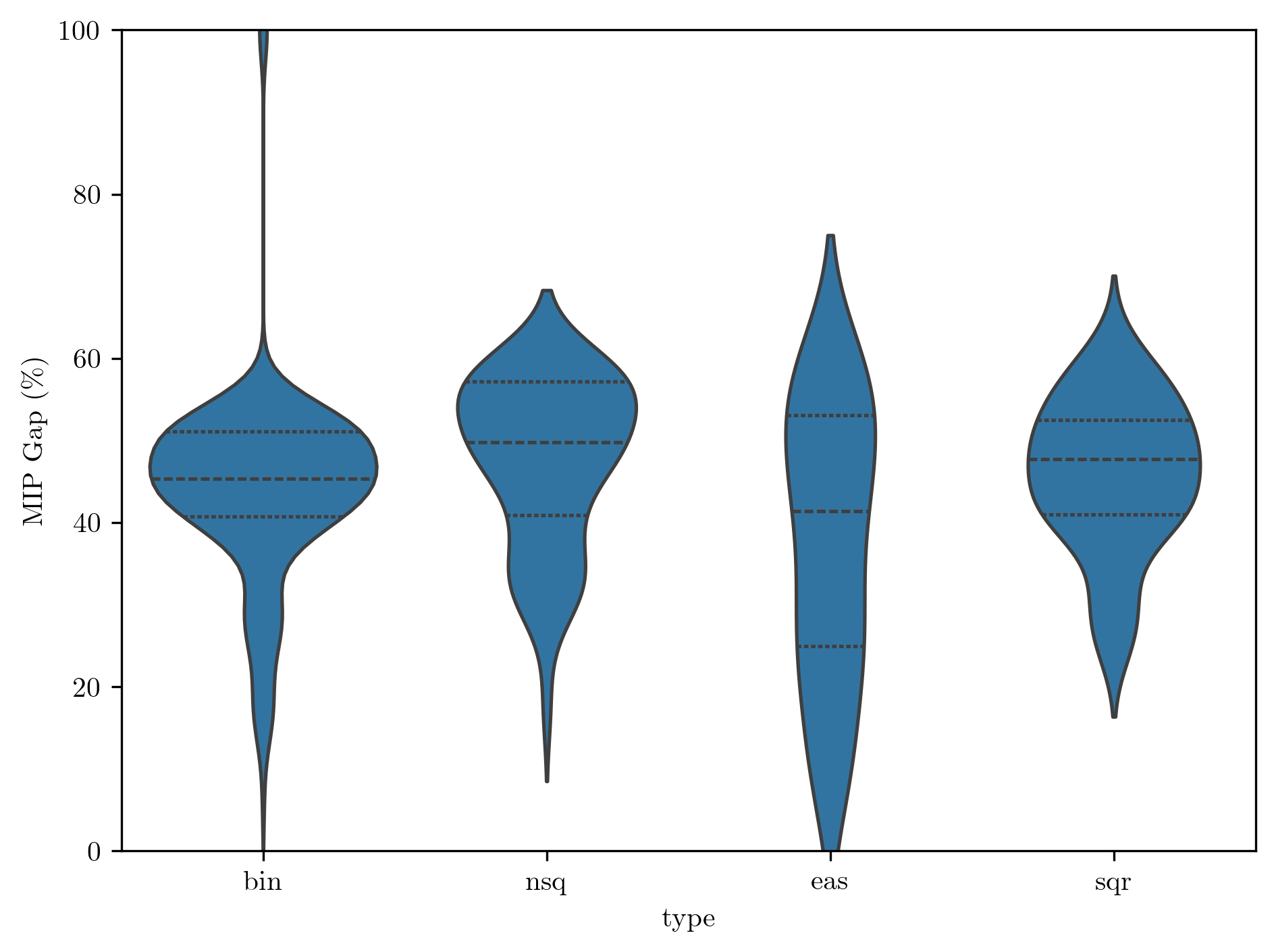}
\caption{MIPGaps of \texttt{HPM} for large instances by size ($n$) and type of instances.\label{fig:MIPGAP_large}}
\end{figure}

In Table \ref{t:tab_devs} we show the main advantage of applying a seriation method to a matrix in terms of the different measures. For each instance, once the optimal seriation is computed, we calculate the relative deviation of the obtained measure of the transformed matrix with respect to that measure on the original matrix. We report the average deviations with respect to the von Neumann (\texttt{Dev\_N}), the Moore (\texttt{Dev\_Mo}), the ME (\texttt{Dev\_ME}) and the homogeneity measure (\texttt{Dev\_Hom}).

\begin{table}[htbp]
\centering
\small
\begin{tabular}{|l|l|l|l|l|l|l|}
\hline
$n$ & $m$ & \texttt{type} & \texttt{Dev\_N} & \texttt{Dev\_Mo} & \texttt{Dev\_ME} & \texttt{Dev\_Hom} \\
\hline
{10} & {10} & \texttt{bin} & 29.11\% & 26.78\% & 26.71\% & 28.20\% \\
 &  & \texttt{eas} & 58.72\% & 50.79\% & 35.47\% & 58.94\% \\
 &  & \texttt{sqr} & 47.04\% & 40.97\% & 18.63\% & 47.73\% \\
\hline
{20} & 10 & \texttt{bin} & 36.55\% & 32.14\% & 31.38\% & 36.05\% \\
 &  & \texttt{nsq} & 51.81\% & 44.61\% & 15.67\% & 51.66\% \\
 & \multirow{3}{*}{20} & \texttt{bin} & 25.75\% & 22.13\% & 21.98\% & 25.35\% \\
 &  & \texttt{eas} & 73.88\% & 69.37\% & 27.48\% & 73.79\% \\
 &  & \texttt{sqr} & 57.50\% & 52.24\% & 16.08\% & 57.49\% \\
\hline
{30} & 10 & \texttt{bin} & 37.65\% & 31.76\% & 30.04\% & 37.38\% \\
 &  & \texttt{nsq} & 57.93\% & 50.64\% & 14.66\% & 57.90\% \\
 & 20 & \texttt{bin} & 34.44\% & 30.17\% & 29.23\% & 33.98\% \\
 &  & \texttt{nsq} & 65.43\% & 60.34\% & 15.02\% & 65.46\% \\
 & {30} & \texttt{bin} & 26.28\% & 22.38\% & 24.41\% & 26.09\% \\
 &  & \texttt{eas} & 81.87\% & 79.33\% & 26.65\% & 81.91\% \\
 &  & \texttt{sqr} & 65.33\% & 61.02\% & 15.62\% & 65.13\% \\
\hline
{40} & 10 & \texttt{bin} & 36.48\% & 30.20\% & 30.64\% & 36.40\% \\
 &  & \texttt{nsq} & 58.25\% & 52.25\% & 12.11\% & 58.14\% \\
 & 20 & \texttt{bin} & 32.00\% & 27.29\% & 27.60\% & 31.57\% \\
 &  & \texttt{nsq} & 66.10\% & 60.46\% & 14.93\% & 66.08\% \\
 & 30 & \texttt{bin} & 31.29\% & 26.67\% & 27.70\% & 31.05\% \\
 &  & \texttt{nsq} & 69.27\% & 64.83\% & 15.05\% & 69.30\% \\
 & {40} & \texttt{bin} & 24.33\% & 20.22\% & 22.90\% & 24.14\% \\
 &  & \texttt{eas} & 87.33\% & 85.22\% & 30.43\% & 87.31\% \\
 &  & \texttt{sqr} & 72.16\% & 67.58\% & 17.11\% & 72.08\% \\
\hline
{50} & 10 & \texttt{bin} & 38.37\% & 31.22\% & 30.60\% & 38.29\% \\
 &  & \texttt{nsq} & 61.40\% & 53.18\% & 14.23\% & 61.21\% \\
 & 20 & \texttt{bin} & 31.64\% & 26.32\% & 27.45\% & 31.34\% \\
 &  & \texttt{nsq} & 69.76\% & 64.87\% & 15.42\% & 69.77\% \\
 & 30 & \texttt{bin} & 29.72\% & 25.29\% & 26.16\% & 29.50\% \\
 &  & \texttt{nsq} & 72.95\% & 68.27\% & 16.15\% & 72.94\% \\
 & 40 & \texttt{bin} & 29.18\% & 24.27\% & 26.30\% & 28.91\% \\
 &  & \texttt{nsq} & 74.36\% & 70.25\% & 15.86\% & 74.34\% \\
 & 50 & \texttt{bin} & 22.55\% & 18.13\% & 21.41\% & 22.35\% \\
 &  & \texttt{eas} & 89.89\% & 88.06\% & 29.90\% & 89.89\% \\
 &  & \texttt{sqr} & 75.48\% & 71.67\% & 15.60\% & 75.44\% \\\hline
{100} & 10 & \texttt{bin} & 38.69\% & 29.89\% & 31.30\% & 39.00\% \\
 &  & \texttt{nsq} & 65.07\% & 56.74\% & 14.75\% & 65.44\% \\
 & 20 & \texttt{bin} & 31.96\% & 24.95\% & 27.53\% & 31.86\% \\
 &  & \texttt{nsq} & 70.49\% & 64.88\% & 14.76\% & 70.50\% \\
 & 30 & \texttt{bin} & 28.35\% & 22.76\% & 25.21\% & 28.29\% \\
 &  & \texttt{nsq} & 73.92\% & 69.15\% & 15.19\% & 73.93\% \\
 & 40 & \texttt{bin} & 27.15\% & 21.55\% & 24.89\% & 27.03\% \\
 &  & \texttt{nsq} & 76.86\% & 72.98\% & 15.88\% & 76.81\% \\
 &  & \texttt{bin} & 24.74\% & 19.41\% & 23.19\% & 24.63\% \\
 &  & \texttt{nsq} & 78.50\% & 74.88\% & 15.62\% & 78.51\% \\
 & \multirow{3}{*}{100} & \texttt{bin} & 16.64\% & 11.76\% & 17.14\% & 16.61\% \\
 &  & \texttt{eas} & 94.18\% & 93.22\% & 29.38\% & 94.19\% \\
 &  & \texttt{sqr} & 83.08\% & 80.00\% & 17.28\% & 83.07\% \\
\hline\multicolumn{3}{|c|}{\bf Total} & 42.49\% & 37.72\% & 23.77\% & 42.33\% \\
\hline
\end{tabular}
\caption{Average improvements of each of the measures comparing the original matrices and the transformed matrix after applying our seriation methods.}
\label{t:tab_devs}
\end{table}

For all these instances, we run the approaches \texttt{BEA} and \texttt{BEA\_TSP} available in the package \texttt{seriation}~\citep{hahsler2008getting} to compute the transformed matrices. As already mentioned, the Bond Energy Algorithm (BEA) is a  seriation method that maximizes the “bond energy”, which encourages similar elements to be placed close to each other. More concretely, the \textit{bond energy} objective for a given permutation of rows $\pi$ is:
$$
\sum_{i\in N} \sum_{k \in M} a_{\pi(i)k}a_{\pi(i+1)k}.
$$
The BEA approach for separately sorting rows and columns consists of alternatively sorting rows and columns until convergence. The greedy approach that is implemented in \texttt{seriation} to find the sorting of the rows/columns starts with a small submatrix, and iteratively inserts the next row/column into the position that maximizes the incremental increase in bond energy, and the process is repeated until all elements are placed.

The \texttt{BEA\_TSP} computes the shortest cycle between rows of the matrix, where the cost between two rows is the distance 
between them. The problem is solved using Concorde in R.

We compared the evaluation of the measures ME,  von Neumann and Moore using our approaches with those obtained with BEA, BEA\_TSP and Heatmap. In all the cases, the measure we obtained was better than or equal to the one obtained with the methods implemented in \texttt{seriation}. Thus, for each of these measures, we compute the deviation of the best solution (for all the methods in R). In Figure \ref{fig:devs} we report the violin plot displaying the distribution of deviation values for three different 
measures with the  when compared to the HPM baseline. As one can observe, the distribution for the von Neumann measure shows a relatively tight and low deviation distribution, with the bulk of the values concentrated below $5\%$. The distribution of the deviations for the Moore measure exhibits the largest spread and highest median deviation, indicating that the methods implemented in \texttt{seriation} consistently deviates more from HPM than the others. The shape also reveals a thicker middle region, suggesting more variability and less precision.
Finally, the deviations for the ME measure seem to have the smallest deviation values and the tightest spreads, with values strongly concentrated around a lower median. This suggests that ME measure aligns most closely with the BEA, BEA\_TSP, and Heatmaps.

\begin{figure}[h]
\centering\includegraphics[width=0.5\textwidth]{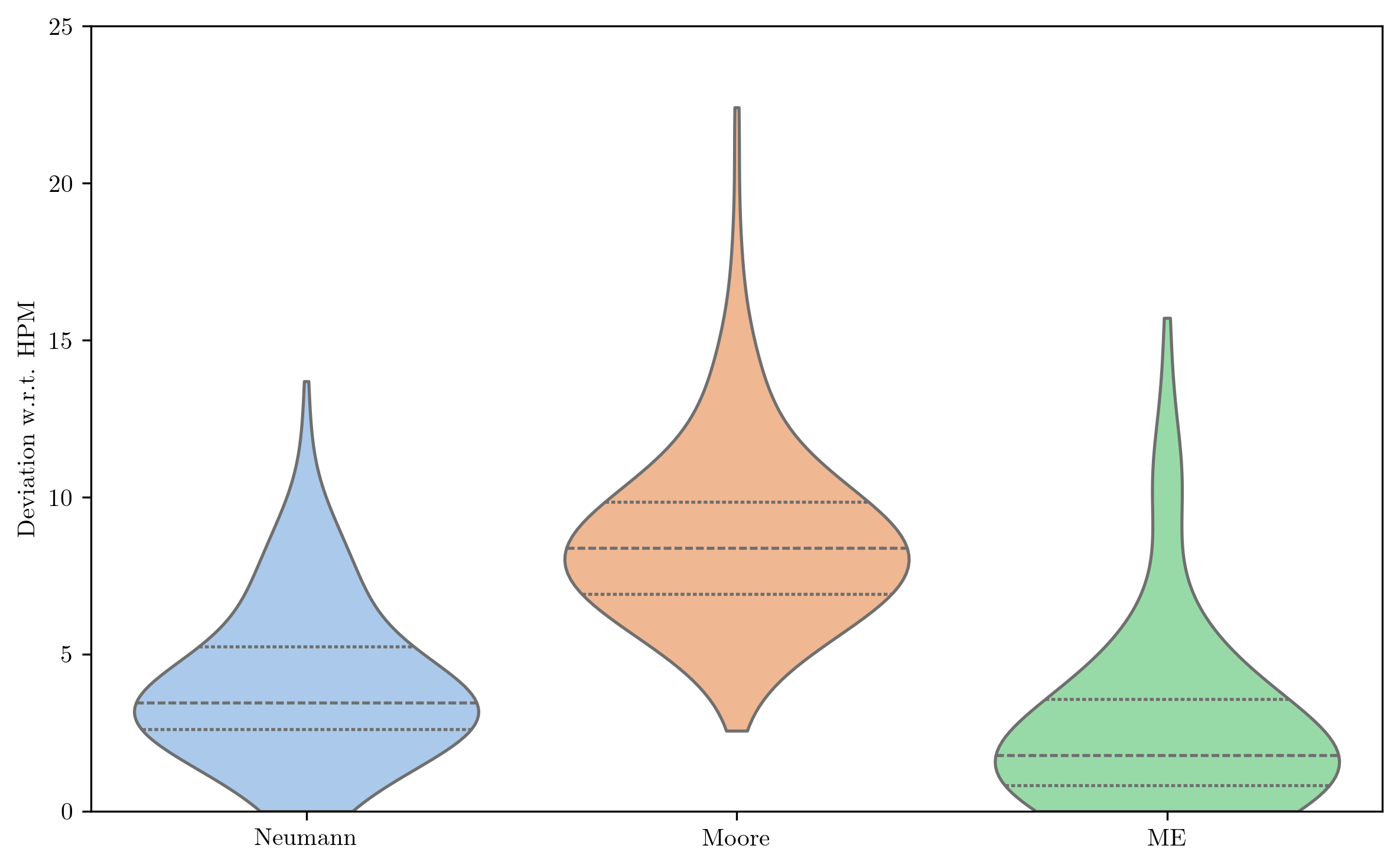}
\caption{Deviations with respect to our approaches\label{fig:devs}}
\end{figure}

\subsection{Real datasets}

\vspace*{0.2cm}

\noindent{\bf Datasets Seriation package}

\vspace*{0.2cm}

We also test our approaches on the datasets provided in the package \texttt{seriation} \citep{hahsler2008getting}, namely \texttt{SupremeCourt}, \texttt{Irish}, \texttt{Munsingen}, \texttt{Townships}, and \texttt{Zoo}, whose descriptions are shown in Table \ref{table_dataR}.

\begin{table}[h!]
{\small\begin{tabular}{p{1.9cm}p{1.5cm}p{9.3cm}}
\texttt{dataset} & Size & Description\\\hline
Irish & (41,8) & Contingency table of responses from Irish voters on political issues.
\\
Munsingen & (59,70) & Archaeological site data from Munsingen (Switzerland), showing the presence/absence or frequencies of grave goods across different graves.\\
Townships & (16,9) & Distance matrix between South African townships based on demographic, social, or economic indicators.\\
Zoo & (101,16) &  Binary dataset describing animal characteristics.\\
Supreme Court & (9,9) &Voting similarity matrix between U.S. Supreme Court justices over a specific period.\\\hline
\end{tabular}}
\caption{Datasets from the \texttt{seriation} package\label{table_dataR}}
\end{table}

For each of the data sets we run the algorithms BEA and BEA-TSP available in the \texttt{seriation} package, as well as our methods \texttt{HPM} for the measures ME (\texttt{HPM-ME}), Moore (\texttt{HPM-Moore}) and von Neumann (\texttt{HPM-VN}). For each of the solutions, we evaluate these three measures and for each data set we compute the deviation of this measure with respect to the best value obtained with all the methods. In Table \ref{tab:deviations} we report the percentage of deviation for all the data sets and all the measures (highlighting in bold face those where the best value was obtained).

\begin{table}[h]
\centering

{\small\begin{tabular}{l l |l l l l l}
dataset & Measure & BEA & BEA\_TSP & \texttt{HPM-ME} & \texttt{HPM-Moore} & \texttt{HPM-VN} \\
\hline
\multirow{3}{*}{Irish} & ME & \textbf{0\%} & \textbf{0\%} & \textbf{0\%} & 3\% & 2\% \\
 & Von Neumann & 50\% & 50\% & 50\% & 1\% & \textbf{0\%} \\
& Moore & 50\% & 50\% & 50\% & \textbf{0\%} & \textbf{0\%} \\
\hline
\multirow{3}{*}{Munsingen} & ME & \textbf{0\%} & \textbf{0\%} & \textbf{0\%}  & \textbf{0\%} & 20\%\\
 & Von Neumann & 10\% & 10\% & \textbf{0\%} & 20\% & \textbf{0\%} \\
 & Moore & 10\% & 10\% & \textbf{0\%} & 20\% & \textbf{0\%} \\
 \hline
\multirow{3}{*}{SupremeCourt} & ME & \textbf{0\%} & \textbf{0\%} & \textbf{0\%} & 5\% & 5\% \\
 & Von Neumann & 21\% & 21\% & 26\% & \textbf{0\%} & \textbf{0\%} \\
 & Moore & 24\% & 22\% & 25\% & \textbf{0\%} & 1\% \\
 \hline
\multirow{3}{*}{Townships} & ME & 10\% & \textbf{0\%} & \textbf{0\%} & \textbf{0\%} & \textbf{0\%} \\
 & Von Neumann & 30\% & 20\% & 20\% & \textbf{0\%} & \textbf{0\%} \\
 & Moore & 30\% & 30\% & 30\% & \textbf{0\%} & \textbf{0\%} \\
 \hline
\multirow{3}{*}{Zoo} & ME & \textbf{0\%} & \textbf{0\%} & \textbf{0\%} & 40\% & 10\% \\
 & Von Neumann & 30\% & 30\% & 30\% & 60\% & \textbf{0\%} \\
 & Moore & 30\% & 30\% & 30\% & \textbf{0\%} & 60\%  \\
\hline
\end{tabular}}
\caption{Deviations with respect to all the measures with the different methods\label{tab:deviations}}
\end{table}

One can observe that \texttt{HPM} always obtain the best value with respect to the measure that is being optimized. 
In some cases, the BEA-based methods obtained a deviation of $50\%$ with respect to this best value. The BEA method is 
a heuristic designed for ME. We found that for the Township dataset, this heuristic solution deviated $10\%$ from the optimal solution obtained with \texttt{HPM-ME}. A similar situation happened with BEA-TSP with the von Neumann measure, since the method behind the TSP solution in this package is not exact. 

Summarizing these results, our proposals are exact approaches that result in more accurate seriations when applied to 
the best known goodness measures under this framework.

\vspace*{0.2cm}

\noindent{\bf Scientific Coauthorship Datasets}

\vspace*{0.2cm}

We extracted 77 scientific publications from the years 1966 to 2024 from the database \url{openalex.com}, in which the phrase \textit{matrix seriation} appears in the title or abstract. Based on this corpus, we constructed a co-authorship matrix where each entry quantifies the number of joint publications between pairs of authors. To explore varying collaboration densities, we filtered the dataset to retain only authors with at least 1, 2, 3, or 5 coauthored publications, resulting in square matrices of sizes $92$, $53$, $40$, and $14$, respectively. These datasets are available in our Github repository \url{github/vblancOR/seriation_mathopt}.

For visualization, we normalized all matrix entries to lie in the interval $[0,1]$, and set all diagonal entries to 1 to represent full self-collaboration.

Figure~\ref{fig:all_originals} presents the original co-authorship matrices for the four scenarios ($n \in \{14, 40, 53, 92\}$), with authors sorted alphabetically. These matrices are shown as grayscale images (black = 1, white = 0). As can be seen, no clear structural patterns emerge from these unordered matrices.

\begin{figure}[h!]
\includegraphics[width=0.5\textwidth]{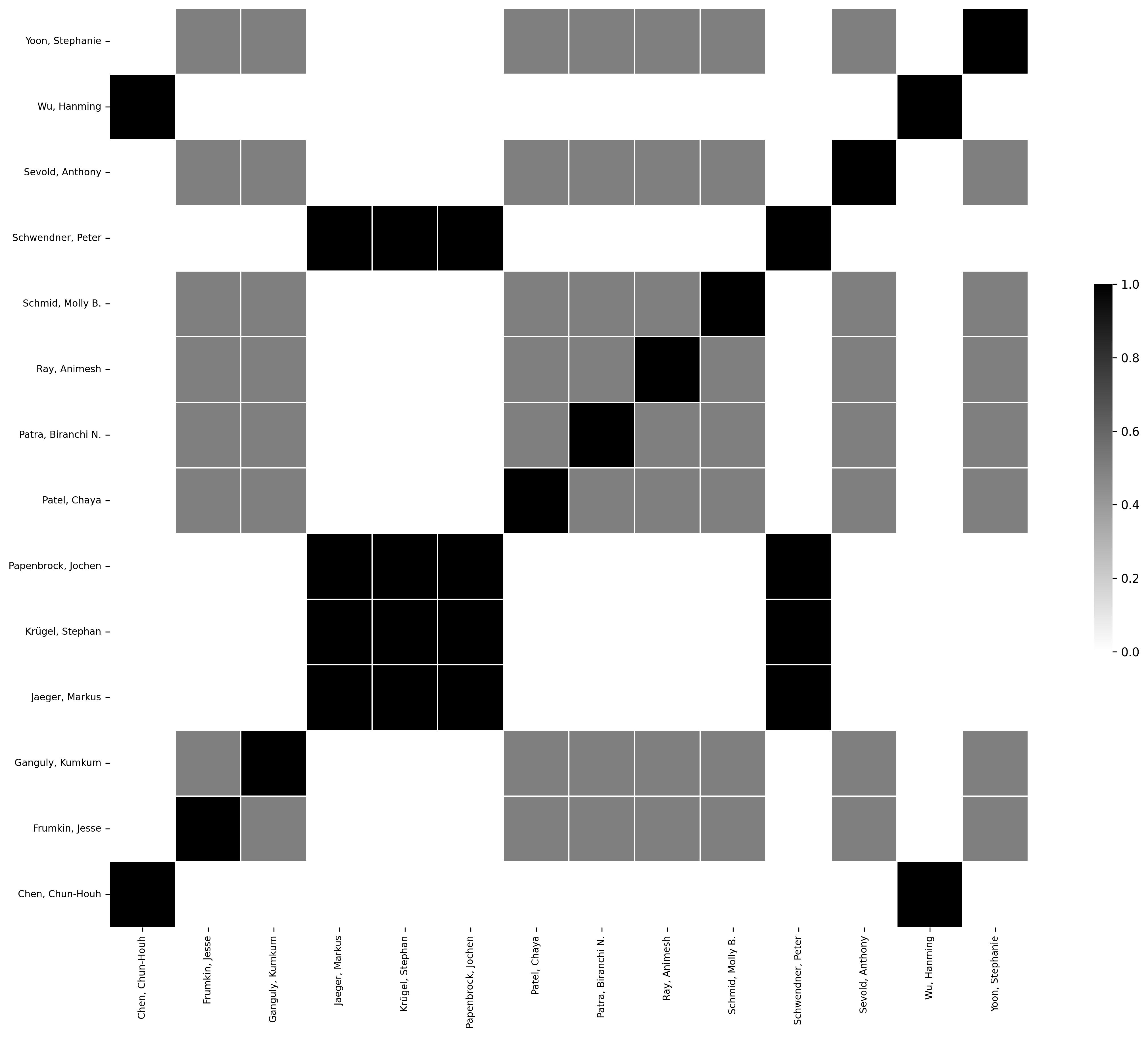}~\includegraphics[width=0.5\textwidth]{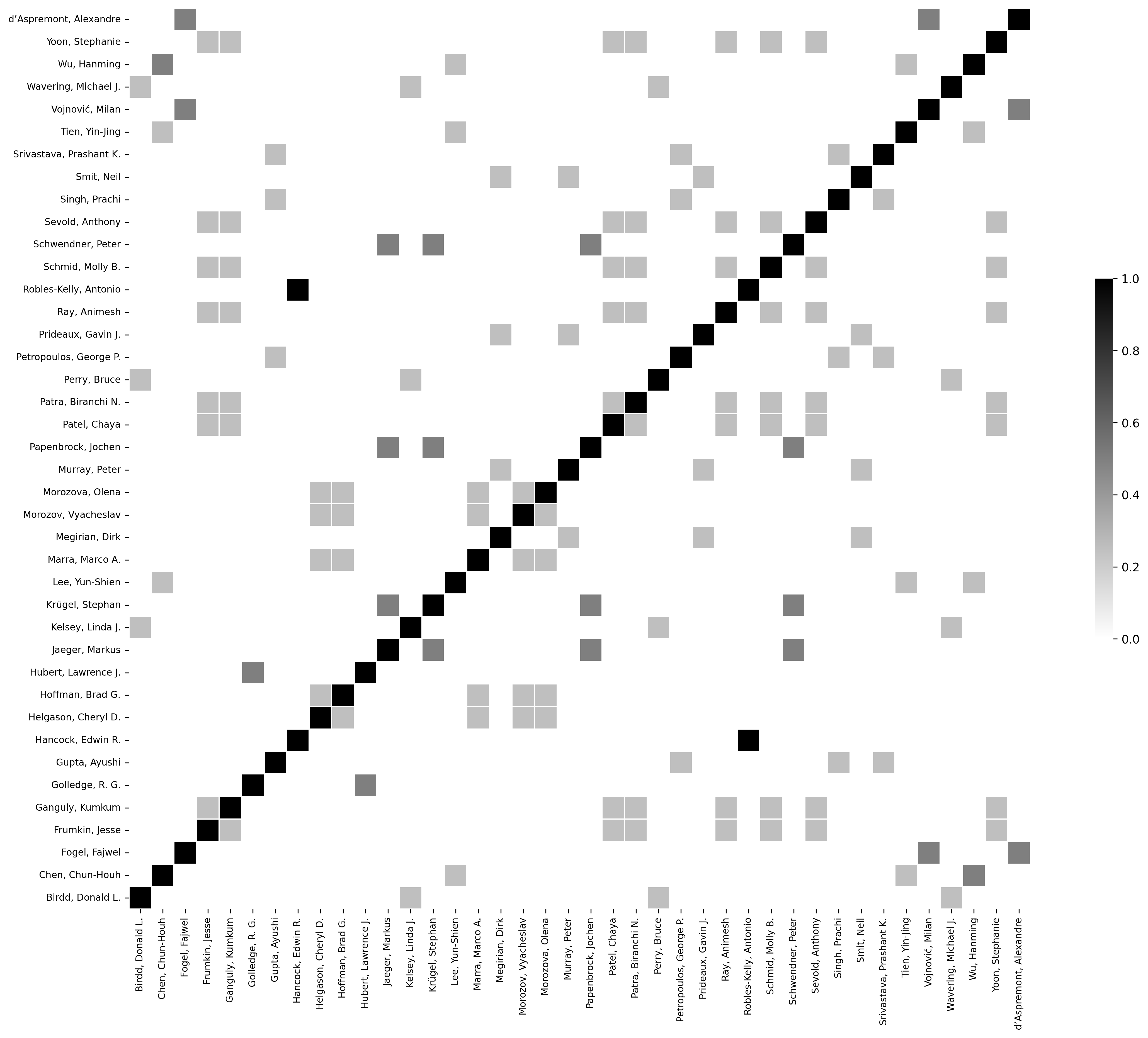} \\
\includegraphics[width=0.5\textwidth]{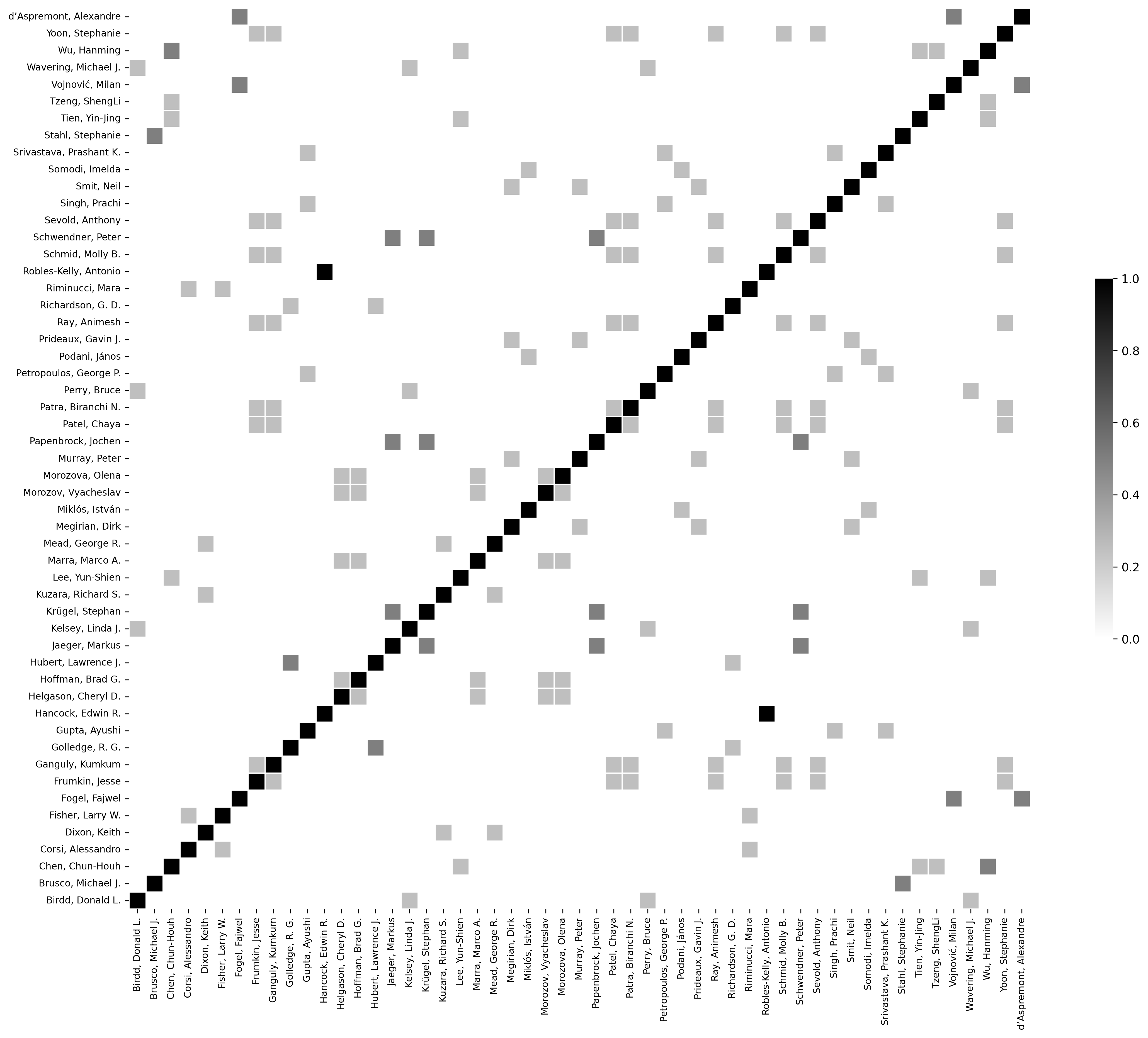}~\includegraphics[width=0.5\textwidth]{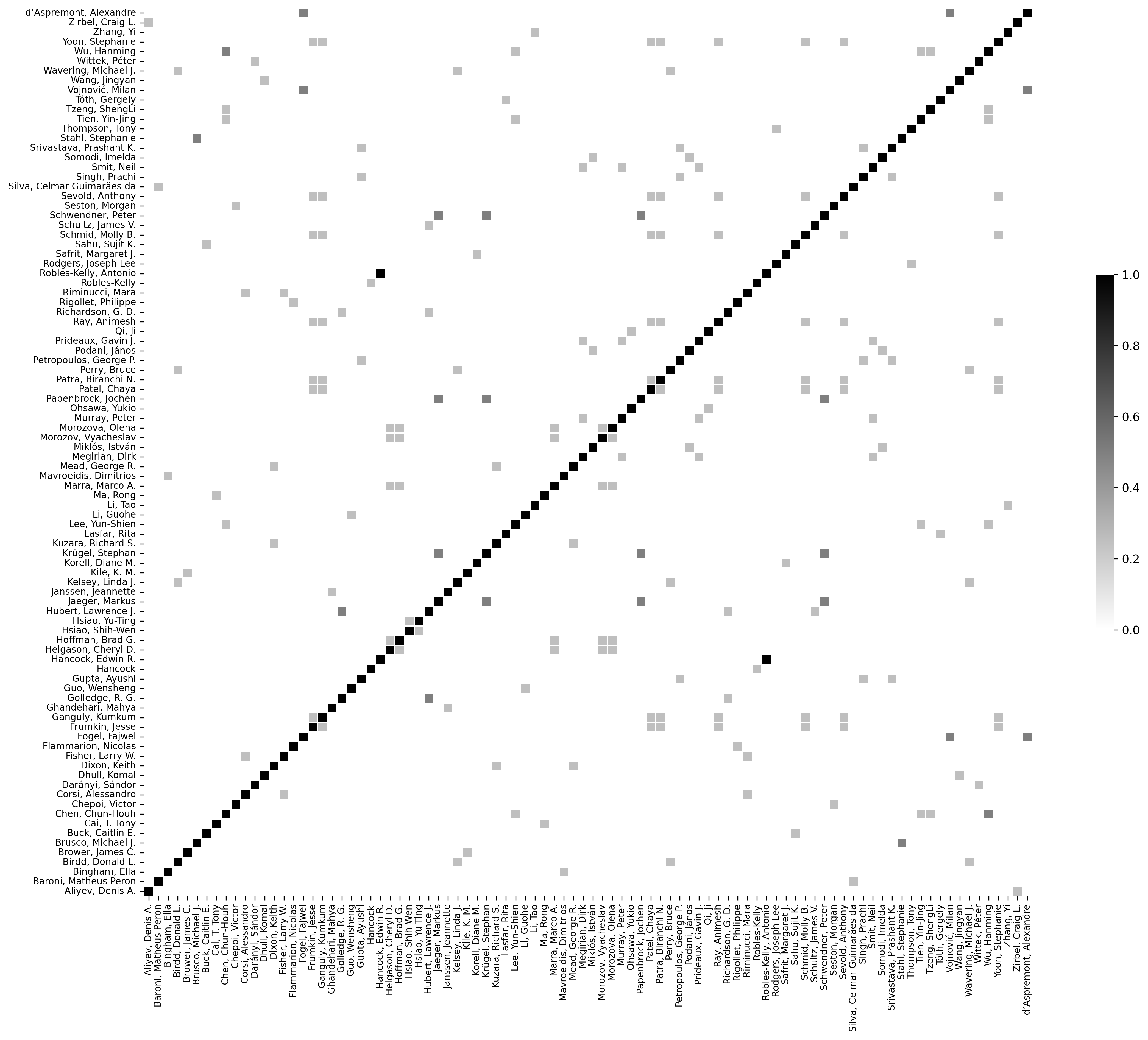}
\caption{Original information for the for scenarios in our case study ($n \in \{14, 40, 53, 92\}$) \label{fig:all_originals}}
\end{figure}

To uncover latent structure, we applied our HPM method, under both coordinated  and uncoordinated settings, using three seriation criteria: \textbf{Moore stress}, \textbf{von Neumann stress}, and the \textbf{Measure of Efficiency (ME)}.

The resulting visualizations are displayed in Figures~\ref{fig:seriation14_sep}--\ref{fig:seriation92_joint}.

For the smallest matrix ($14 \times 14$), coordination leads to nearly identical results across the three measures, differing only in the internal ordering within the three clearly separated clusters (Figure~\ref{fig:seriation14_joint}). These clusters are reported in Table~\ref{table:clusters14}.

Without coordination (Figure~\ref{fig:seriation14_sep}), the same three-cluster structure is evident, though with variation in the intra-cluster ordering.

\begin{figure}[h!]
\includegraphics[width=0.5\textwidth]{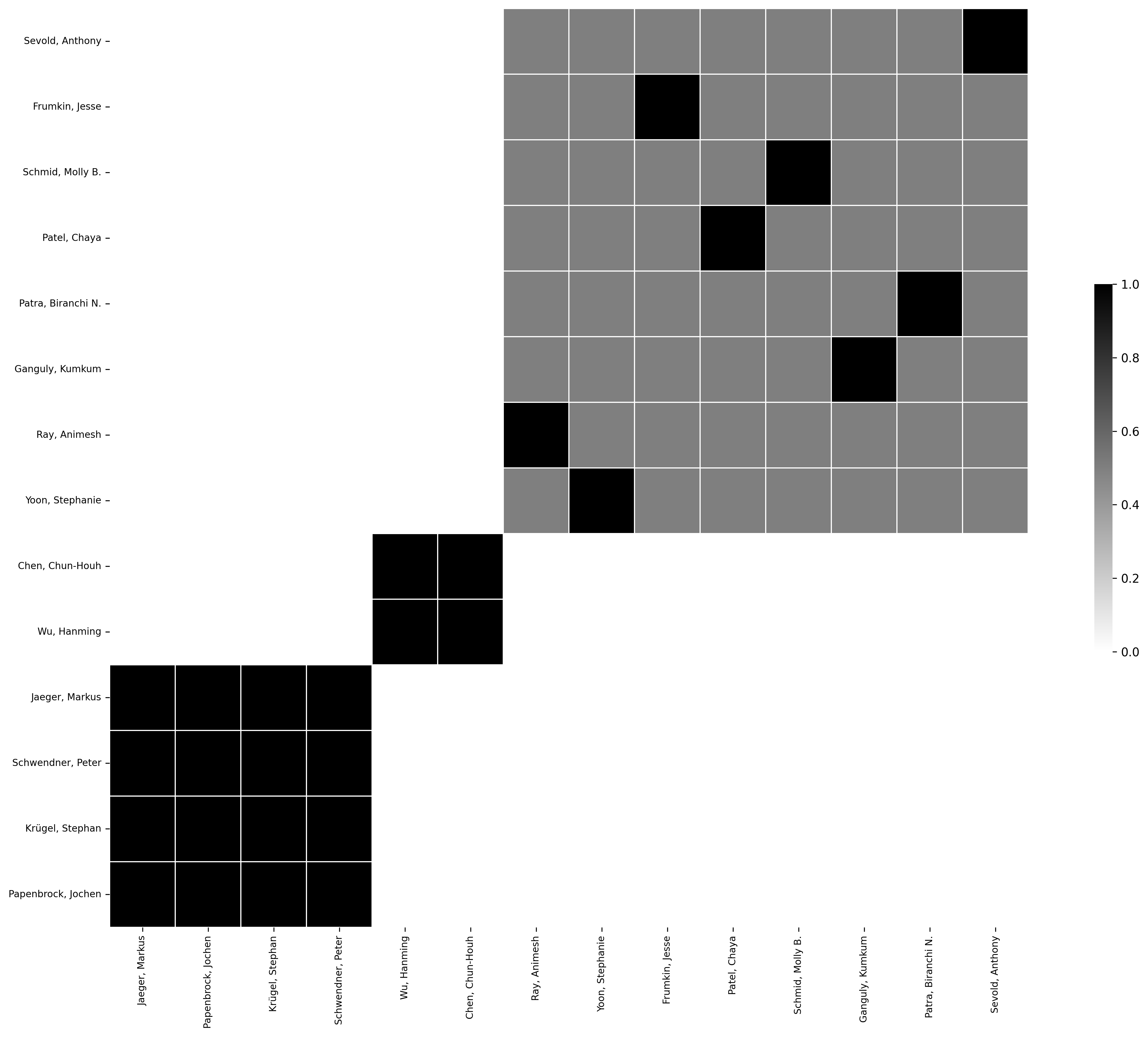}~\includegraphics[width=0.5\textwidth]{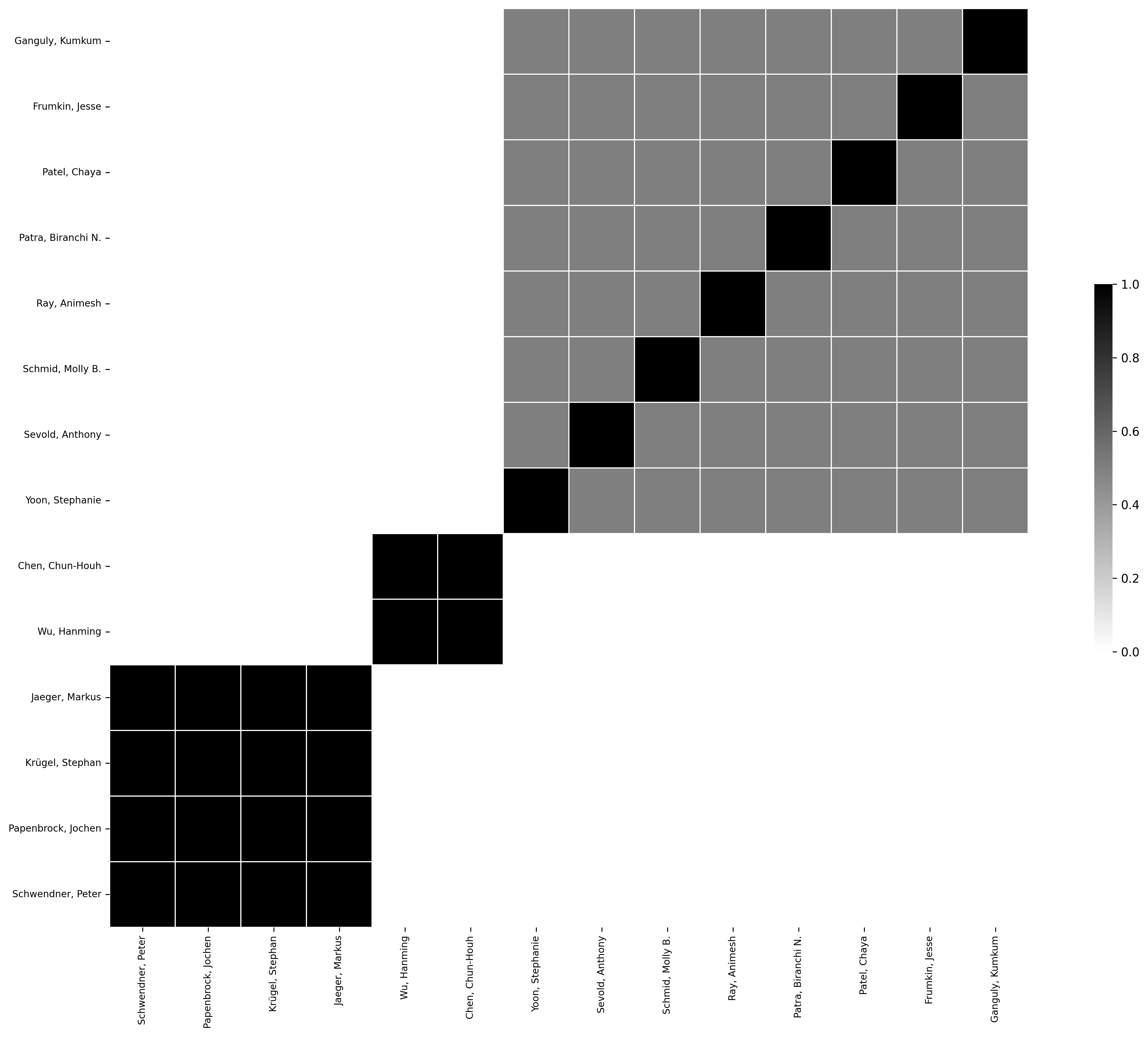}\\
\centerline{\includegraphics[width=0.5\textwidth]{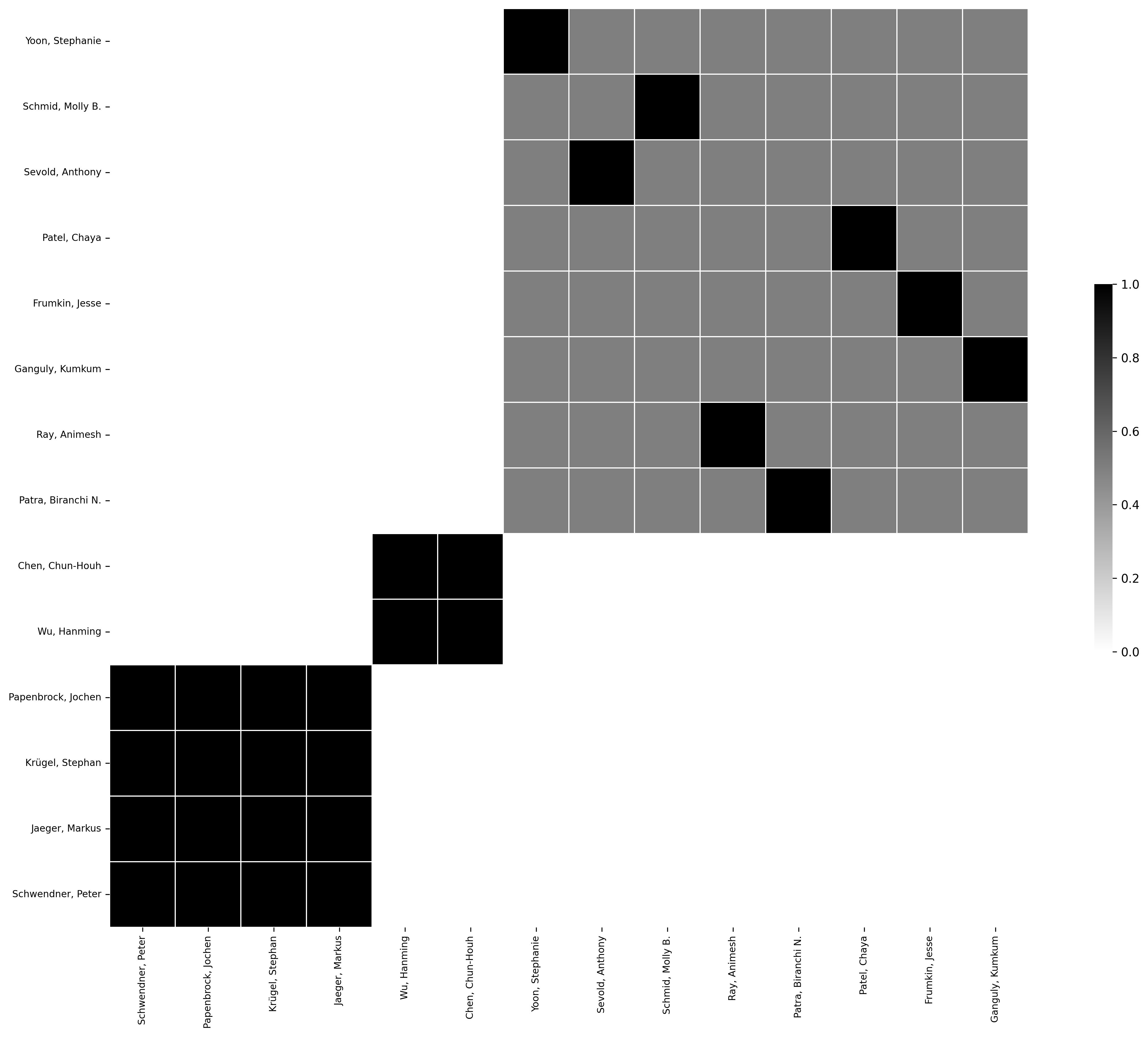}}
\caption{Representation of the optimal solutions for the $14$-seriation authorship instance with non coordinated rows/columns under von Neumann stress measure (left), Moore stress measure (right), and ME measure (down)\label{fig:seriation14_sep}}
\end{figure}

\begin{figure}[h!]
\includegraphics[width=0.33\textwidth]{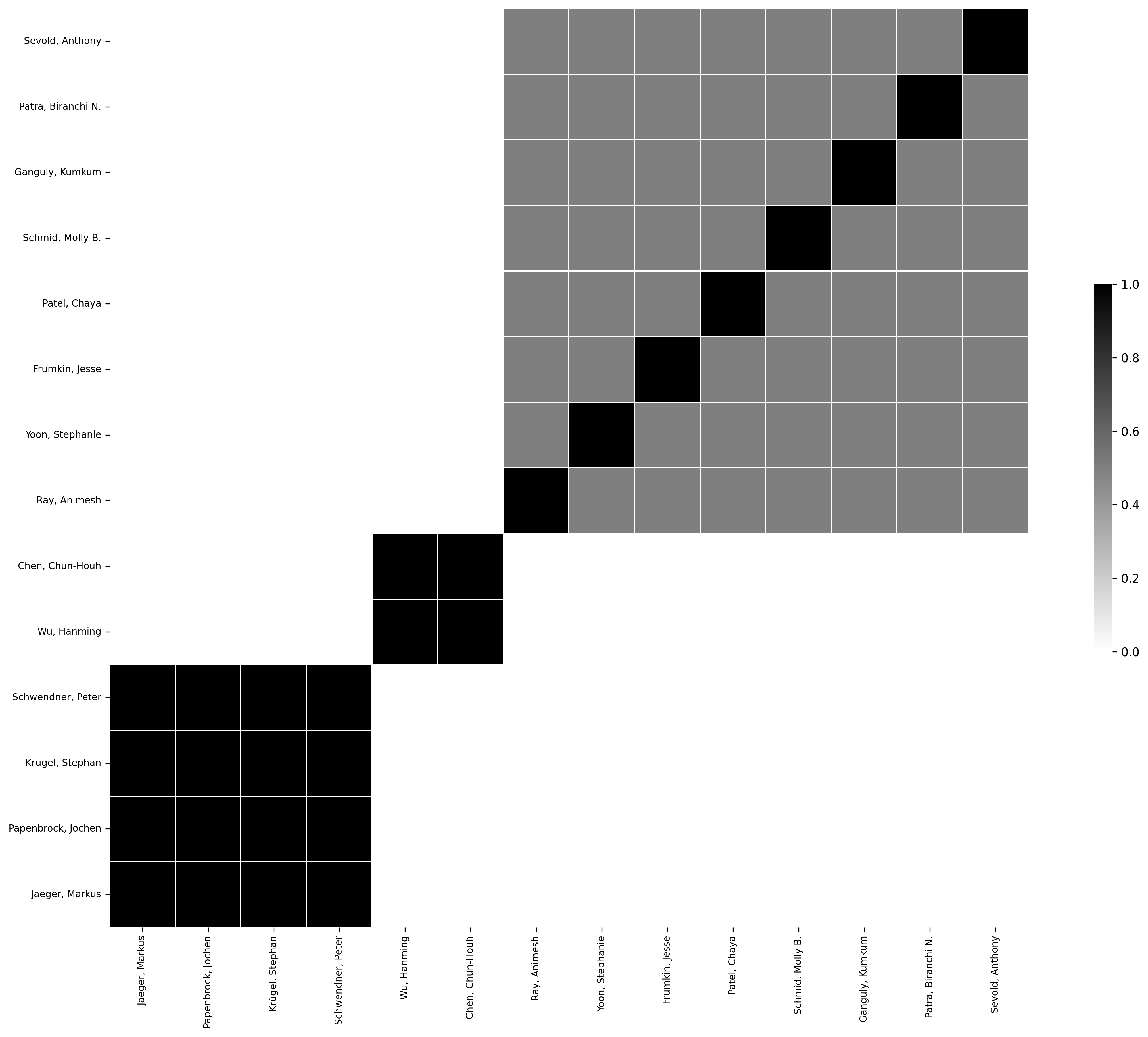}~\includegraphics[width=0.33\textwidth]{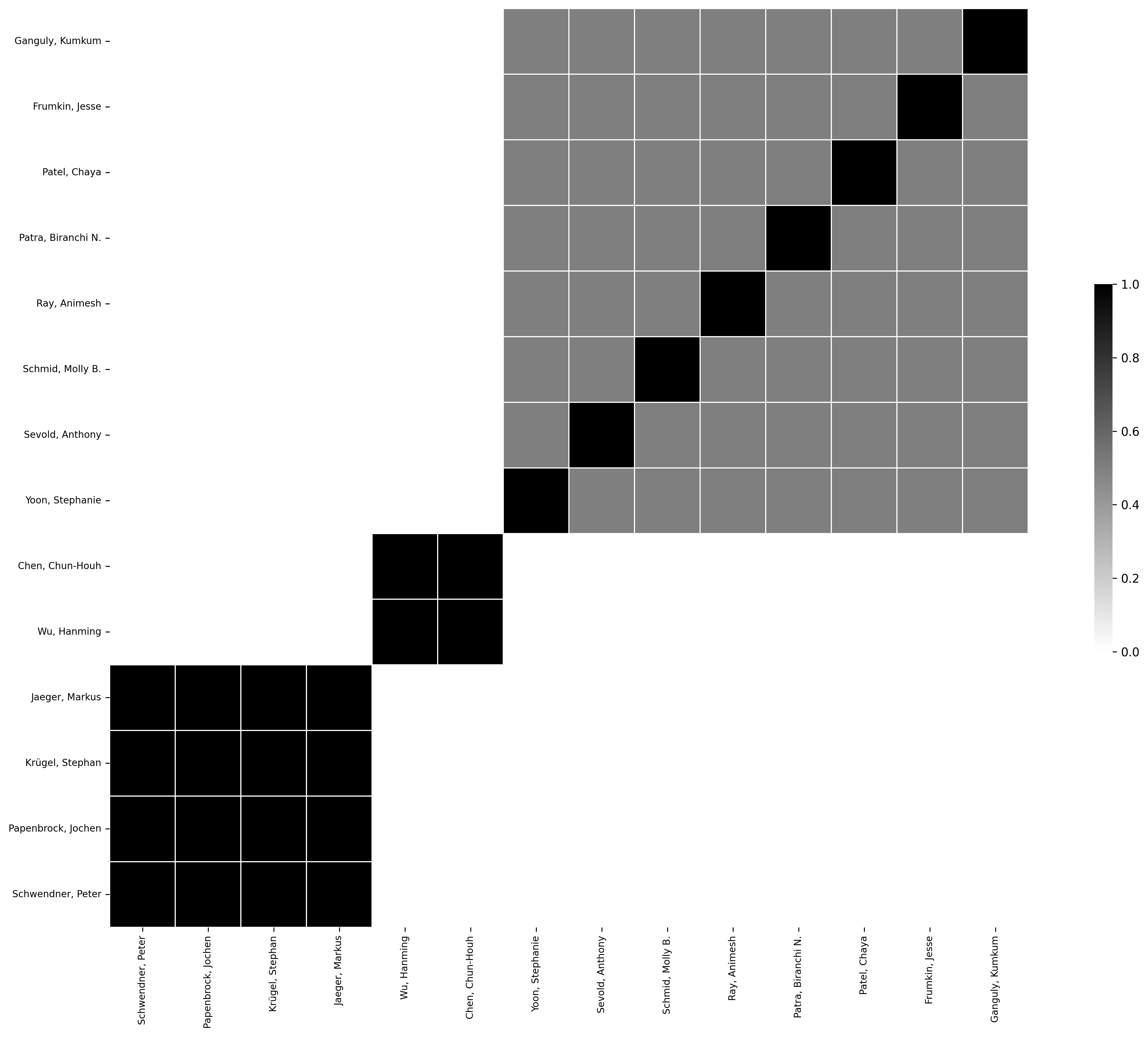}~{\includegraphics[width=0.33\textwidth]{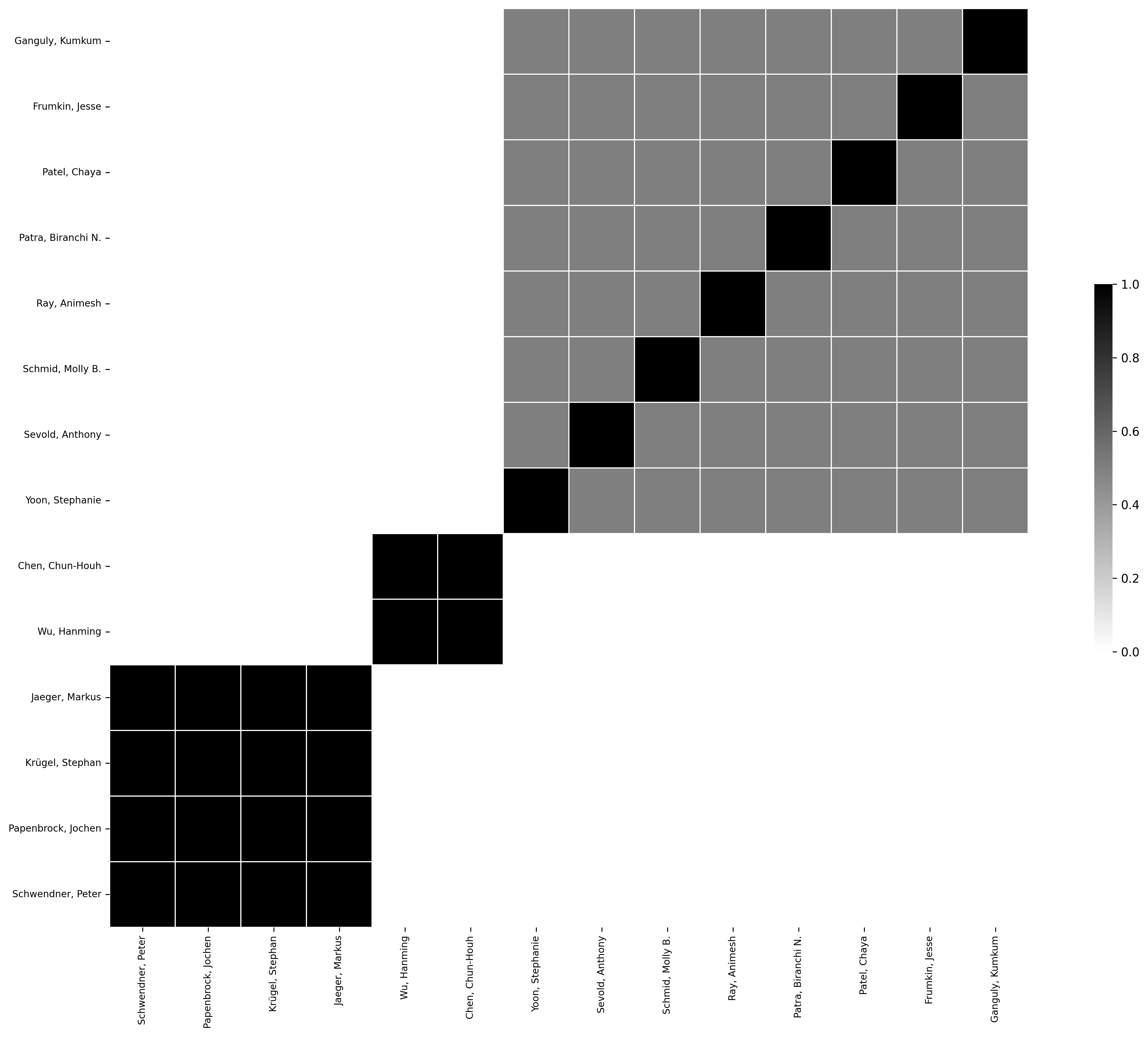}}
\caption{Representation of the optimal solutions for the $14$-seriation authorship instance with coordinated rows/columns under von Neumann stress measure (left), Moore stress measure (right) and ME measure (down)\label{fig:seriation14_joint}}
\end{figure}

For the 40-author dataset, coordinated solutions (Figure~\ref{fig:seriation40_joint}) reveal 10 distinct blocks of non-interacting author groups, including four weakly connected sub-clusters. The identified clusters are listed in Table~\ref{table:clusters40}.

\begin{figure}[h]
\includegraphics[width=0.33\textwidth]{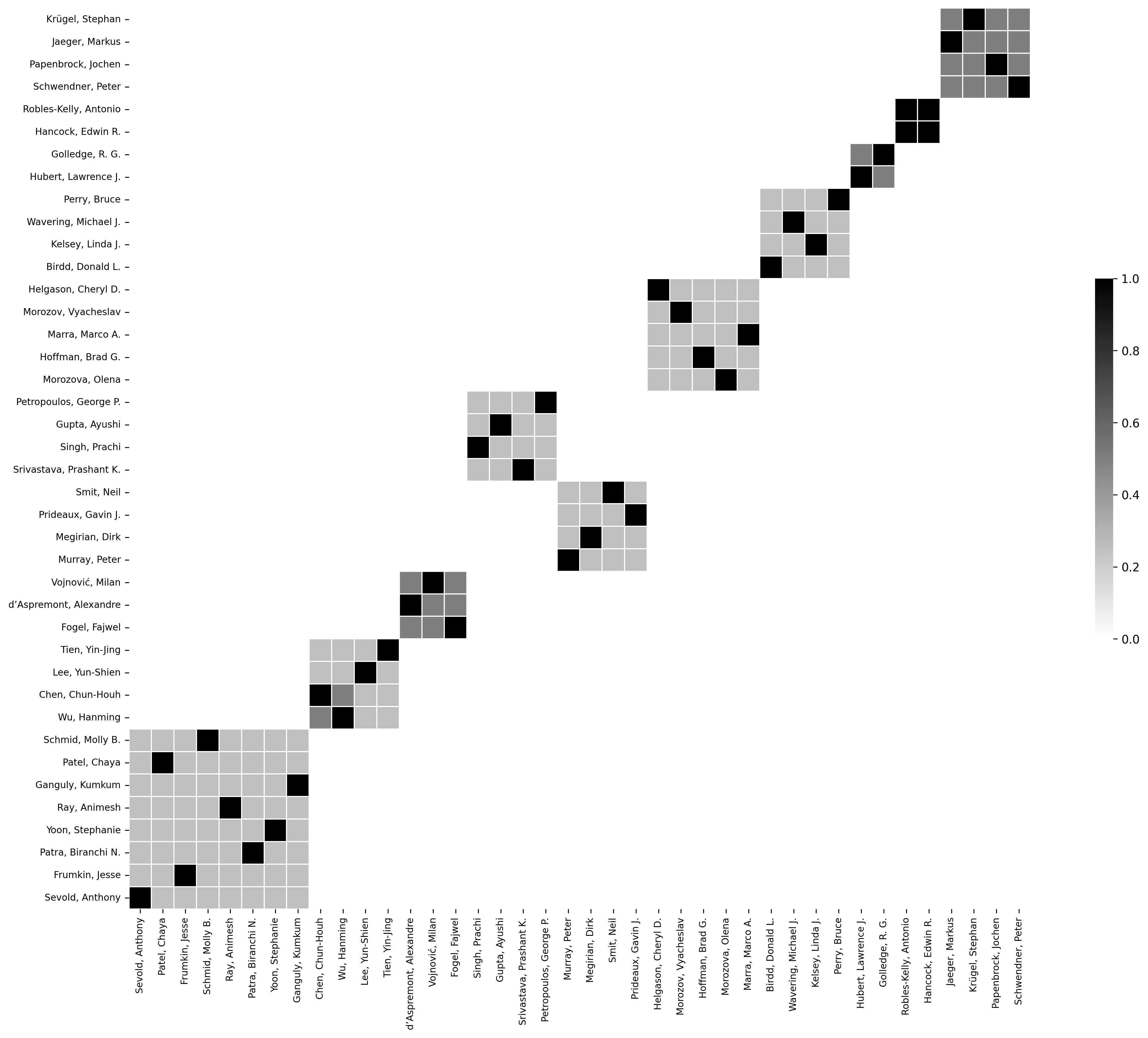}~\includegraphics[width=0.33\textwidth]{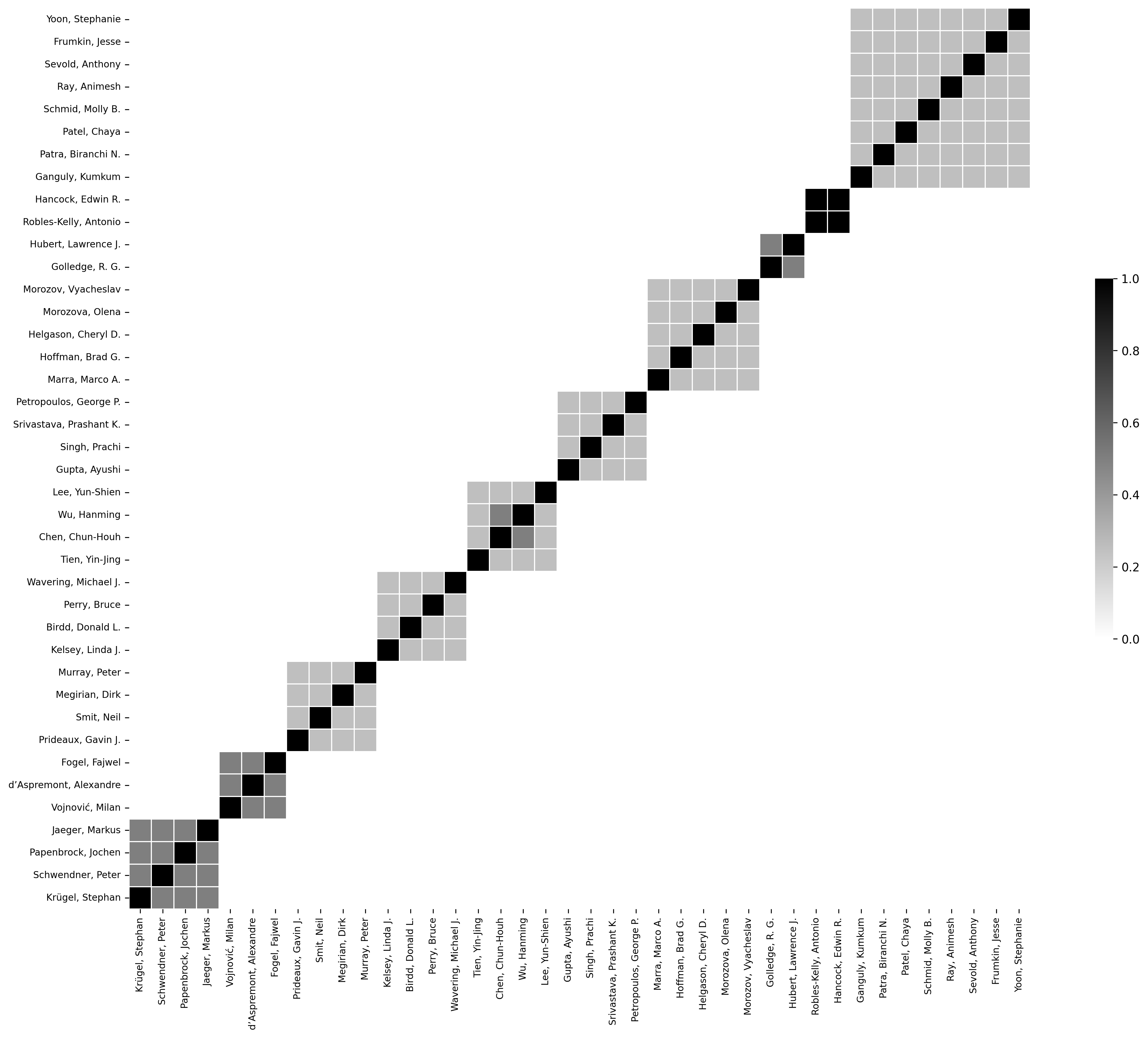}~{\includegraphics[width=0.33\textwidth]{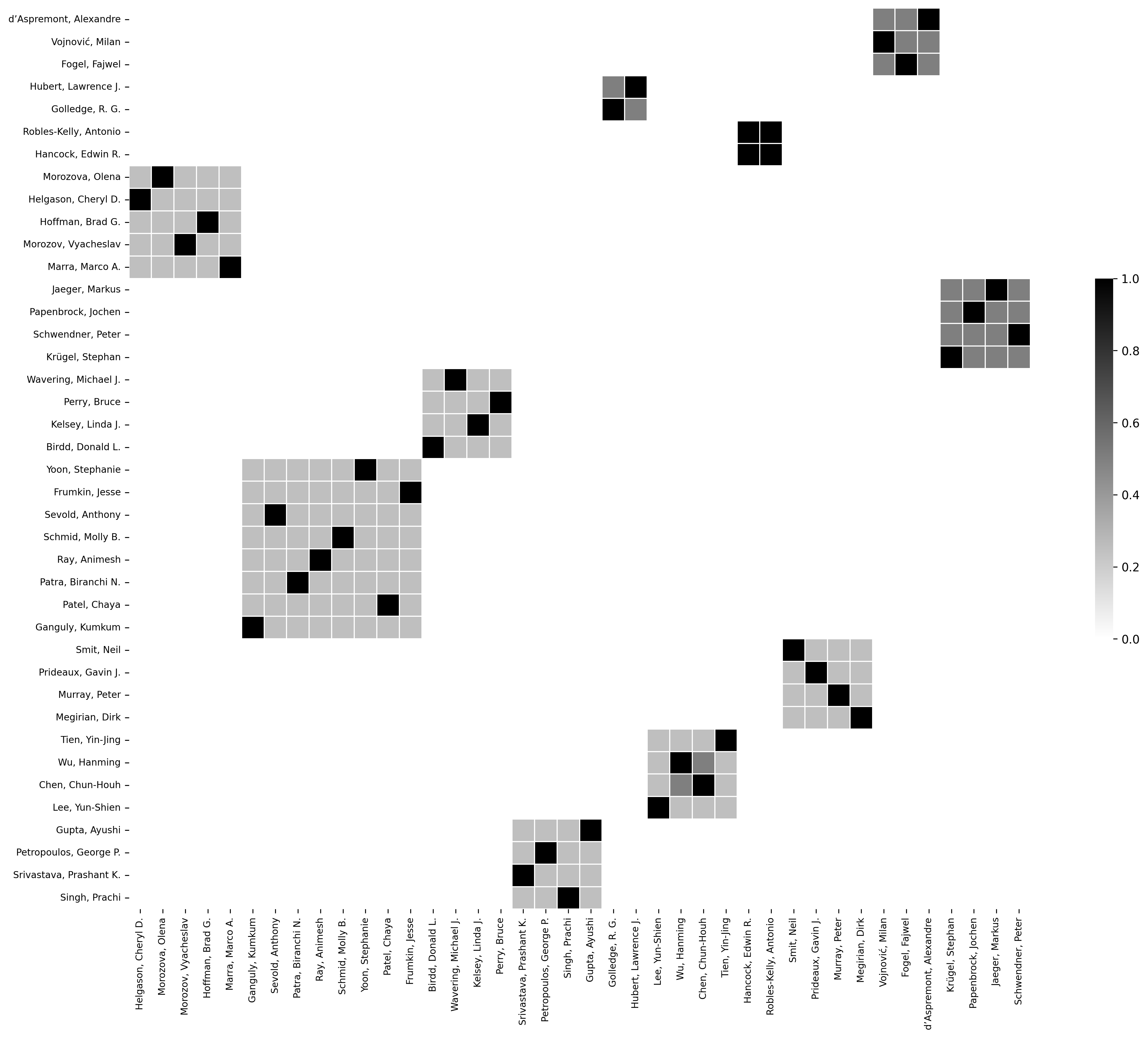}}
\caption{Representation of the optimal solutions for the $40$-seriation authorship instance with non coordinated rows/columns under von Neumann stress measure (left), Moore stress measure (center) and ME measure (right)\label{fig:seriation40_sep}}
\end{figure}

\begin{figure}[h]
\includegraphics[width=0.33\textwidth]{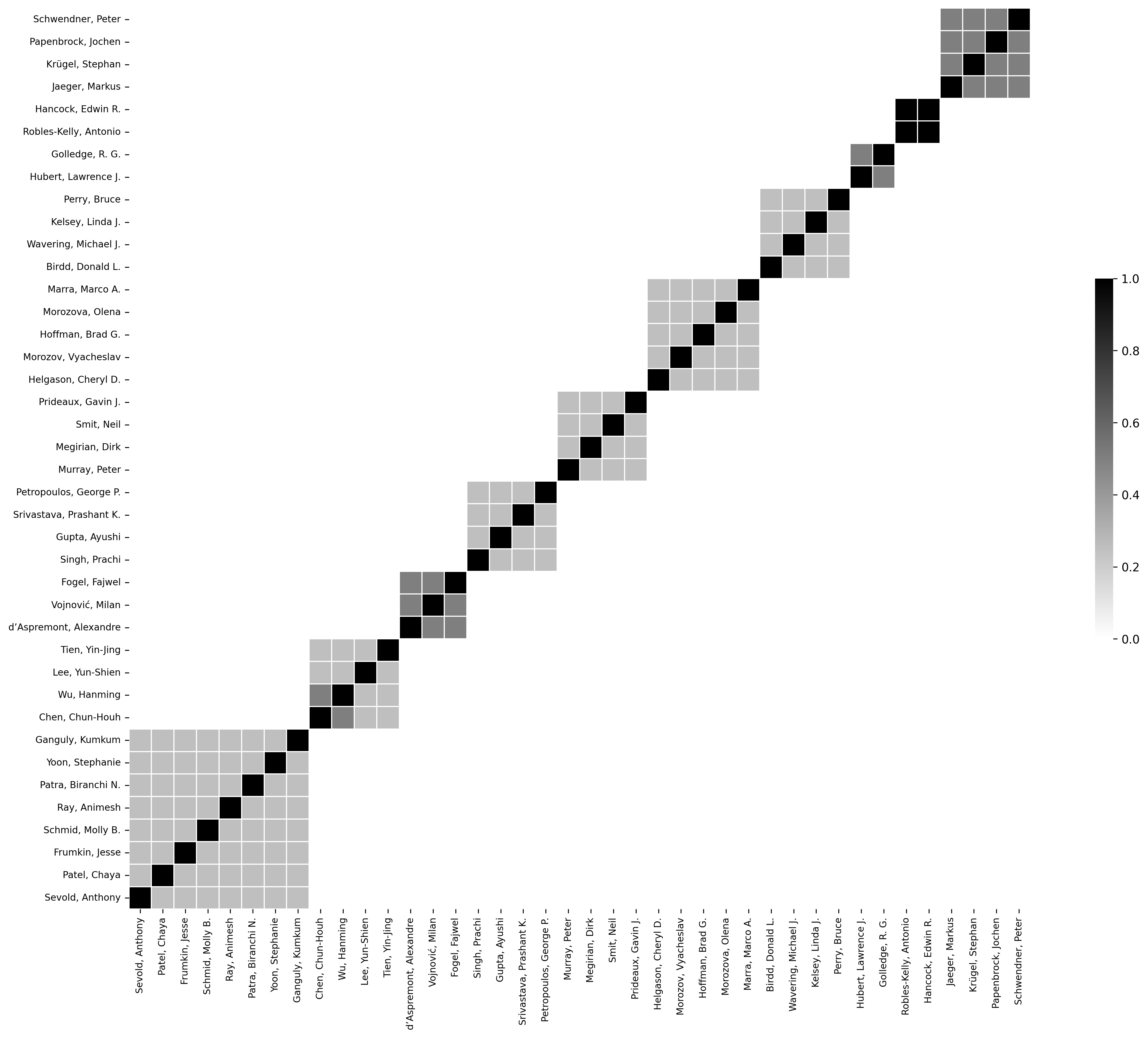}~\includegraphics[width=0.33\textwidth]{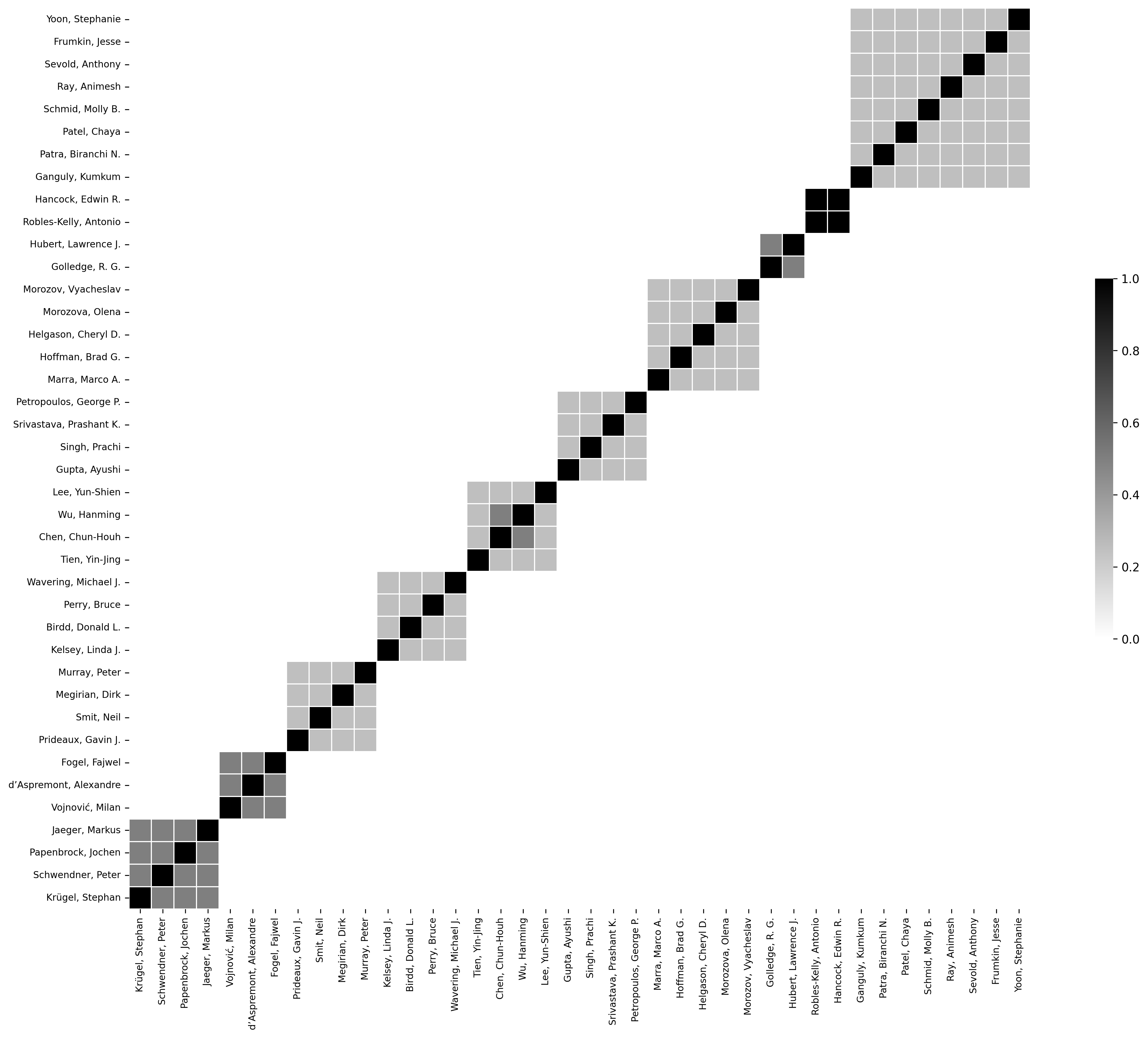}~{\includegraphics[width=0.33\textwidth]{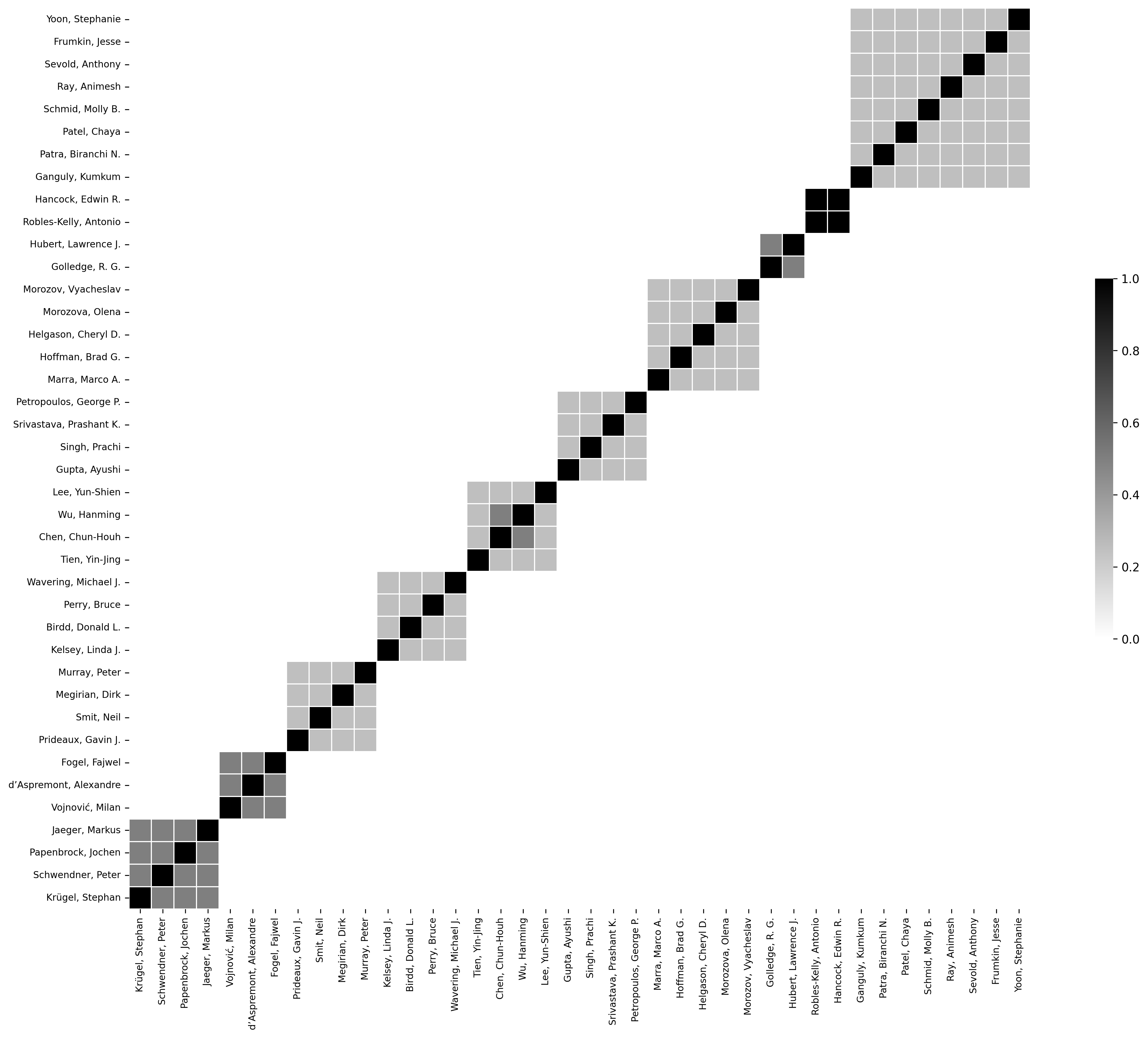}}
\caption{Representation of the optimal solutions for the $40$-seriation authorship instance with coordinated rows/columns under von Neumann stress measure (left), Moore stress measure (center) and ME measure (right).\label{fig:seriation40_joint}}
\end{figure}

Finally, for the 53-author dataset, Figures~\ref{fig:seriation53_joint} (coordinated) and \ref{fig:seriation53_sep} (uncoordinated) show consistent clustering results across all three criteria. Moore stress identifies 14 author clusters, while ME and von Neumann stress detect 15 clusters, two of which are weakly interconnected.

\begin{figure}[h]
\includegraphics[width=0.33\textwidth]{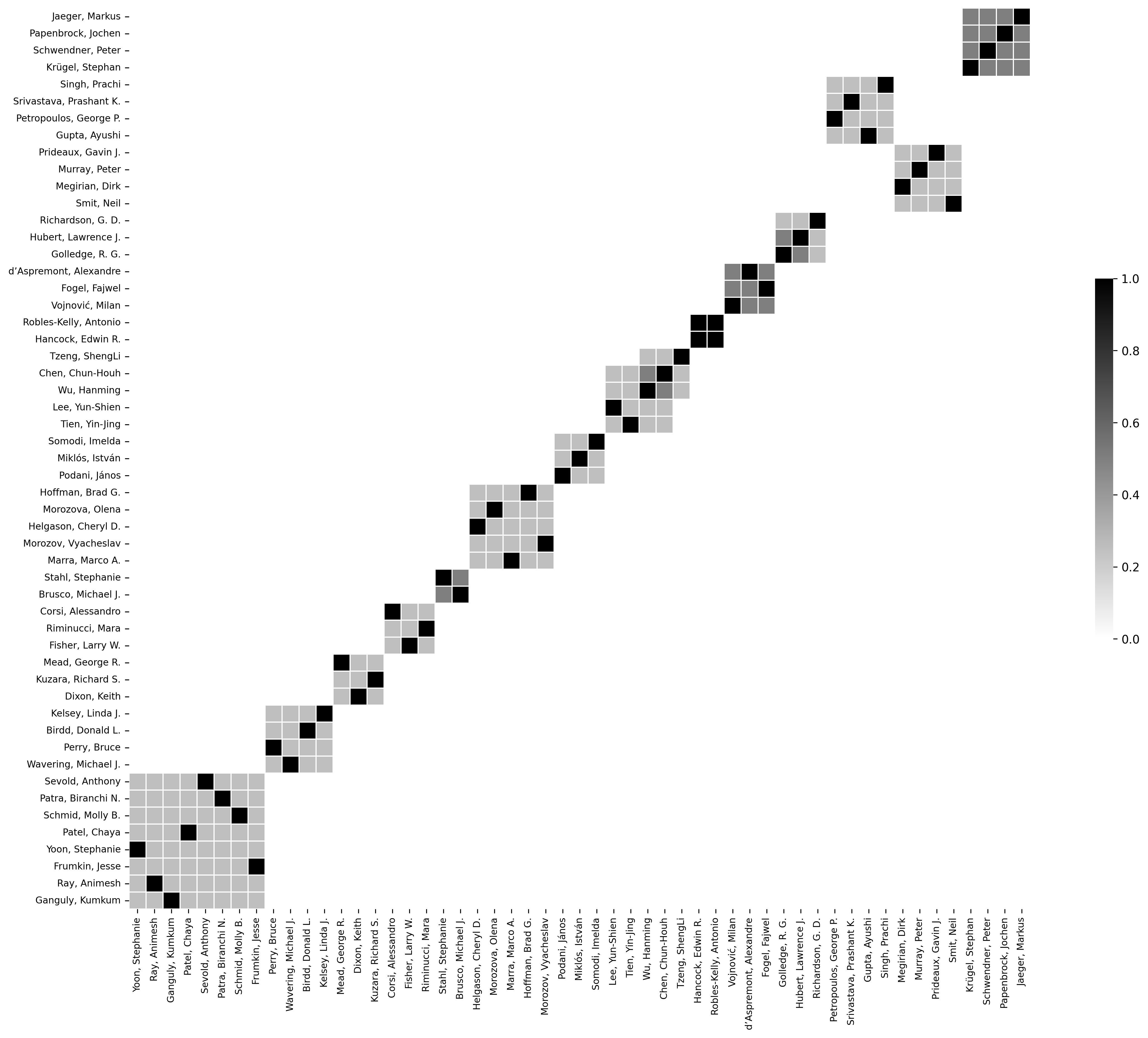}~\includegraphics[width=0.33\textwidth]{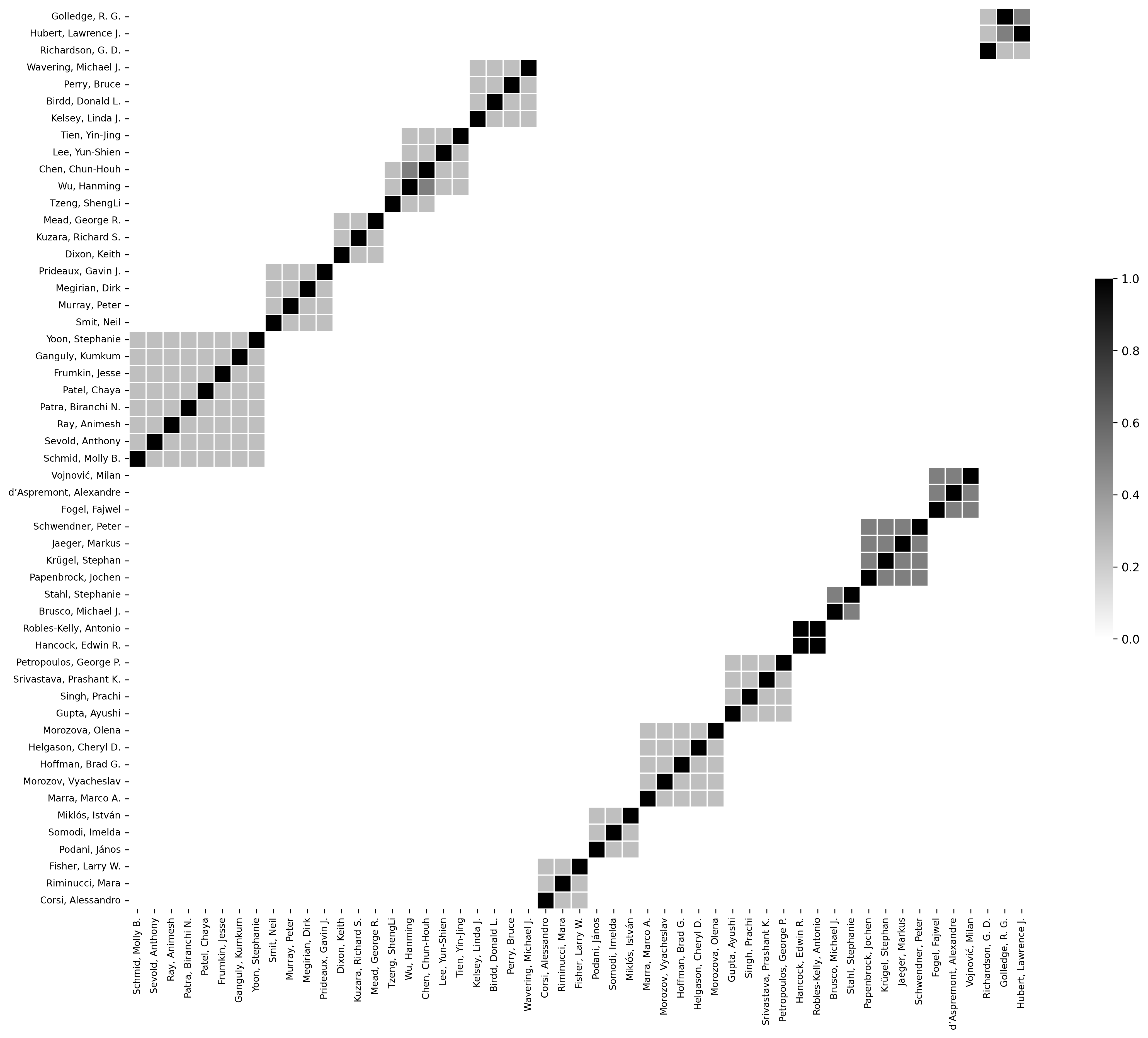}~{\includegraphics[width=0.33\textwidth]{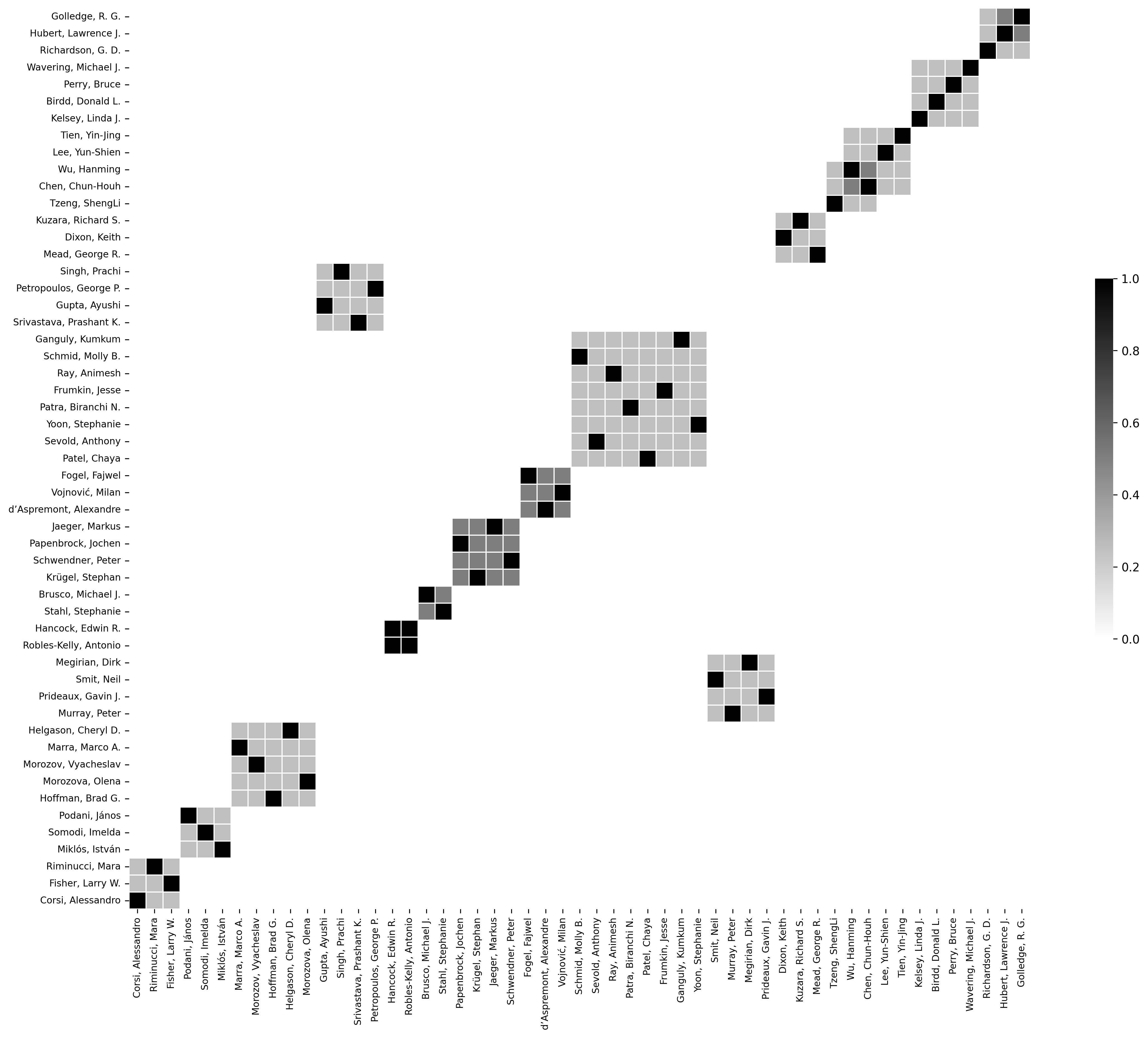}}
\caption{Representation of the optimal solutions for the $53$-seriation authorship instance with non coordinated rows/columns under von Neumann stress measure (left), Moore stress measure (center) and ME measure (right).\label{fig:seriation53_sep}}
\end{figure}

\begin{figure}[h]
\includegraphics[width=0.33\textwidth]{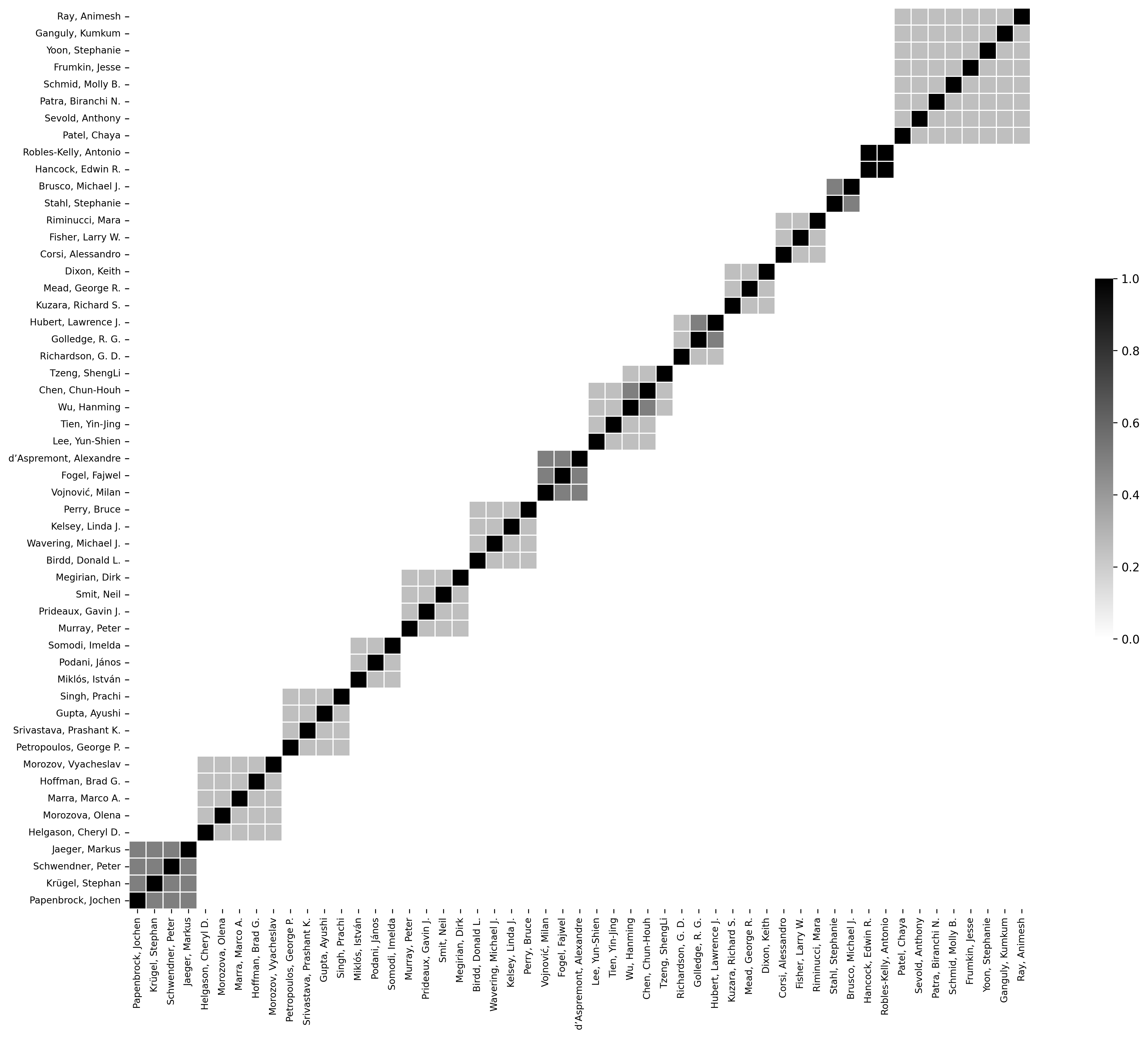}~\includegraphics[width=0.33\textwidth]{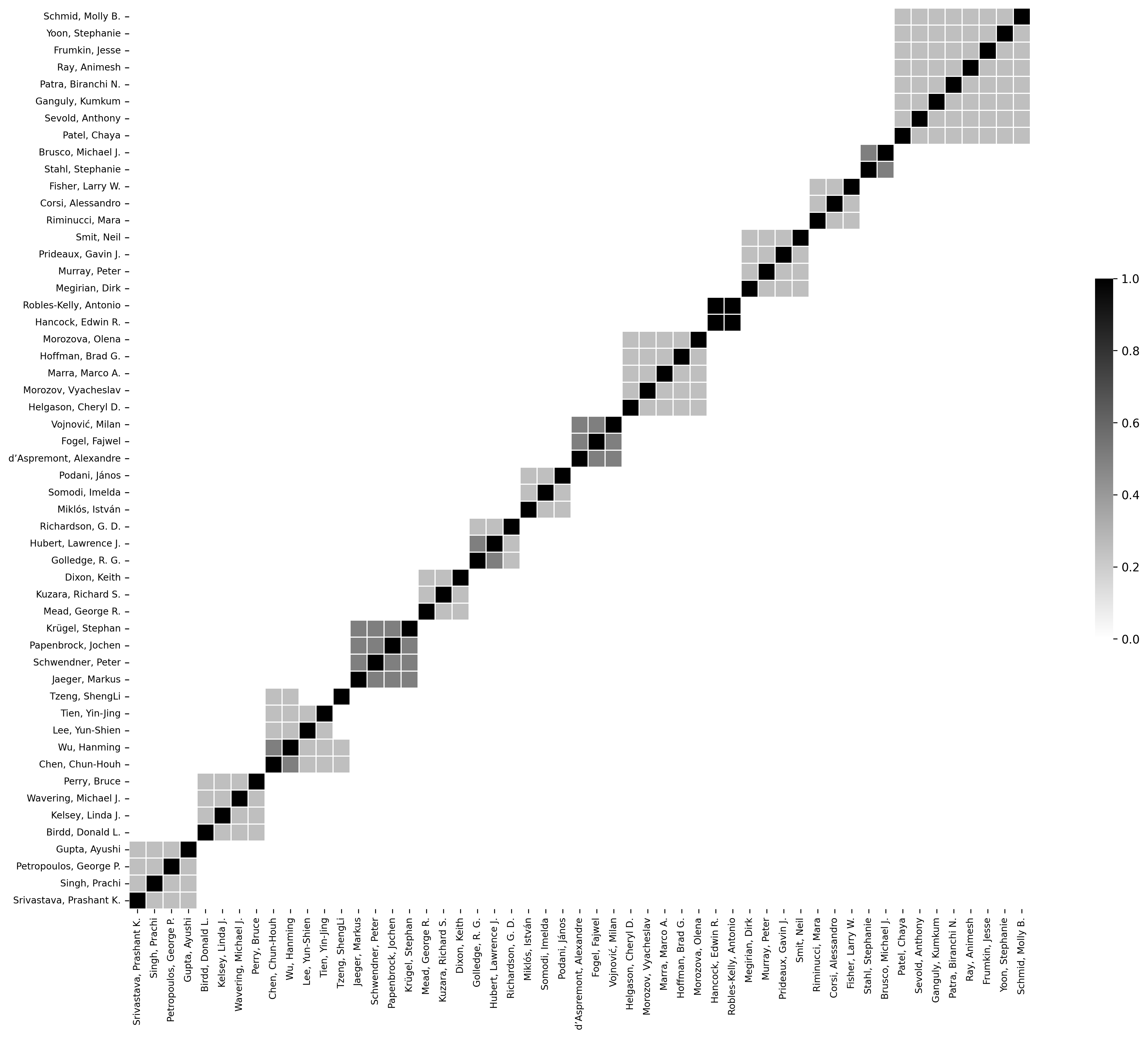}~{\includegraphics[width=0.33\textwidth]{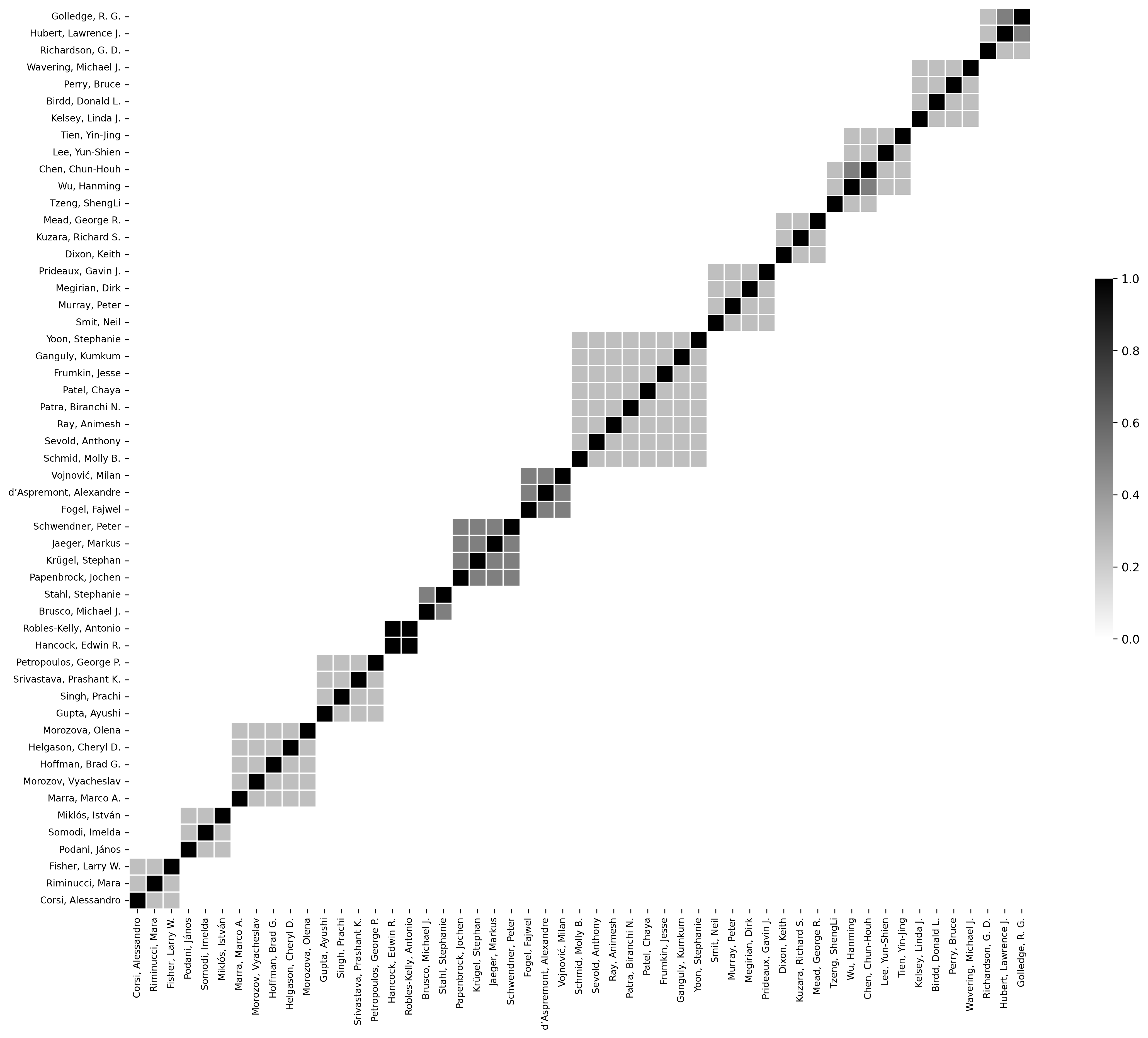}}
\caption{Representation of the optimal solutions for the $53$-seriation authorship instance with coordinated rows/columns under von Neumann stress measure (left), Moore stress measure (center) and ME measure (right).\label{fig:seriation53_joint}}
\end{figure}

Similar conclusions are obtained observing the transformed matrices for the $92$-authorship instance (Figures \ref{fig:seriation92_joint} and \ref{fig:seriation92_sep}).

In general, The von Neumann stress criterion seems to provide a fine-grained local organization with tight clusters and minimal overlap, highlighting immediate neighbor relationships. The Moore stress criterion maintains similar global structure but introduces slight misplacements and a few distant weak similarities, likely due to the broader neighborhood definition. In contrast, the Measure of Effectiveness criterion emphasizes global cohesion, producing an ordering that balances compact clustering with larger-scale structural regularity.

\begin{figure}[h]
\includegraphics[width=0.33\textwidth]{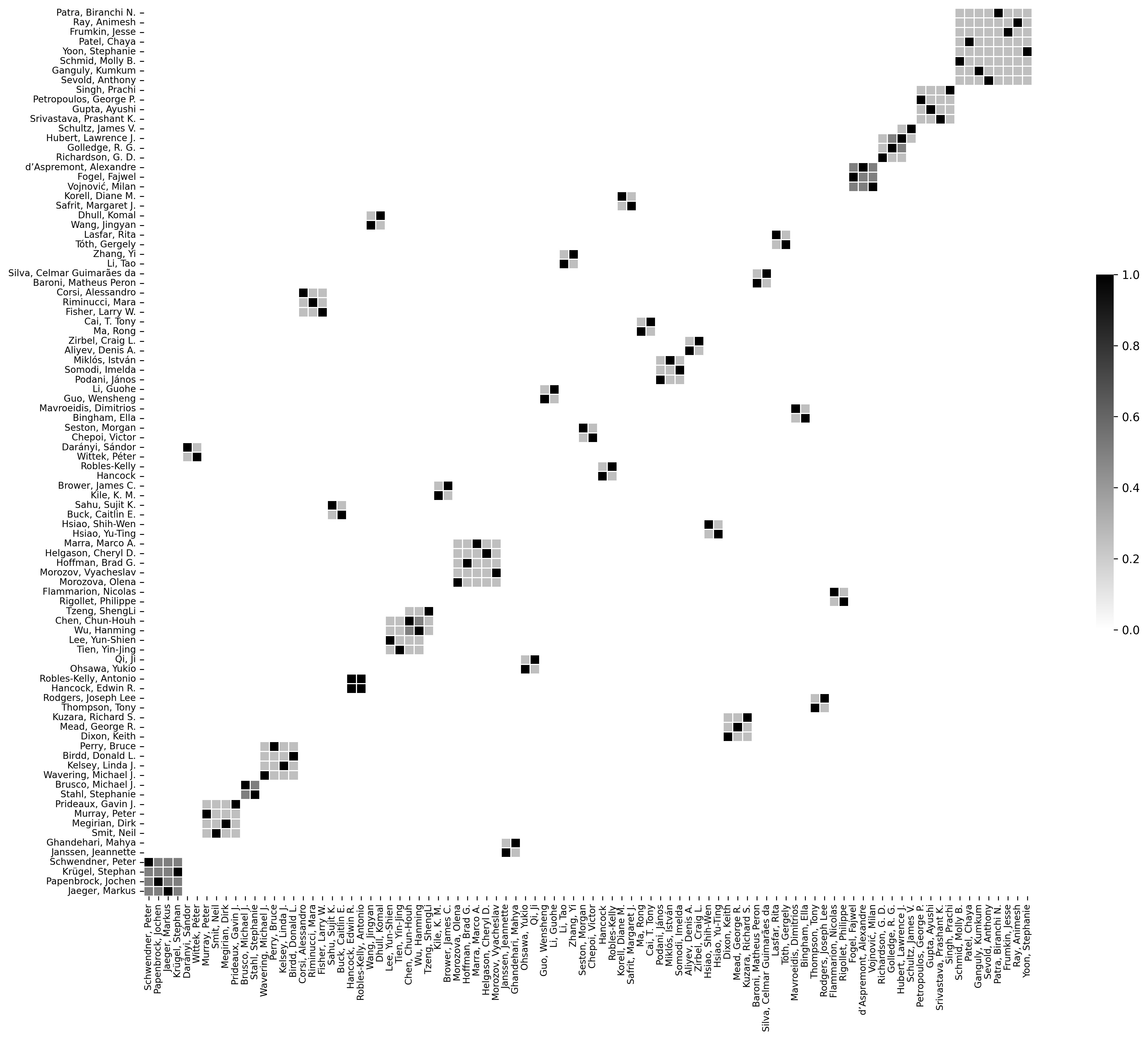}~\includegraphics[width=0.33\textwidth]{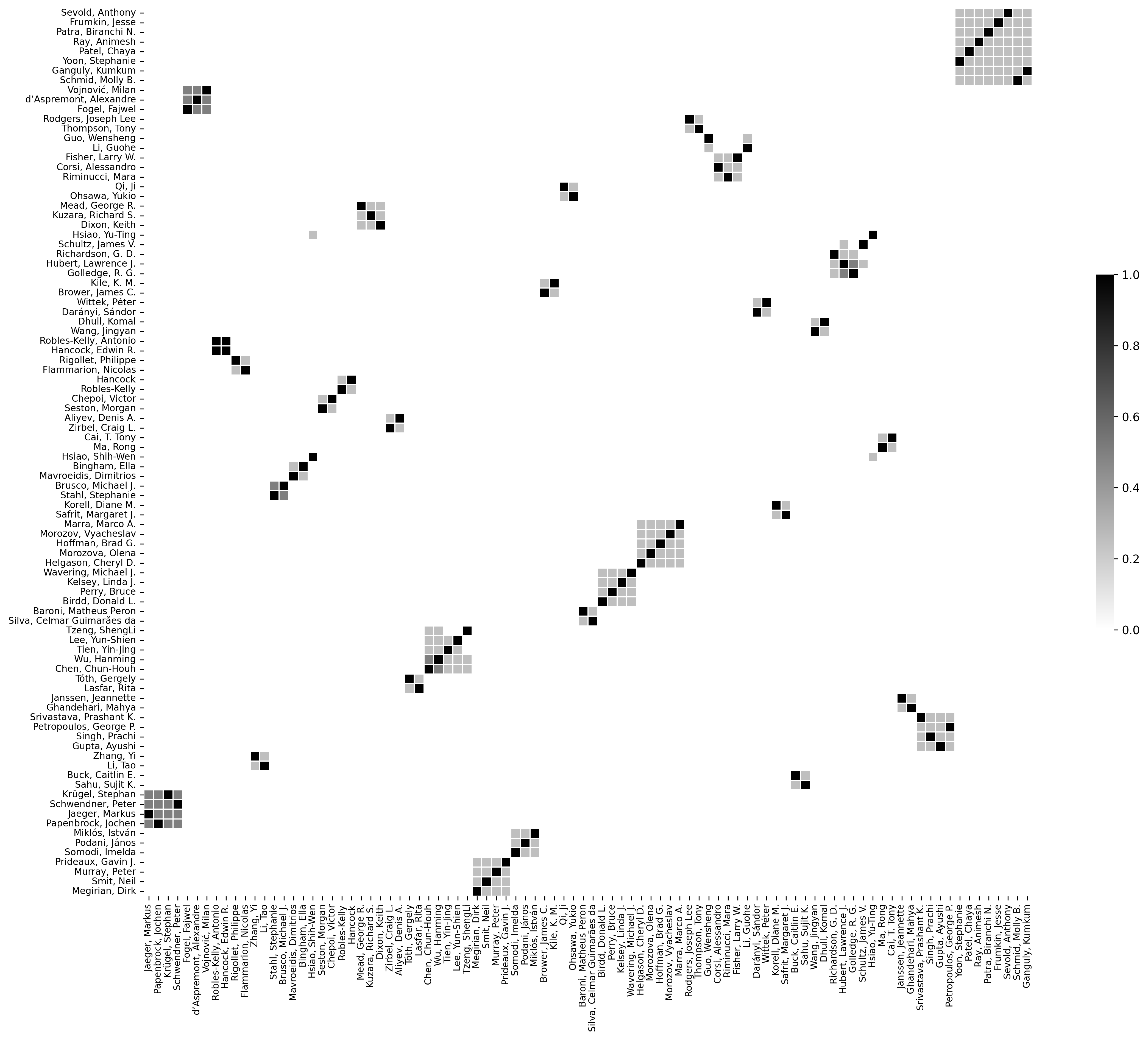}~{\includegraphics[width=0.33\textwidth]{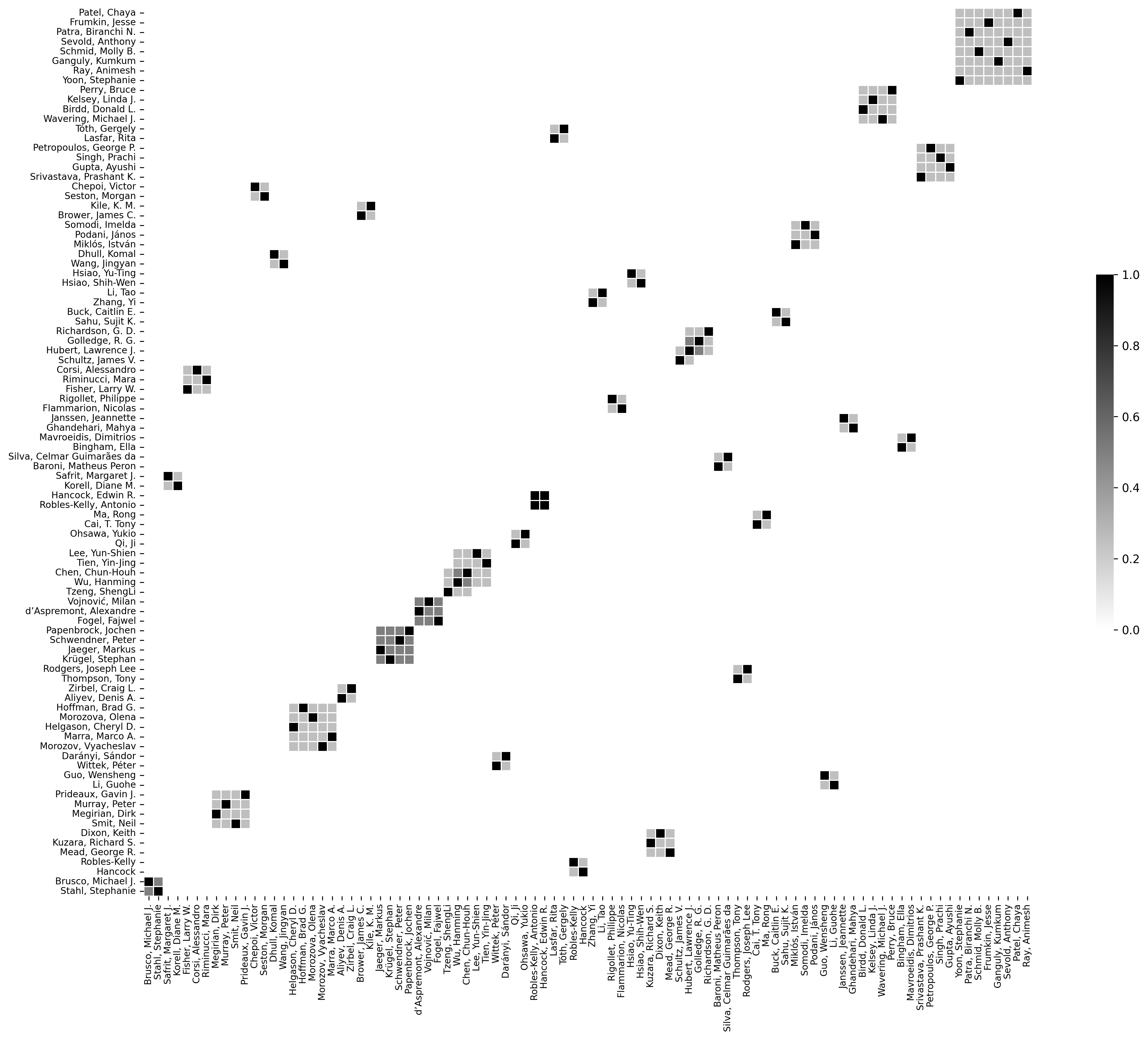}}
\caption{Representation of the optimal solutions for the $92$-seriation authorship instance with non coordinated rows/columns under von Neumann stress measure (left), Moore stress measure (center) and ME measure (right).\label{fig:seriation92_sep}}
\end{figure}

\begin{figure}[h]
\includegraphics[width=0.33\textwidth]{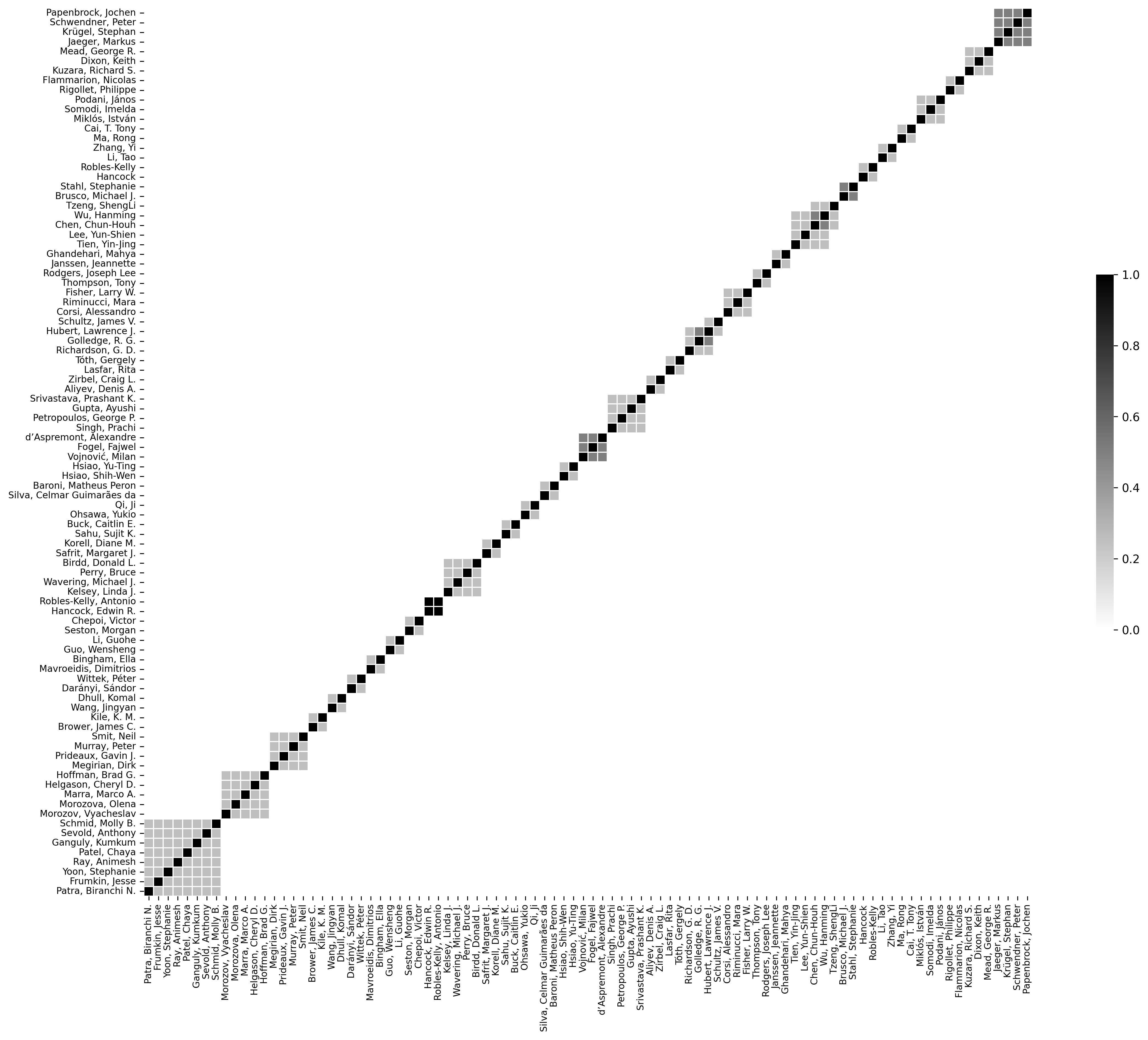}~\includegraphics[width=0.33\textwidth]{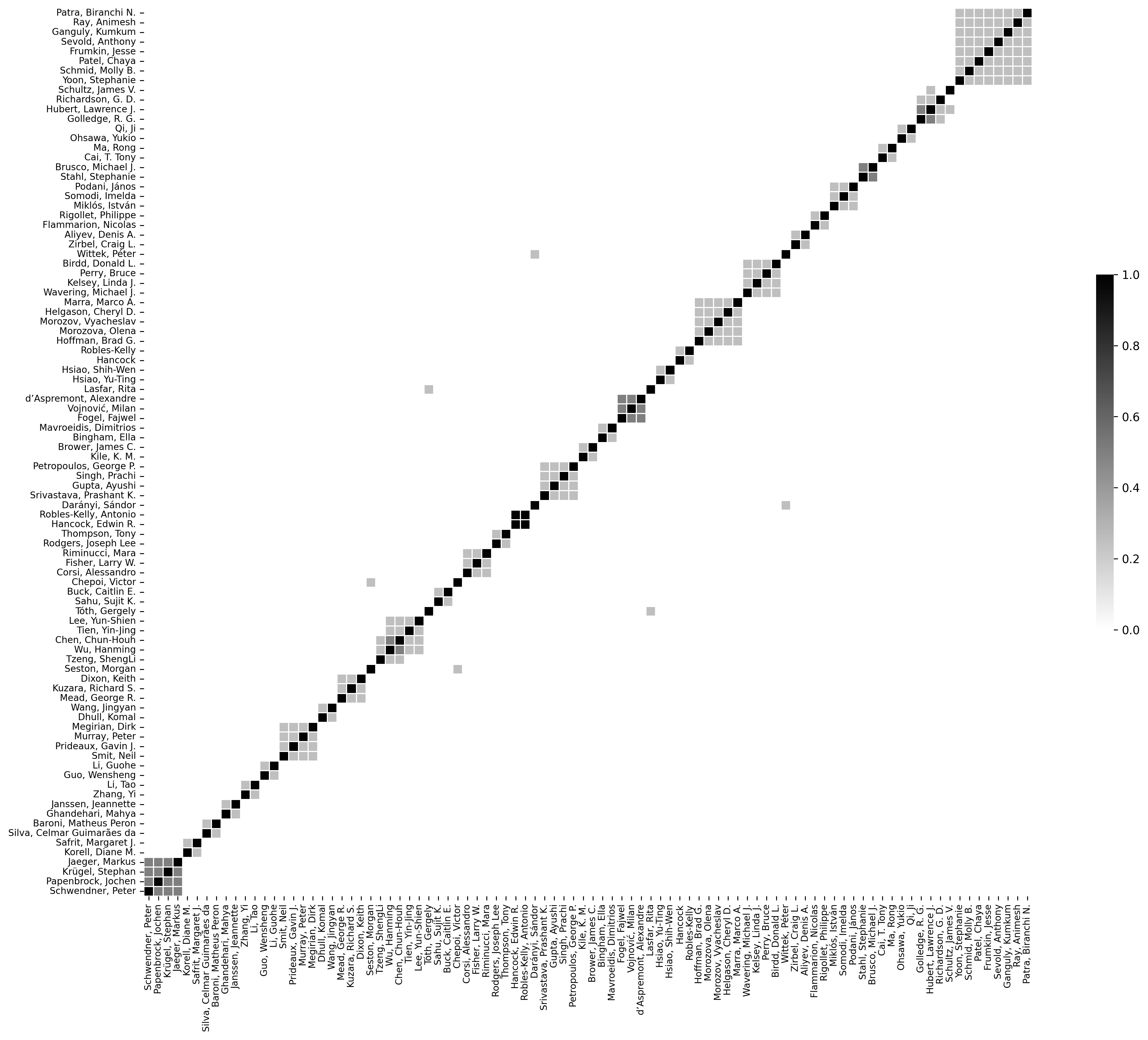}~{\includegraphics[width=0.33\textwidth]{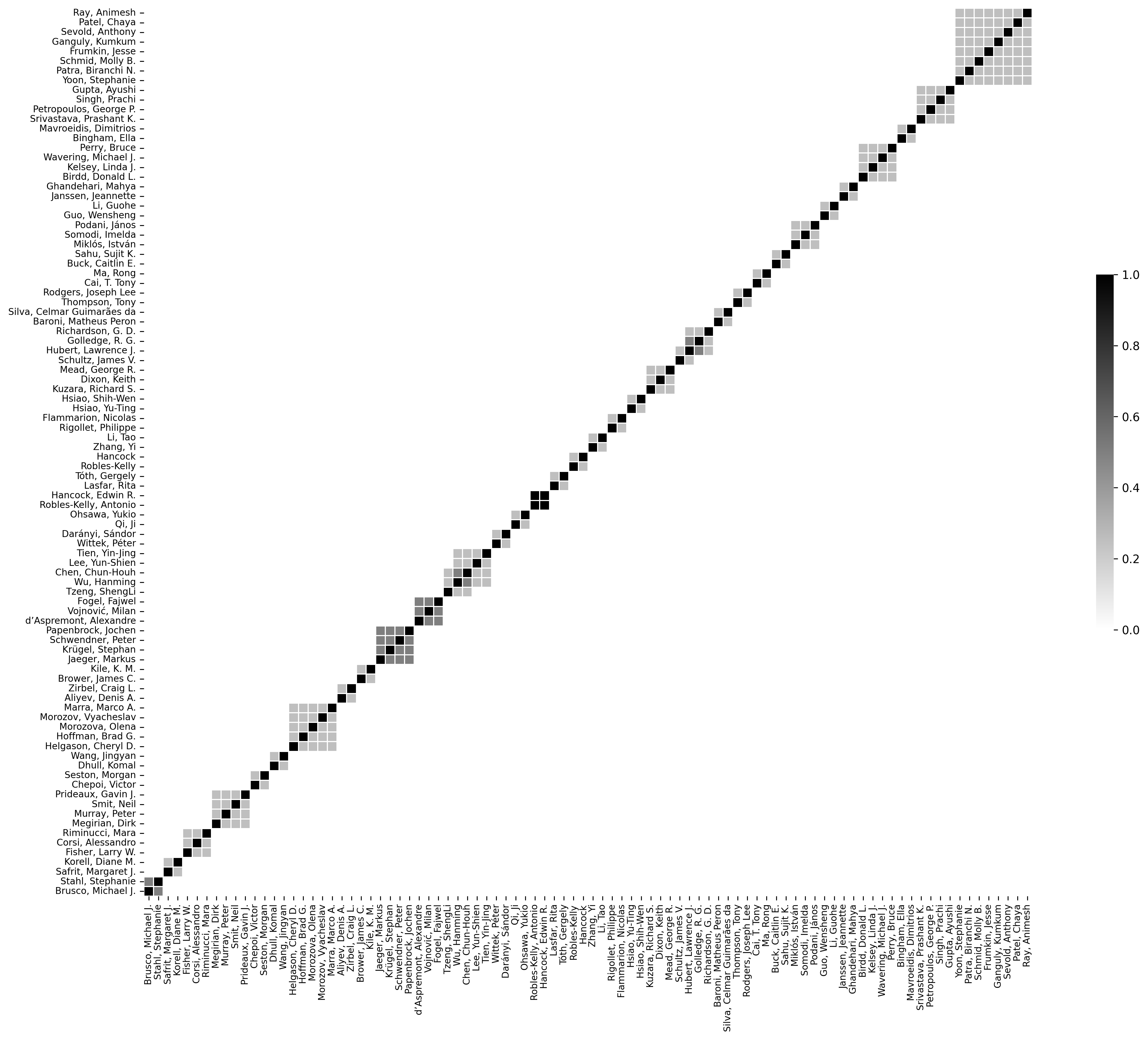}}
\caption{Representation of the optimal solutions for the $92$-seriation authorship instance with coordinated rows/columns under von Neumann stress measure (left), Moore stress measure (center) and ME measure (right).\label{fig:seriation92_joint}}
\end{figure}

In Table \ref{ultima} we summarize the results of the different measures obtained with our models. We highlight in bold face the best values for each of the datasets under each of the measures. These values are reached when the adequate measure is optimized (objective function), but in some cases some of the solutions coincide in value.

\begin{table}
\centering
{\small\begin{tabular}{l l | l l l}
& & \multicolumn{3}{c}{Measures}\\
dataset ($n$) &  Transformation & Von Neumann & Moore & ME \\
\hline
\textbf{14} & Original & 230 & 462 & 36 \\
 & N & \textbf{78} & 169 & \textbf{63} \\
 & Mo &\textbf{78} & \textbf{164} & \textbf{63} \\
 & ME & \textbf{78} & 167.5 & \textbf{63} \\
\textbf{40} & Original & 601 & 1045 & 6.25 \\
 & N & \textbf{299} & 577.5 & 53.87 \\
 & Mo & \textbf{299} & \textbf{545} & \textbf{54} \\
 & ME & 301.25 & 588.75 & \textbf{54} \\
\textbf{53} & Original & 761 & 1321 & 6.5 \\
 & N & \textbf{403} & 767.25 & \textbf{65.37} \\
 & Mo & 408.25 & \textbf{746.75} & 64.87 \\
 & ME & 411 & 803 & \textbf{65.37} \\
\textbf{92} & Original & 1151 & 1958 & 7.5 \\
 & N & \textbf{715} & 1339.75 & \textbf{85.37} \\
 & Mo & 727.75 & \textbf{1324} & 81.78 \\
 & ME & 719 & 1356.25 & \textbf{85.37}\\\hline
\end{tabular}}
\caption{Summary of measures obtained with the three methods for the co-authorship dataset.\label{ultima}} 
\end{table}

\section{Conclusions and Further Research}

In this paper, we have approached the matrix seriation problem from the perspective of mathematical optimization, focusing on two prominent criteria: neighborhood-based stress measures and the measure of effectiveness. We develop exact optimization models that compute globally optimal permutations of rows and/or columns; a key step in enhancing the interpretability and structure-revealing power of similarity and dissimilarity matrices, such as those visualized in heatmaps.

While much of the existing literature relies on heuristic methods, our work advances the field by introducing a rigorous optimization framework rooted in combinatorial and mixed-integer programming. We begin with a general binary programming model for seriation, which, despite its inherent complexity, captures the full expressiveness of the problem. For structured cases, we derive an equivalent minimum-cost Hamiltonian path formulation; this enables us to solve relevant instances more efficiently and precisely.

We validate our models through extensive computational experiments. These include synthetic datasets with controlled characteristics and real-world data from the R package \texttt{seriation}. Comparisons with heuristic alternatives show that our optimization-based solutions achieve clearer structural insights and more meaningful visualizations; this highlights that, in this setting, optimality translates directly into interpretability. Furthermore, we introduce a novel benchmark derived from a coauthorship network in the matrix seriation research community; this demonstrates how our methods can uncover meaningful clusters and relationships.

Looking ahead, several directions remain open for exploration. These include the integration of alternative objective functions and data-driven quality measures; the development of valid inequalities and cutting-plane techniques to strengthen our formulations; and the investigation of scalable relaxations or decomposition strategies for large-scale problems. These extensions will broaden the applicability of exact optimization tools for matrix seriation and solidify their role in solving complex problems in data science and visual analytics.

\section*{Acknowledgments}

Alfredo Mar\'in acknowledges that this article is part of the project TED2021-130875B-I00, 
funded by MCIN/AEI/10.13039/501100011033 and the European Union “NextGenerationEU”/PRTR
and has been also partially supported by research projects PID2022-137818OB-I00 (Ministerio
de Ciencia e Innovaci\'on, Spain), "Data Science and Resources Optimization in Comprehensive Waste Management (DataOpt-Waste)" 
TED2021-130875B-I00. V\'ictor Blanco and Justo Puerto acknowledge financial support by  grants PID2020-114594GB-C21 funded by MICIU/AEI /10.13039/ 501100011033, PCI2024-155024-2 - ``Optimization over Nonlinear Model Spaces: Where Discrete Meets Continuous Optimization'' funded by AEI and EU funds; FEDER+Junta de Andalucía projects C‐EXP‐139‐UGR23, and AT 21\_00032; VII PPIT-US (Ayudas Estancias Breves, Modalidad III.2A); and the IMAG-Maria de Maeztu grant CEX2020-001105-M /AEI /10.13039/501100011033 and IMUS-Maria de Maeztu grant CEX2024-001517-M - Apoyo a Unidades de Excelencia María de Maeztu. All the authors acknowledge financial support by grant RED2022-134149-T funded by MICIU/AEI/10.13039/501100011033 (Thematic Network on Location Science and Related Problems).

{\small


\appendix

\section*{List of authors in our case study}

\centering\begin{adjustbox}{max width=0.9\textwidth}
\begin{tabular}{l l}
 {\bf Author Name} & {\bf \# Co-authorships} \\
Aliyev, Denis A. & 1 \\
Baroni, Matheus Peron & 1 \\
Bingham, Ella & 1 \\
Birdd, Donald L. & 3 \\
Brower, James C. & 1 \\
Brusco, Michael J. & 2 \\
Buck, Caitlin E. & 1 \\
Cai, T. Tony & 1 \\
Chen, Chun--Houh & 5 \\
Chepoi, Victor & 1 \\
Corsi, Alessandro & 2 \\
D'Aspremont, Alexandre & 4 \\
Darányi, Sándor & 1 \\
Dhull, Komal & 1 \\
Dixon, Keith & 2 \\
Fisher, Larry W. & 2 \\
Flammarion, Nicolas & 1 \\
Fogel, Fajwel & 4 \\
Frumkin, Jesse & 7 \\
Ganguly, Kumkum & 7 \\
Ghandehari, Mahya & 1 \\
Golledge, R. G. & 3 \\
Guo, Wensheng & 1 \\
Gupta, Ayushi & 3 \\
Hancock, Edwin R. & 5 \\
Helgason, Cheryl D. & 4 \\
Hoffman, Brad G. & 4 \\
Hsiao, Shih-Wen & 1 \\
Hsiao, Yu-Ting & 1 \\
Hubert, Lawrence J. & 4 \\
Jaeger, Markus & 6 \\
Janssen, Jeannette & 1 \\
Kelsey, Linda J. & 3 \\
Kile, K. M. & 1 \\
Korell, Diane M. & 1 \\
Kr\"ugel, Stephan & 6 \\
Kuzara, Richard S. & 2 \\
Lasfar, Rita & 1 \\
Lee, Yun--Shien & 3 \\
Li, Guohe & 1 \\
Li, Tao & 1 \\
Ma, Rong & 1 \\
Marra, Marco A. & 4 \\
Mavroeidis, Dimitrios & 1 \\
Mead, George R. & 2 \\
&\\
\end{tabular}~\begin{tabular}{l l}
 {\bf Author Name} & {\bf \# Co-authorships} \\
 Megirian, Dirk & 3 \\
 Miklós, István & 2 \\
Morozov, Vyacheslav & 4 \\
Morozova, Olena & 4 \\
Murray, Peter & 3 \\
Ohsawa, Yukio & 1 \\
Papenbrock, Jochen & 6 \\
Patel, Chaya & 7 \\
Patra, Biranchi N. & 7 \\
Perry, Bruce & 3 \\
Petropoulos, George P. & 3 \\
Podani, János & 2 \\
Prideaux, Gavin J. & 3 \\
Qi, Ji & 1 \\
Ray, Animesh & 7 \\
Richardson, G. D. & 2 \\
Rigollet, Philippe & 1 \\
Riminucci, Mara & 2 \\
Robles--Kelly, Antonio & 5 \\
Rodgers, Joseph Lee & 1 \\
Safrit, Margaret J. & 1 \\
Sahu, Sujit K. & 1 \\
Schmid, Molly B. & 7 \\
Schultz, James V. & 1 \\
Schwendner, Peter & 6 \\
Seston, Morgan & 1 \\
Sevold, Anthony & 7 \\
Silva, Celmar Guimarães da & 1 \\
Singh, Prachi & 3 \\
Smit, Neil & 3 \\
Somodi, Imelda & 2 \\
Srivastava, Prashant K. & 3 \\
Stahl, Stephanie & 2 \\
Thompson, Tony & 1 \\
Tien, Yin-Jing & 3 \\
Tzeng, ShengLi & 2 \\
Tóth, Gergely & 1 \\
Vojnović, Milan & 4 \\
Wang, Jingyan & 1 \\
Wavering, Michael J. & 3 \\
Wittek, Péter & 1 \\
Wu, Hanming & 5 \\
Yoon, Stephanie & 7 \\
Zhang, Yi & 1 \\
Zirbel, Craig L. & 1 \\
&\\
\end{tabular}\end{adjustbox}

\section*{Clusters obtained with our models}

\begin{table}[ht]
\centering
\renewcommand{\arraystretch}{1.2}
\begin{tabular}{|c|p{11cm}|}
\hline
\textbf{Cluster} & \textbf{Authors} \\
\hline
Cluster A & Ganguly, Kumkum — Frumkin, Jesse — Patel, Chaya — Patra, Biranchi N. — Ray, Animesh — Schmid, Molly B. — Sevold, Anthony — Yoon, Stephanie \\
\hline
Cluster B & Jaeger, Markus — Kr\"ugel, Stephan — Papenbrock, Jochen — Schwendener, Peter \\
\hline
Cluster C & Chen, Chun-Houh — Wu, Hanming \\
\hline
\end{tabular}
\caption{Clusters detected with optimal seriation in the 14-authors dataset.}
\label{table:clusters14}
\end{table}
\begin{table}[ht]
\centering
\renewcommand{\arraystretch}{1.2}
\begin{tabular}{|c|p{12cm}|}
\hline
\textbf{Cluster} & \textbf{Authors} \\
\hline
A & Sevold, Anthony — Frumkin, Jesse — Patel, Chaya — Patra, Biranchi N. — Ray, Animesh — Schmid, Molly B. — Ganguly, Kumkum — Yoon, Stephanie \\
\hline
B & Chen, Chun-Houh — Wu, Hanming \\
\hline
C & d'Aspremont, Alexandre — Vojnovic, Milan — Fogel, Fajwel — Hubert, Lawrence J. — Golledge, R. G. \\
\hline
D & Petropoulos, George P. — Srivastava, Prashant K. — Gupta, Ayush — Singh, Prachi \\
\hline
E & Morozov, Vyacheslav — Hoffman, Brad G. — Marra, Marco A. — Morozova, Olena — Helgason, Cheryl D. \\
\hline
F & Schwendener, Peter — Papenbrock, Jochen — Kr\"ugel, Stephan — Jaeger, Markus \\
\hline
\multicolumn{2}{c}{Weakly Linked Clusters}\\\hline
\rowcolor{gray!15}
G & Hancock, Edwin R. — Robles-Kelly, Antonio \\
\hline
\rowcolor{gray!15}
H  & Perry, Bruce — Kelsey, Linda J. — Waverling, Michael J. \\
\hline
\rowcolor{gray!15}
I & Smit, Neil — Prideaux, Gavin J. — Murray, Peter \\
\hline
\rowcolor{gray!15}
J & Megriana, Dirk — Tien, Yin-jing — Lee, Yun-Shien \\
\hline
\end{tabular}
\caption{Clusters detected with optimal seriation in the 40-authors dataset.}
\label{table:clusters40}
\end{table}

\begin{table}[ht]
\centering
\renewcommand{\arraystretch}{1.2}
\begin{tabular}{|c|p{12.5cm}|}
\hline
\textbf{Cluster ID} & \textbf{Authors} \\
\hline
A & Frumkin, Jesse — Patel, Chaya — Patra, Biranchi N. — Ray, Animesh — Schmid, Molly B. — Ganguly, Kumkum — Yoon, Stephanie — Sevold, Anthony \\
\hline
B & Jaeger, Markus — Krügel, Stephan — Papenbrock, Jochen — Schwendener, Peter \\
\hline
C & Morozova, Olena — Hoffman, Brad G. — Helgason, Cheryl D. — Marra, Marco A. — Morozov, Vyacheslav \\
\hline
D & Podani, Janos — Miklós, István — Somodi, Imelda \\
\hline
E & Petropoulos, George P. — Srivastava, Prashant K. — Gupta, Ayush — Singh, Prachi \\
\hline
F & d’Aspremont, Alexandre — Fogel, Fajwel — Vojnović, Milan — Hubert, Lawrence J. — Golledge, R. G. \\
\hline
G & Wu, Hanming — Chen, Chun-Houh — Lee, Yun-Shien — Tien, Yin-jing — Tezeng, Sheng-Li \\
\hline
H & Hancock, Edwin R. — Robles-Kelly, Antonio — Stahl, Stephanie — Brusco, Michael J. \\
\hline
I & Riminucci, Mara — Fisher, Larry W. — Corsi, Alessandro \\
\hline
\multicolumn{2}{c}{Weakly linked} \\
\hline
\rowcolor{gray!15}J & Smit, Neil — Prideaux, Gavin J. — Murray, Peter — Megríana, Dirk \\
\hline
\rowcolor{gray!15}K & Perry, Bruce — Kelsey, Linda J. — Waverling, Michael J. \\
\hline
\rowcolor{gray!15}L & Bird, Donald L. — Richardson, G. D. — Dixon, Keith — Kuzara, Richard S. — Mead, George R. \\
\hline
\multicolumn{2}{c}{Very Weakly linked} \\
\hline
\rowcolor{gray!20}M & C + D (merged in \texttt{N}): Morozova, Olena — Hoffman, Brad G. — Helgason, Cheryl D. — Marra, Marco A. — Morozov, Vyacheslav — Podani, Janos — Miklós, István — Somodi, Imelda \\
\hline
\rowcolor{gray!20}N & Tezeng, Sheng-Li — Kuzara, Richard S. — Mead, George R. \\
\hline
\end{tabular}
\caption{Clusters detected with optimal seriation in the 53-author dataset.}
\label{table:clusters53_updated}
\end{table}

\begin{table}[ht]
\centering
\renewcommand{\arraystretch}{1.2}
\begin{tabular}{|c|p{12.5cm}|}
\hline
\textbf{Cluster ID} & \textbf{Authors} \\
\hline
A & Frumkin, Jesse — Patel, Chaya — Patra, Biranchi N. — Ray, Animesh — Schmid, Molly B. — Ganguly, Kumkum — Yoon, Stephanie — Sevold, Anthony \\
\hline
B & Jaeger, Markus — Krügel, Stephan — Papenbrock, Jochen — Schwendener, Peter \\
\hline
C & Morozov, Vyacheslav — Marra, Marco A. — Helgason, Cheryl D. — Hoffman, Brad G. — Morozova, Olena \\
\hline
D & Podani, Janos — Miklós, István — Somodi, Imelda \\
\hline
E & Petropoulos, George P. — Srivastava, Prashant K. — Gupta, Ayush — Singh, Prachi \\
\hline
\multicolumn{2}{c}{\textit{Moderately linked}} \\
\hline
F & Hancock, Edwin R. — Robles-Kelly, Antonio — Stahl, Stephanie — Brusco, Michael J. \\
\hline
G & Riminucci, Mara — Fisher, Larry W. — Corsi, Alessandro \\
\hline
\multicolumn{2}{c}{\textit{Weakly linked}} \\
\hline
\rowcolor{gray!15}H & Perry, Bruce — Kelsey, Linda J. — Waverling, Michael J. \\
\hline
\rowcolor{gray!15}I & Smit, Neil — Prideaux, Gavin J. — Murray, Peter — Megríana, Dirk \\
\hline
\rowcolor{gray!15}J & Dixon, Keith — Kuzara, Richard S. — Mead, George R. — Richardson, G. D. \\
\hline
\multicolumn{2}{c}{\textit{Very weakly linked}} \\
\hline
\rowcolor{gray!20}K & Lee, Yun-Shien — Chen, Chun-Houh — Tezeng, Sheng-Li — Schultz, James V. — Hsiao, Hsu-Ying — Hsiao, Han-Chi — Zhang, Yi — Silva, C.G. — Darwiche, Sandra — etc. \\
\hline
\end{tabular}
\caption{Clusters detected with optimal seriation in the 92-author dataset.}
\label{table:clusters92_updated}
\end{table}

\end{document}